\documentclass[oneside, a4paper, 12pt, bibliography=totoc]{scrartcl}
\usepackage[a4paper]{anysize}
\usepackage{marginnote}
\usepackage{times}
\usepackage[hyphens]{url}
\usepackage{hyperref}
\usepackage[round,colon]{natbib}
\usepackage[utf8]{inputenc}
\usepackage[fleqn]{amsmath}
\usepackage{amsfonts}
\usepackage{amssymb}
\usepackage{mleftright}
\usepackage{amsthm}
\usepackage{thmtools}
\usepackage{mathtools}
\usepackage{graphicx}
\usepackage{caption}
\usepackage{afterpage}
\usepackage{calc}
\declaretheorem[name=Lemma]{lemma}
\begin{document}
\setlength{\mathindent}{0mm}
\newlength{\symaeqmargin}
\setlength{\symaeqmargin}{2cm}
\newlength{\symaeqwidth}
\setlength{\symaeqwidth}{\textwidth - \symaeqmargin}
\allowdisplaybreaks

\newcommand{\mb}[1]{\mathbf{#1}}
\newcommand{\bs}[1]{\boldsymbol{#1}}
\newcommand{\opnl}[1]{\operatorname{#1}\nolimits}
\newcommand{\opl}[1]{\operatorname*{#1}}
\setlength\nulldelimiterspace{0pt}
\title{Exploring an Alternative Line-Search Method for Lagrange-Newton
Optimization}

\author{Ralf M\"oller\\Computer Engineering, Faculty of
Technology\\Bielefeld University\\\url{www.ti.uni-bielefeld.de}}

\date{version of \today}

\maketitle

\begin{abstract}

\noindent In the Lagrange-Newton method, where Newton's method is
applied to a Lagrangian function that includes equality constraints, all
stationary points are saddle points. It is therefore not possible to use
a line-search method based on the value of the objective function;
instead, the line search can operate on merit functions. In this report,
we explore an alternative line-search method which is applicable to this
case; it particulary addresses the damping of the step length in tight
valleys. We propose a line-search criterion based on the divergence of
the field of Newton step vectors. The visualization of the criterion for
two-dimensional test functions reveals a network of ravines with flat
bottom at the zero points of the criterion. The ravines are typically
connected to stationary points. To traverse this ravine network in order
to approach a stationary point, a zigzag strategy is devised. Numerical
experiments demonstrate that the novel line-search strategy succeeds
from most starting points in all test functions, but only exhibits the
desired damping of the step length in some situations. At the present
stage it is therefore difficult to appraise the utility of this
contribution.

\end{abstract}

\newpage

\tableofcontents

\newpage
\renewcommand{\textfraction}{0.1}
\renewcommand{\topfraction}{1.0}

\section{Introduction}\label{sec_intro}
%======================================

%
Objective functions with equality constraints can be optimized using
Newton's method. Shortening (``damping'') Newton steps by using an
additional line-search method can improve the convergence. However, when
applying Newton's method to a Lagrangian function (which combines the
objective function with the constraints using Lagrange multipliers), all
stationary points are saddle points. While the plain Newton method can
approach such a saddle point, a line-search based on the value of an
objective function is not possible in this case. Instead, merit
functions need to be used; these combine objective function and
constraint in a different way compared to the Lagrangian.

In this report, an alternative line-search method for this case is
explored. It uses the divergence of the field of Newton steps as
criterion. Experiments with two-dimensional test functions reveal that
the zero points of the criterion form a network of ``ravines'' which
include all stationary points (with more dimensions, we presumably get
``ravine manifolds''). As these ravines have a flat bottom, a special
line-search strategy is required to traverse the ravines. The zigzag
strategy suggested in this work first descends into the ravine and then
alternatingly follows the Newton vector up to a given height at the
ravine's wall (zig) and descends back to the ravine bottom (zag). Since
at least in some test functions the ravine runs along narrow valleys in
the objective function, the novel method should enable line search in
these valleys (but may possibly not address other situations where
damping of the step length is required).

Section \ref{sec_notation} provides a list of mathematical symbols used
for large parts of this report (other parts define their own notation).
In section \ref{sec_newton_ls}, we start with a problem analysis from
the perspective of line-search in the plain Newton method. The
divergence-based criterion is derived in section \ref{sec_div_crit}.
Section \ref{sec_line_search} introduces the line-search strategy
required to traverse the ravine network produced from the zero points of
the criterion. The mathematical background of the Lagrange-Newton
method, i.e. the application of Newton's method to equality-constrained
problems, is recapitulated in section \ref{sec_lagrange_newton}. In
addition, we provide a proof that, in this case, all stationary points
are saddle points. We explore the novel method in numerical experiments
on several two-dimensional test functions (section \ref{sec_exp_twod})
and on one higher-dimensional test function (section
\ref{sec_exp_highd}). The results are discussed in section
\ref{sec_discussion} and conclusions are presented in section
\ref{sec_conclusions}. The software accompanying this report is referred
to in section \ref{sec_software}.

\section{Notation}\label{sec_notation}
%=====================================

%
\textbf{Dimensions:}
\begin{description}
\setlength{\itemsep}{0pt}
\item[] $n$: dimensionality of optimization problem, $1 \leq n$
\end{description}
\textbf{Matrices:}
\begin{description}
\setlength{\itemsep}{0pt}
\item[] $\mb{x}$: state vector, arbitrary, $n \times 1$
\item[] $f$: objective function, scalar, $1 \times 1$
\item[] $\mb{H}$: Hessian, symmetric, $n \times n$
\item[] $\bs{\gamma}$: gradient, arbitrary, $n \times 1$
\item[] $\bs{\nu}$: Newton step, arbitrary, $n \times 1$
\item[] $\mb{N}$: Jacobian of Newton step, square, $n \times n$
\item[] $\mb{C}$: shape matrix of quadratic function, symmetric, $n \times n$
\item[] $\mb{c}$: center of quadratic function, arbitrary, $n \times 1$
\item[] $\mb{s}$: small step, arbitrary, $n \times 1$
\item[] $\alpha$: damping factor for line search, scalar, $1 \times 1$
\item[] $\mb{v}$: generic vector field, arbitrary, $n \times 1$
\item[] $\tau$: divergence of Newton step, scalar, $1 \times 1$
\item[] $\check{\tau}$: divergence criterion, scalar, $1 \times 1$
\item[] $\tilde{\mb{q}}$: unnormalized pullback direction (approximated gradient), arbitrary, $n \times 1$
\item[] $\mb{q}$: normalized pullback direction (approximated gradient), arbitrary, $n \times 1$
\item[] $\tilde{\mb{p}}$: unnormalized pullback direction (singularity curves), arbitrary, $n \times 1$
\item[] $\mb{p}$: normalized pullback direction (singularity curves), arbitrary, $n \times 1$
\end{description}

\section{Newton's Method and Line Search}\label{sec_newton_ls}
%=============================================================

%
Our goal is to find a criterion for line search in Newton's method which
is also applicable to approach saddle points, as required for the
Lagrange-Newton method. We start our problem analysis from the generic
Newton method. The Lagrange-Newton method (see section
\ref{sec_lagrange_newton}) can for now be considered as a special case
of the Newton method where the state vector is extended by the Lagrange
multipliers.

In Newton's method \citep[see e.g.][ch.~9]{nn_Chong08}, a quadratic
approximation of the given function $f\mleft( \mb{x} \mright)$ is
determined from a truncated Taylor series. For this approximation, a
``Newton step'' $\bs{\nu}\mleft( \mb{x} \mright)$ towards the stationary
point of the approximation is computed. If $f\mleft( \mb{x} \mright)$ is
not purely quadratic, this procedure needs to be repeated. The Newton
step is obtained from
\begin{align}
\label{eq_newton_nu}
\bs{\nu}\mleft( \mb{x} \mright)
&=
- \mleft( \mb{H}\mleft( \mb{x} \mright) \mright)^{-1} \bs{\gamma}\mleft(
\mb{x} \mright)
\end{align}
where the gradient of $f\mleft( \mb{x} \mright)$ (a column vector) is
given by
\begin{align}
\bs{\gamma}\mleft( \mb{x} \mright)
&=
\mleft( \frac{\partial}{\partial \mb{x}}\,f\mleft( \mb{x} \mright)
\mright)^T
\end{align}
and the Hessian of $f\mleft( \mb{x} \mright)$ is given by
\begin{align}
\mb{H}\mleft( \mb{x} \mright) & = \frac{\partial}{\partial \mb{x}}\,\bs{\gamma}\mleft( \mb{x} \mright) = \frac{\partial}{\partial \mb{x}}\,\mleft( \frac{\partial}{\partial
\mb{x}}\,f\mleft( \mb{x} \mright) \mright)^T.
\end{align}
The Newton method converges towards any stationary point, regardless of
whether it is a minimum, a maximum, or a saddle point. This can be
explained informally as follows: At a stationary point, the gradient is
zero, and in the vicinity, third- and higher-order terms are small in
comparison to the quadratic term, so we can approximate $f\mleft( \mb{x}
\mright)$ by a quadratic function
\begin{align}
\label{eq_f_quad_approx}
f\mleft( \mb{x} \mright)
&\approx
\frac{1}{2}\, \mleft( \mb{x} - \mb{c} \mright)^T \mb{C} \mleft( \mb{x} -
\mb{c} \mright)
\end{align}
where $\mb{C}$ is a symmetric matrix describing the shape of the
quadratic function and vector $\mb{c}$ is the center of the quadratic
function. Note that the eigenvalues of $\mb{C}$ determine whether we
have a saddle point (mixed signs), a minimum (all negative), or a
maximum (all positive). With $\bs{\gamma}\mleft( \mb{x} \mright) =
\mb{C} \mleft( \mb{x} - \mb{c} \mright) $ and $\mb{H}\mleft( \mb{x}
\mright) = \mb{C}$ we get the Newton step
\begin{align}
\label{eq_nu_quad_approx}
\bs{\nu}\mleft( \mb{x} \mright)
&=
- \mleft( \mb{x} - \mb{c} \mright)
\end{align}
which is independent of $\mb{C}$ and therefore independent of the type
of stationary point. If we deviate from the center $\mb{c}$ by a step
$\mb{s}$, and thus have $\mb{x} = \mb{c} + \mb{s}$, we obtain
\begin{align}
\bs{\nu}\mleft( \mb{c} + \mb{s} \mright) & = - \mleft( \mleft[ \mb{c} + \mb{s} \mright] - \mb{c} \mright) = - \mb{s}
\end{align}
so the deviation $\mb{s}$ is always eliminated by the Newton step,
leading to $\mb{x} = \mb{c}$.\footnote{Note that we work on the local
quadratic approximation of $f\mleft( \mb{x} \mright)$, so we need to
assume that the deviation $\mb{s}$ is small in order to make a statement
on the original $f\mleft( \mb{x} \mright)$.}

Note that while the convergence of Newton's method to all types of
stationary points is a problem for unconstrained optimization where only
either minima or maxima are of interest, the convergence to saddle
points is a prerequisite for optimization under equality constraints in
the Lagrange-Newton method --- a gradient descent in the Lagrangian
\eqref{eq_lagrangian} would fail.

Thus, at least in the vicinity of the stationary point, the Newton step
leads us towards the stationary point. However, Newton's method has two
problems which can occur at larger distances from the stationary point.
First, the quadratic approximation underlying Newton's method may be
``wrong'' (in the sense that the quadratic approximation at the current
location is at odds with the quadratic approximation at the stationary
point) and so the Newton step may actually increase the distance to the
stationary point. This problem is not addressed in this work.

Second, even if the quadratic terms are appropriate, higher-order
components in the function $f\mleft( \mb{x} \mright)$ get more influence
at larger distance from the stationary point. The Newton step --- which
is perfect for a quadratic function --- may therefore be too
long,\footnote{Apparently it is assumed that the Newton step may be {\em
too long} and needs to be ``damped'', i.e. shortened, but may not be
{\em too short}: e.g. \citet[eq.~13.3.2]{nn_Jarre19} modify the damping
factor in the range $[0,1]$ in SQP methods;
\citet[algorithm~3.1]{nn_Nocedal06} suggest an initial factor of $1$ for
backtracking in Newton methods (backtracking reduces the factor
iteratively). It is not clear to me whether there is a mathematical or
an empirical argument underlying this choice.} Such an ``overstepping''
may also occur close to singularities of the Hessian (as the values of
the inverse Hessian used in Newton's method get large) and in narrow,
winding valleys of the objective function.\footnote{Some websites also
list flat regions of the objective function as problematic, but I don't
see a direct relationship between the slope and the Hessian.} If
overstepping is not prevented, additional Newton steps may be required,
impairing the performance; a one-dimensional example is given by
\citet[sec.~4.3.1]{nn_Alt11}. In some cases, the method may even diverge
after an such an overstepping.

This is the reason why {\em line-search methods} are used to find a
reduced step $\alpha \bs{\nu}\mleft( \mb{x} \mright) $ with $\alpha \in
(0,1]$. For a minimum search, this is relatively straight-forward: The
most trivial method would be to search through discrete values of
$\alpha$ and look for the largest reduction of $f\mleft( \mb{x} + \alpha
\bs{\nu}\mleft( \mb{x} \mright) \mright)$. However, this approach fails
for the Lagrange-Newton method where we need to approach a saddle point.
At a saddle point, it is not known whether a decrease or an increase in
the function value indicates progress towards the stationary point.

For the Lagrange-Newton case where a function is minimized under an
equality constraint, {\em merit functions} may be used for the line
search. According to \citet[sec.~15.2]{nn_Andrei22}, these merit
functions are the sum of two terms: one related to the objective
function, the other expressing the fulfillment of the constraints. The
merit function needs to balance the two parts, introducing weighting
parameters which may be difficult to chose. It is not clear to me from
the study of optimization textbooks whether merit functions like
augmented Lagrangians \citep[][sec.~15.4]{nn_Nocedal06} are routinely
used for line search in Lagrange-Newton methods (or in the related SQP
methods); it is however at least mentioned by
\citet[sec.~13.3.1]{nn_Jarre19} and \citet[sec.~11.4.4.1]{nn_Corriou21}.
Moreover, there may be other cases besides the Lagrange-Newton method
where we want to approach a saddle point, but where there is no clear
distinction between two criteria that need to be satisfied
simultaneously.\footnote{An example is an ``information criterion'' for
principal component analysis which we used to derive online neural
networks \citep[][]{own_Moeller04a}.} This may prevent the
straight-forward design of a merit function.

In this work we therefore explore whether it is possible to find a
criterion which can be used for the line search in Newton's method, even
if we approach a saddle point. Since the function values $f\mleft(
\mb{x} \mright)$ are not indicative of progress in this case, we need to
resort to derivatives of the objective function. The method developed
below may actually have a more restricted applicability compared to a
classical line-search method, as it probably only works in narrow,
winding valleys of the objective function, but possibly not in other
situations where the Newton vector gets too long, e.g. close to the
singularity of the Hessian.\footnote{However, as will be shown below,
there often is a damping in one of the first steps which is cut short at
the nearest ravine.}

\section{Divergence-based Criterion}\label{sec_div_crit}
%=======================================================

%
The criterion explored here is based on the ``divergence'', a measure
for detecting sources and sinks in vector fields known from physics and
electrical engineering. If $\mb{v}\mleft( \mb{x} \mright)$ is a vector
field over a spatial coordinate $\mb{x}$, the divergence is defined as
\begin{align}
\opnl{div} \mleft\{ \mb{v}\mleft( \mb{x} \mright) \mright\}
&=
\sum\limits_{i = 1}^{n}\frac{\partial}{\partial x_{i}}\,\mleft(
\mb{v}\mleft( \mb{x} \mright) \mright)_{i}.
\end{align}
We see that the divergence adds up scalar derivatives of vector field
component $i$ with respect to location component $i$.\footnote{Due to
this close relation to the components of the vector field, I was not
sure whether the divergence is actually invariant to rotations of the
vector field. However, invariance to rotations is given; a proof is
provided in section~\ref{sec_inv_div_rot}.}

The divergence can alternatively be defined via the trace\footnote{ This
form is useful for derivations and proofs, but should of course not be
used in a software implementation due to the computational overhead for
computing the unused off-diagonal elements.} as
\begin{align}
\opnl{div} \mleft\{ \mb{v}\mleft( \mb{x} \mright) \mright\}
&=
\opnl{tr} \mleft\{ \frac{\partial}{\partial \mb{x}}\,\mb{v}\mleft(
\mb{x} \mright) \mright\}.
\end{align}
The divergence criterion we suggest for the Newton method is
\begin{align}
\label{eq_divcrit1}
\tau\mleft( \mb{x} \mright) & = - {\,\frac{1}{n}\,} \opnl{div} \mleft\{ \bs{\nu}\mleft( \mb{x} \mright)
\mright\} = - {\,\frac{1}{n}\,} \opnl{tr} \mleft\{ \frac{\partial}{\partial
\mb{x}}\,\bs{\nu}\mleft( \mb{x} \mright) \mright\}\\
%====================
\label{eq_divcrit2}
& = {\,\frac{1}{n}\,} \opnl{div} \mleft\{ \mleft( \mb{H}\mleft( \mb{x}
\mright) \mright)^{-1} \bs{\gamma}\mleft( \mb{x} \mright) \mright\}.
\end{align}
The normalization of the divergence by $n$ makes the value independent
of the number of dimensions and guarantees that the value of
$\tau\mleft( \mb{x} \mright)$ is exactly $1$ at stationary points. The
minus sign is required as the value at stationary points would be $-1$
without. For the quadratic approximation \eqref{eq_f_quad_approx} we
insert the Newton step from \eqref{eq_nu_quad_approx} and obtain
\begin{align}
\tau\mleft( \mb{x} \mright) & = - {\,\frac{1}{n}\,} \opnl{tr} \mleft\{ \frac{\partial}{\partial
\mb{x}}\,\bs{\nu}\mleft( \mb{x} \mright) \mright\} = - {\,\frac{1}{n}\,} \opnl{tr} \mleft\{ \frac{\partial}{\partial
\mb{x}}\,\mleft( - \mleft[ \mb{x} - \mb{c} \mright] \mright) \mright\} = {\,\frac{1}{n}\,} \opnl{tr} \mleft\{ \frac{\partial}{\partial
\mb{x}}\,\mb{x} \mright\} = {\,\frac{1}{n}\,} \opnl{tr} \mleft\{ \mb{I}_{n} \mright\} = 1
\end{align}
which proves the desired property.

The divergence is applied to the Newton step rather than the gradient.
Due to the multiplication of the gradient by the inverse Hessian, the
Newton steps are directed ``inwards'', i.e. all vectors point
approximately towards a nearby stationary point, regardless of its type
(i.e. minimum, maximum, or saddle). Put differently, sinks (minima),
sources (maxima), and saddle points in the gradient field are turned
into sinks in the field of Newton steps. In addition, the shape of the
resulting sink-like vector field is distorted such that the vectors
point almost in a straight line towards the stationary point (which
would not be the case for the gradient, e.g. in elongated depressions,
even in minima where the vectors are already directed inwards).

At any stationary point, the divergence criterion is $1$, as shown
above. My original idea behind using the divergence was that this
criterion would decrease with increasing distance from the stationary
point. However, this is not the case as was revealed by studying the
criterion in two-dimensional test functions (see below). Instead, when I
plotted locations where the divergence is $1$, I observed curves
typically running through stationary points.\footnote{However, as will
be shown below, there can be isolated stationary points not lying on
such a curve. Moreover, ``spurious'' curves may appear which do not
include a stationary point.} Particularly in the Rosenbrock function
which exhibits a narrow, winding valley (where also the stationary point
is located), the curves run along the bottom of this valley. The curves
therefore offer themselves as the possible basis of a different class of
line-search methods which follow the Newton step as long at is stays
close to the curves.

For functions with higher dimension $n$, I assume that we don't have
lines with unit divergence criterion but manifolds of dimension $n - 1$.
It is not sufficient to have a {\em narrow} valley in the objective
function, but the valley also needs to be {\em winding} for manifolds
with unit divergence to occur which intersects with its
bottom.\footnote{Actually it seems that even straight valleys can have
such a manifold if the slopes of the valley are bent, e.g. at the
entrance to the valley where the slopes are forming a ``funnel''.
Moreover, in one of the test functions studied below, there is
non-winding valley without bending walls where the criterion is zero at
the bottom.} The direction in which the valley is bent probably
determines the orientation of the manifold relative to the valley
bottom.

The divergence criterion contains derivatives up to the third order. I
assume that first-order derivatives alone would not be sufficient to
define a criterion since the gradient is affected by the type of
stationary point (minimum, maximum, saddle). Applying the divergence to
the gradient would just require second-order derivatives, but only
applying the divergence to the Newton step turns all types of stationary
points into attractors --- at the price of introducing third-order
derivatives into the divergence criterion.

To further process the expression of the divergence criterion
\eqref{eq_divcrit2}, we introduce and transform the Jacobian
$\mb{N}\mleft( \mb{x} \mright)$ of the Newton step $\bs{\nu}\mleft(
\mb{x} \mright)$
\begin{align}
&\phantom{{}={}} \mb{N}\mleft( \mb{x} \mright) \nonumber \\
%====================
& = \frac{\partial}{\partial \mb{x}}\,\bs{\nu}\mleft( \mb{x} \mright)\\
%====================
& = \frac{\partial}{\partial \mb{x}}\,\mleft( - \mleft[ \mb{H}\mleft( \mb{x}
\mright) \mright]^{-1} \bs{\gamma}\mleft( \mb{x} \mright) \mright)\\
%====================
\shortintertext{with \eqref{eq_ddx_Axbx} and omitting argument $\mb{x}$ we get}
%====================
& = - \mleft( \mleft( \mleft( \frac{\partial}{\partial x_{j}}\,\mb{H}^{-1}
\mright)_{i,*} \bs{\gamma} \mright)^{n\times n}_{i,j} + \mb{H}^{-1}
\mleft[ \frac{\partial}{\partial \mb{x}}\,\bs{\gamma} \mright] \mright)\\
%====================
& = - \mleft( \mleft( \mleft( \frac{\partial}{\partial x_{j}}\,\mb{H}^{-1}
\mright)_{i,*} \bs{\gamma} \mright)^{n\times n}_{i,j} + \mb{H}^{-1}
\mb{H} \mright)\\
%====================
& = \mleft( - \mb{I}_{n} \mright) - \mleft( \mleft( \frac{\partial}{\partial
x_{j}}\,\mb{H}^{-1} \mright)_{i,*} \bs{\gamma} \mright)^{n\times
n}_{i,j}\\
%====================
\shortintertext{we insert the derivative of matrix inverse given in \eqref{eq_inv_deriv}}
%====================
& = \mleft( - \mb{I}_{n} \mright) + \mleft( \mleft( \mb{H}^{-1} \mleft[
\frac{\partial}{\partial x_{j}}\,\mb{H} \mright] \mb{H}^{-1}
\mright)_{i,*} \bs{\gamma} \mright)^{n\times n}_{i,j}\\
%====================
\shortintertext{use \eqref{eq_ABix}}
%====================
& = \mleft( - \mb{I}_{n} \mright) + \mleft( \mleft( \mb{H}^{-1}
\mright)_{i,*} \mleft[ \frac{\partial}{\partial x_{j}}\,\mb{H} \mright]
\mb{H}^{-1} \bs{\gamma} \mright)^{n\times n}_{i,j}\\
%====================
\shortintertext{and insert Newton descent vector \eqref{eq_newton_nu}:}
%====================
&\label{eq_N_1}
 = \mleft( - \mb{I}_{n} \mright) - \mleft( \mleft( \mb{H}^{-1}
\mright)_{i,*} \mleft[ \frac{\partial}{\partial x_{j}}\,\mb{H} \mright]
\bs{\nu} \mright)^{n\times n}_{i,j}.\\
%====================
\shortintertext{An alternative form of this expression can be obtained from \eqref{eq_mat_prod}:}
%====================
&\label{eq_N_2}
 = \mleft( - \mb{I}_{n} \mright) - \mb{H}^{-1} \mleft( \mleft[
\frac{\partial}{\partial x_{j}}\,\mb{H} \mright] \bs{\nu}
\mright)^{n\times n}_{*,j}.
\end{align}
We observe that the Jacobian of the Newton step is not obviously
symmetric. This probably results from the fact that it isn't the second
derivative of a potential function, since the inverse Hessian modifies
the gradient.

Since for the divergence we only need the main diagonal of $\mb{N}$, we
take \eqref{eq_N_1} and get the diagonal element
\begin{align}
\mleft( \mb{N} \mright)_{i,i}
&=
\mleft( - 1 \mright) - \mleft( \mb{H}^{-1} \mright)_{i,*} \mleft(
\frac{\partial}{\partial x_{i}}\,\mb{H} \mright) \bs{\nu}.
\end{align}
We can insert this into \eqref{eq_divcrit1} and obtain
\begin{align}
&\phantom{{}={}} \tau\mleft( \mb{x} \mright) \nonumber \\
%====================
& = - {\,\frac{1}{n}\,} \opnl{tr} \mleft\{ \frac{\partial}{\partial
\mb{x}}\,\bs{\nu} \mright\}\\
%====================
& = - {\,\frac{1}{n}\,} \sum\limits_{i = 1}^{n}\mleft( \mb{N} \mright)_{i,i}\\
%====================
& = {\,\frac{1}{n}\,} \sum\limits_{i = 1}^{n}\mleft( 1 + \mleft( \mb{H}^{-1}
\mright)_{i,*} \mleft[ \frac{\partial}{\partial x_{i}}\,\mb{H} \mright]
\bs{\nu} \mright)\\
%====================
&\label{eq_tau}
 = 1 + {\,\frac{1}{n}\,} \sum\limits_{i = 1}^{n}\mleft( \mb{H}^{-1}
\mright)_{i,*} \mleft( \frac{\partial}{\partial x_{i}}\,\mb{H} \mright)
\bs{\nu}.
\end{align}
We see that for a purely quadratic function we have $\tau\mleft( \mb{x}
\mright) = 1$ regardless of the location $\mb{x}$ since the third-order
derivative in the second factor disappears. This is also an alternative
proof that the divergence is $1$ at stationary points where the function
is approximately quadratic.

A discussion of the effort for the computation of the third-order
derivatives of the Hessian $\mb{H}$ is provided in
section~\ref{sec_effort_third_order}.

For the line-search mechanism, we don't apply the divergence criterion
$\tau\mleft( \mb{x} \mright)$ directly, but define a criterion
\begin{align}
\label{eq_tauv}
\check{\tau}\mleft( \mb{x} \mright)
&=
\mleft( \tau\mleft( \mb{x} \mright) - 1 \mright)^2
\end{align}
which is zero at points where $\tau\mleft( \mb{x} \mright) = 1$ and
increases at larger distance from these points, forming a valley (with
level bottom). We see that the summand $1$ disappears from
\eqref{eq_tau} when inserted into \eqref{eq_tauv}.

\section{Line-Search Method}\label{sec_line_search}
%==================================================

%
The criterion \eqref{eq_tauv} alone is not sufficient for a line search,
and the line-search procedure in this criterion differs from the
classical one. This will be explored in the following.

\subsection{Zigzag Strategy}\label{sec_zigzag}
%---------------------------------------------

%
The valley\footnote{For simplicity, I will stick to the term ``valley''
(relating to the criterion), even though the term only applies to
two-dimensional functions. In higher dimensions, we probably have
manifolds where the criterion is zero, see above. Moreover, I will
switch to the term ``ravine'' for a valley in the criterion
$\check{\tau}$ later in the report.} in the criterion $\check{\tau}$
from \eqref{eq_tauv} has a {\em level} bottom. This is a crucial
difference to the classic line search which follows a valley (in the
objective function or merit function) with an {\em inclined} bottom. In
the classic case, the value of the function decreases along the line
until the value rises again when the line moves up the slope; therefore,
only this type of step is repeatedly applied.

In contrast, in the divergence criterion, we need an explicit ``zigzag''
strategy to follow a valley. The basic outline of the novel line-search
strategy is:

(1) \textbf{``down'' phase}: follow the Newton step and stop at the
bottom of a valley in criterion \eqref{eq_tauv} (where $\check{\tau}$ is
below a given first threshold, called the ``entry threshold''),

(2) \textbf{``zig'' phase}: follow the Newton step until criterion
\eqref{eq_tauv} has reached a second, higher threshold (called the
``escape threshold''), ending somewhere up the slopes of the valley, and

(3) \textbf{``zag'' phase}: return to the bottom of the valley be
sliding down the slope until the criterion again falls below the entry
threshold.

Phases (2) and (3) are repeated until some termination criterion is
fulfilled. Note that only phase (1) and (2) follow the Newton step.
Phase (1) may need to be repeated until a valley is found; this
corresponds to a Newton method without line search (but the last step is
cut short). A more detailed description of the implementation of the
zigzag strategy is given in section~\ref{sec_twod_newton_ls_impl}.

\subsection{Pullback Directions}\label{sec_pullback}
%---------------------------------------------------

%
To define a direction for the ``zag'' phase, we would need the gradient
of $\check{\tau}$ from \eqref{eq_tauv}. If we are just interested in the
gradient line (without direction and ignoring the length of the
gradient), this can be obtained by computing the gradient of $\tau$ from
\eqref{eq_tau}. However, doing this would produce equations with
fourth-order derivatives which might entail a prohibitively large
computational effort. We therefore explore two approaches to compute the
``pullback'' directions for the ``zag'' phase: use an approximation of
the gradient of $\tau$, or use the gradient of the determinant of the
Hessian. The latter is motivated by the observation (see below) that
often the valleys of $\check{\tau}$ and the singularity curves
$\opnl{det} \mleft\{ \mb{H} \mright\} = 0$ are running alongside each
other.

Interestingly, as is shown in appendix \ref{sec_pullback_equiv}, these
two approaches produce the same pullback directions (due to Schwarz's
theorem), in this way establishing a relationship between the
approximation of the gradient of the criterion and the more indirectly
motivated use of the singularity of the Hessian. Different
approximations of the gradient of $\tau$ could be explored in future
work.

\subsubsection{Approximated Gradient}
%''''''''''''''''''''''''''''''''''''

%
After the zig step, we are at the walls of the ravine, in some distance
of the zero curve of $\check{\tau}$ from \eqref{eq_tauv}. A search in
the direction of the negative gradient of $\check{\tau}$ would lead us
back to the ravine bottom. By applying the chain rule
\eqref{eq_chainrule1} we obtain
\begin{align}
&\phantom{{}={}} - \frac{\partial}{\partial \mb{x}}\,\check{\tau}\mleft( \mb{x} \mright) \nonumber \\
%====================
& = - \mleft( \frac{\partial}{\partial \tau}\,\mleft[ \tau - 1 \mright]^2
\mright) \mleft( \frac{\partial}{\partial \mb{x}}\,\tau\mleft( \mb{x}
\mright) \mright)\\
%====================
& = - 2 \mleft( \tau - 1 \mright) \mleft( \frac{\partial}{\partial
\mb{x}}\,\tau\mleft( \mb{x} \mright) \mright).
\end{align}
Since we are only interested in the line along which we can search and
not in the scaling or sign of the gradient, we only need to determine
the gradient of $\tau$ from \eqref{eq_tau} which can be done as follows:
\begin{align}
&\phantom{{}={}} \frac{\partial}{\partial \mb{x}}\,\tau \nonumber \\
%====================
& = \frac{\partial}{\partial \mb{x}}\,\mleft( {\,\frac{1}{n}\,}
\sum\limits_{i = 1}^{n}\mleft( \mb{H}^{-1} \mright)_{i,*} \mleft[
\frac{\partial}{\partial x_{i}}\,\mb{H} \mright] \bs{\nu} \mright)\\
%====================
\shortintertext{insert $\bs{\gamma}$}
%====================
& = - \frac{\partial}{\partial \mb{x}}\,\mleft( {\,\frac{1}{n}\,}
\sum\limits_{i = 1}^{n}\mleft( \mb{H}^{-1} \mright)_{i,*} \mleft[
\frac{\partial}{\partial x_{i}}\,\mb{H} \mright] \mb{H}^{-1} \bs{\gamma}
\mright)\\
%====================
\shortintertext{omit all derivatives above second order by assuming that only $\bs{\gamma}$ depends on $\mb{x}$, apply \eqref{eq_ddx_Abx}}
%====================
& \approx - {\,\frac{1}{n}\,} \sum\limits_{i = 1}^{n}\mleft( \mb{H}^{-1}
\mright)_{i,*} \mleft( \frac{\partial}{\partial x_{i}}\,\mb{H} \mright)
\mb{H}^{-1} \mleft( \frac{\partial}{\partial \mb{x}}\,\bs{\gamma}
\mright)\\
%====================
& = - {\,\frac{1}{n}\,} \sum\limits_{i = 1}^{n}\mleft( \mb{H}^{-1}
\mright)_{i,*} \mleft( \frac{\partial}{\partial x_{i}}\,\mb{H} \mright)
\mb{H}^{-1} \mb{H}\\
%====================
& = - {\,\frac{1}{n}\,} \sum\limits_{i = 1}^{n}\mleft( \mb{H}^{-1}
\mright)_{i,*} \mleft( \frac{\partial}{\partial x_{i}}\,\mb{H} \mright).
\end{align}
The length and sign of this approximated gradient is not relevant, so we
omit the normalization factor. We transpose the result to obtain the
gradient as a column vector (note that $\mb{H}$ is symmetric, and thus
also $\mb{H}^{-1}$ and the derivatives of $\mb{H}$). Finally, the result
is normalized to unit length:
\begin{align}
\label{eq_pt}
\tilde{\mb{p}}\mleft( \mb{x} \mright)
&=
\sum\limits_{i = 1}^{n}\mleft( \frac{\partial}{\partial x_{i}}\,\mb{H}
\mright) \mleft( \mb{H}^{-1} \mright)_{*,i}\\
%====================
\mb{p}\mleft( \mb{x} \mright)
&=
{\,\frac{\tilde{\mb{p}}\mleft( \mb{x} \mright)}{\mleft\|
\tilde{\mb{p}}\mleft( \mb{x} \mright) \mright\|}\,}.
\end{align}

\subsubsection{Singularity Curves}
%'''''''''''''''''''''''''''''''''

%
The singularity is described by the zero-crossings of $\opnl{det}
\mleft\{ \mb{H} \mright\}$. To compute the gradient of $\opnl{det}
\mleft\{ \mb{H} \mright\}$, we can use Jacobi's formula\footnote{Taken
from \url{https://en.wikipedia.org/wiki/Jacobi\%27s_formula}.}, here
adapted for the problem at hand:
\begin{align}
\frac{\partial}{\partial x_{k}}\,\opnl{det} \mleft\{ \mb{H} \mright\}
&=
\opnl{det} \mleft\{ \mb{H} \mright\} \opnl{tr} \mleft\{ \mb{H}^{-1}
\mleft( \frac{\partial}{\partial x_{k}}\,\mb{H} \mright) \mright\}
\end{align}
We determine the gradient as a column vector:
\begin{align}
&\phantom{{}={}} \mleft( \frac{\partial}{\partial \mb{x}}\,\opnl{det} \mleft\{ \mb{H}
\mright\} \mright)^T \nonumber \\
%====================
& = \opnl{det} \mleft\{ \mb{H} \mright\} \mleft( \opnl{tr} \mleft\{
\mb{H}^{-1} \mleft( \frac{\partial}{\partial x_{k}}\,\mb{H} \mright)
\mright\} \mright)^{n\times 1}_{k,1}\\
%====================
& = \opnl{det} \mleft\{ \mb{H} \mright\} \mleft( \sum\limits_{i =
1}^{n}\mleft( \mb{H}^{-1} \mleft[ \frac{\partial}{\partial
x_{k}}\,\mb{H} \mright] \mright)_{i,i} \mright)^{n\times 1}_{k,1}\\
%====================
& = \opnl{det} \mleft\{ \mb{H} \mright\} \mleft( \sum\limits_{i =
1}^{n}\mleft( \mb{H}^{-1} \mright)_{i,*} \mleft(
\frac{\partial}{\partial x_{k}}\,\mb{H} \mright)_{*,i} \mright)^{n\times
1}_{k,1}.
\end{align}
We are only interested in the line of the gradient vector, not in its
magnitude. Moreover, even the sign of the determinant in the first
factor is not required for the optimization step in the ``zag'' phase.
Therefore, we omit the factor with the determinant, normalize to unit
length and obtain
\begin{align}
\label{eq_qt}
\tilde{\mb{q}}\mleft( \mb{x} \mright)
&=
\mleft( \sum\limits_{i = 1}^{n}\mleft( \mb{H}^{-1} \mright)_{i,*}
\mleft( \frac{\partial}{\partial x_{k}}\,\mb{H} \mright)_{*,i}
\mright)^{n\times 1}_{k,1}\\
%====================
\mb{q}\mleft( \mb{x} \mright)
&=
{\,\frac{\tilde{\mb{q}}\mleft( \mb{x} \mright)}{\mleft\|
\tilde{\mb{q}}\mleft( \mb{x} \mright) \mright\|}\,}.
\end{align}

\section{Lagrange-Newton Method}\label{sec_lagrange_newton}
%==========================================================

%
The line-search method explored in this work was developed specifically
for use in the Lagrange-Newton method, i.e. the optimization of an
objective function under equality constraints. My interest in this
problem is mainly motivated by two applications of the Lagrange-Newton
method, one for principal component analysis \citep[][]{own_Moeller22},
the other for Smoothing-and-Mapping of maps \citep[][]{own_Moeller24}.
In the Lagrange-Newton framework, all stationary points are saddle
points (for which we provide a proof below). Therefore, we can't use the
value of the objective function (in this case the Lagrangian) for line
search, but would have to resort to merit functions (or to the
alternative method presented in this work). This section provides the
mathematical background.

\subsection{Notation}
%--------------------

%
\textbf{Dimensions:}
\begin{description}
\setlength{\itemsep}{0pt}
\item[] $n$: dimensionality of optimization problem, $3 \leq n$
\item[] $m$: number of equality constraints, $1 \leq m \leq n$
\end{description}
\textbf{Matrices:}
\begin{description}
\setlength{\itemsep}{0pt}
\item[] $\mb{x}$: state vector, arbitrary, $n \times 1$
\item[] $\bar{\mb{x}}$: stationary point, arbitrary, $n \times 1$
\item[] $\mb{z}$: extended state vector, arbitrary, $(n + m) \times 1$
\item[] $f\mleft( \mb{x} \mright)$: objective function, scalar, $1 \times 1$
\item[] $\mb{g}\mleft( \mb{x} \mright)$: vector of equality constraints, arbitrary, $m \times 1$
\item[] $g_{j}\mleft( \mb{x} \mright)$: single equality constraint, scalar, $1 \times 1$, vector element of $\mb{g}$
\item[] $\bs{\lambda}$: vector of Lagrange multipliers, arbitrary, $m \times 1$
\item[] $\lambda_{j}\mleft( \mb{x} \mright)$: single Lagrange multiplier, scalar, $1 \times 1$, vector element of $\bs{\lambda}$
\item[] $L\mleft( \mb{x},\bs{\lambda} \mright)$: Lagrangian, scalar, $1 \times 1$
\item[] $\mb{H}_b$: bordered Hessian, symmetric, $(n + m) \times (n + m)$
\item[] $\bar{\mb{H}}_b$: bordered Hessian at stationary point, symmetric, $(n + m) \times (n + m)$
\end{description}

\subsection{Lagrangian}\label{sec_lagrangian}
%--------------------------------------------

%
Given an objective function $f\mleft( \mb{x} \mright)$, equality
constraints $\mb{g}\mleft( \mb{x} \mright) = \mb{0}_{m}$, and Lagrange
multipliers $\bs{\lambda}$, the Lagrangian of the constrained
optimization problem is defined as
\begin{align}
\label{eq_lagrangian}
L\mleft( \mb{x},\bs{\lambda} \mright)
&=
f\mleft( \mb{x} \mright) + \bs{\lambda}^T \mb{g}\mleft( \mb{x} \mright).
\end{align}
We can then introduce a joint (or extended) state vector
\begin{align}
\mb{z}
&=
\begin{pmatrix}\mb{x}\\\bs{\lambda}\end{pmatrix}
\end{align}
and (in a trivial implementation) apply Newton's method to the
Lagrangian as objective function. We will show below that all stationary
points of the Lagrangian are saddle points. For this, we need to
introduce the special form of the Hessian for the constrained problem,
called the ``bordered Hessian''.

\subsection{Bordered Hessian}\label{sec_bordered_hessian}
%--------------------------------------------------------

%
In the following we explore whether the extended objective function
(over an extended variable vector $\mb{z}$ containing the original
variables and the Lagrange multipliers) has saddle points at all fixed
points. We start by deriving the general bordered Hessian.

The first-order derivatives (gradient) of the Lagrangian
\eqref{eq_lagrangian} are
\begin{align}
\label{eq_lagrangian_dx}
\frac{\partial}{\partial \mb{x}}\,L\mleft( \mb{x},\bs{\lambda} \mright)
&=
\frac{\partial}{\partial \mb{x}}\,f\mleft( \mb{x} \mright) +
\bs{\lambda}^T \mleft( \frac{\partial}{\partial \mb{x}}\,\mb{g}\mleft(
\mb{x} \mright) \mright)\\
%====================
\frac{\partial}{\partial \bs{\lambda}}\,L\mleft( \mb{x},\bs{\lambda}
\mright)
&=
\mleft( \mb{g}\mleft( \mb{x} \mright) \mright)^T.
\end{align}
To form the bordered Hessian
\begin{align}
\mb{H}_b
&=
\begin{pmatrix}\frac{\partial}{\partial \mb{x}}\,\mleft(
\frac{\partial}{\partial \mb{x}}\,L\mleft( \mb{x},\bs{\lambda} \mright)
\mright)^T & \frac{\partial}{\partial \bs{\lambda}}\,\mleft(
\frac{\partial}{\partial \mb{x}}\,L\mleft( \mb{x},\bs{\lambda} \mright)
\mright)^T\\\frac{\partial}{\partial \mb{x}}\,\mleft(
\frac{\partial}{\partial \bs{\lambda}}\,L\mleft( \mb{x},\bs{\lambda}
\mright) \mright)^T & \frac{\partial}{\partial \bs{\lambda}}\,\mleft(
\frac{\partial}{\partial \bs{\lambda}}\,L\mleft( \mb{x},\bs{\lambda}
\mright) \mright)^T\end{pmatrix}
\end{align}
we need to determine the second-order derivatives in the four blocks. We
start with the upper-left block:
\begin{align}
&\phantom{{}={}} \frac{\partial}{\partial \mb{x}}\,\mleft( \frac{\partial}{\partial
\mb{x}}\,L\mleft( \mb{x},\bs{\lambda} \mright) \mright)^T \nonumber \\
%====================
& = \frac{\partial}{\partial \mb{x}}\,\mleft( \frac{\partial}{\partial
\mb{x}}\,f\mleft( \mb{x} \mright) \mright)^T + \frac{\partial}{\partial
\mb{x}}\,\mleft( \bs{\lambda}^T \mleft[ \frac{\partial}{\partial
\mb{x}}\,\mb{g}\mleft( \mb{x} \mright) \mright] \mright)^T\\
%====================
& = \frac{\partial}{\partial \mb{x}}\,\mleft( \frac{\partial}{\partial
\mb{x}}\,f\mleft( \mb{x} \mright) \mright)^T + \frac{\partial}{\partial
\mb{x}}\,\mleft( \mleft[ \frac{\partial}{\partial \mb{x}}\,\mb{g}\mleft(
\mb{x} \mright) \mright]^T \bs{\lambda} \mright).
\end{align}
For the second term, we cannot find a closed-form matrix expression, but
we can compute the derivative for row $i$ according to
\eqref{eq_ddx_AxTb}
\begin{align}
&\phantom{{}={}} \mleft( \frac{\partial}{\partial \mb{x}}\,\mleft[ \mleft\{
\frac{\partial}{\partial \mb{x}}\,\mb{g}\mleft( \mb{x} \mright)
\mright\}^T \bs{\lambda} \mright] \mright)_{i,*} \nonumber \\
%====================
& = \bs{\lambda}^T \mleft( \frac{\partial}{\partial \mb{x}}\,\mleft[
\frac{\partial}{\partial x_{i}}\,\mb{g}\mleft( \mb{x} \mright) \mright]
\mright)
\end{align}
which in turn can be expressed in single-element form as
\begin{align}
\frac{\partial}{\partial \mb{x}}\,\mleft( \mleft[
\frac{\partial}{\partial \mb{x}}\,\mb{g}\mleft( \mb{x} \mright)
\mright]^T \bs{\lambda} \mright)
&=
\mleft( \sum\limits_{k = 1}^{m}\lambda_{k} \mleft[
\frac{\partial}{\partial x_{j}}\,\mleft\{ \frac{\partial}{\partial
x_{i}}\,g_{k}\mleft( \mb{x} \mright) \mright\} \mright]
\mright)^{n\times n}_{i,j},
\end{align}
thus we get
\begin{align}
\frac{\partial}{\partial \mb{x}}\,\mleft( \frac{\partial}{\partial
\mb{x}}\,L\mleft( \mb{x},\bs{\lambda} \mright) \mright)^T
&=
\frac{\partial}{\partial \mb{x}}\,\mleft( \frac{\partial}{\partial
\mb{x}}\,f\mleft( \mb{x} \mright) \mright)^T + \mleft( \sum\limits_{k =
1}^{m}\lambda_{k} \mleft[ \frac{\partial}{\partial x_{j}}\,\mleft\{
\frac{\partial}{\partial x_{i}}\,g_{k}\mleft( \mb{x} \mright) \mright\}
\mright] \mright)^{n\times n}_{i,j}.
\end{align}
The other three blocks of the bordered Hessian are easier to determine
(note that a Hessian is always symmetric):
\begin{align}
\frac{\partial}{\partial \bs{\lambda}}\,\mleft( \frac{\partial}{\partial
\mb{x}}\,L\mleft( \mb{x},\bs{\lambda} \mright) \mright)^T
&=
\mleft( \frac{\partial}{\partial \mb{x}}\,\mb{g}\mleft( \mb{x} \mright)
\mright)^T
\end{align}
\begin{align}
\frac{\partial}{\partial \mb{x}}\,\mleft( \frac{\partial}{\partial
\bs{\lambda}}\,L\mleft( \mb{x},\bs{\lambda} \mright) \mright)^T
&=
\frac{\partial}{\partial \mb{x}}\,\mb{g}\mleft( \mb{x} \mright)
\end{align}
\begin{align}
\frac{\partial}{\partial \bs{\lambda}}\,\mleft( \frac{\partial}{\partial
\bs{\lambda}}\,L\mleft( \mb{x},\bs{\lambda} \mright) \mright)^T
&=
\mb{0}_{m,m}.
\end{align}
The eigenvalues of the bordered Hessian at a given fixed point determine
the stability of this fixed point. I'm not aware of any statement on the
specific eigenvalues for the general case. However, in the following
section, we prove that the bordered Hessian is indefinite and thus has
both negative and positive eigenvalues. This demonstrates that all
stationary points are saddle points, demonstrating the necessity of
applying Newton's method rather than a gradient descent.

\subsection{Proof of Saddle Points}
%----------------------------------

%
We start the proof that all stationary points of the Lagrangian
\eqref{eq_lagrangian} are saddle points\footnote{A hint at the basic
idea of the proof is given at
\url{https://en.wikipedia.org/wiki/Hessian_matrix\#Bordered_Hessian}.}
from the following Lemma:
\begin{lemma}\label{lemma_zero_dir}
Given $n$: $1 \leq n$ and $\mb{H}$: symmetric, $n \times n$;
$\bar{\mb{H}}$: symmetric, $n \times n$; $\mb{x}$: arbitrary, $n \times
1$; $\bar{\mb{x}}$: arbitrary, $n \times 1$; $\mb{0}_{n}$: zero, $n
\times 1$; $f$: scalar, $1 \times 1$.

Let $\bar{\mb{H}} = \mb{H}\mleft( \bar{\mb{x}} \mright)$ be the Hessian
matrix of a function $f\mleft( \mb{x} \mright)$, evaluated at an extreme
point $\bar{\mb{x}}$ of $f$. Assume that $\bar{\mb{H}}$ is non-singular.
If an $\mb{x} \ne \mb{0}_{n}$ can be found for which $\mb{x}^T
\bar{\mb{H}} \mb{x} = 0$, then $\bar{\mb{H}}$ is indefinite and
$\bar{\mb{x}}$ is a saddle point of $f$.
\end{lemma}
\begin{proof}\label{proof_zero_dir}
Kindly provided by Jyrki Lahtonen
(\url{https://math.stackexchange.com/a/4805925}; adapted notation).

Given $\bs{\Lambda}$: diagonal, $n \times n$; $\lambda_{i}$: scalar, $1
\times 1$, diagonal element of $\bs{\Lambda}$; $\mb{V}$: orthogonal, $n
\times n$; $\mb{v}_{i}$: arbitrary, $n \times 1$, column vector of
$\mb{V}$; $\mb{z}$: arbitrary, $n \times 1$; $z_{i}$: scalar, $1 \times
1$, vector element of $\mb{z}$; $a$: scalar, $1 \times 1$.

Since $\bar{\mb{H}}$ is symmetric, we can apply a spectral decomposition
(yielding real eigenvalues $\lambda_{i}$ and orthogonal eigenvectors
$\mb{v}_{i}$):
\begin{align}
\bar{\mb{H}}
&=
\sum\limits_{i = 1}^{n}\lambda_{i} \mb{v}_{i} \mb{v}_{i}^T.
\end{align}
Now we can write
\begin{align}
\label{eq_zero_dir_Qx}
Q\mleft( \mb{x} \mright) & = \mb{x}^T \bar{\mb{H}} \mb{x} = \sum\limits_{i = 1}^{n}\lambda_{i} \mleft( \mb{x}^T \mb{v}_{i} \mright)
\mleft( \mb{v}_{i}^T \mb{x} \mright) = \sum\limits_{i = 1}^{n}\lambda_{i} z_{i}^{2}
\end{align}
where the $z_{i}$ are the projections of $\mb{x}$ onto the eigenvectors
$\mb{v}_{i}$.

Since $\bar{\mb{H}}$ is assumed to be non-singular, we have $\lambda_{i}
\ne 0 \; \forall i = 1\ldots n$.

The chosen vector $\mb{x}$ cannot be a multiple of a single eigenvector
$\mb{v}_{i}$ (i.e. $\mb{x} = a \mb{v}_{i} $ for $a \ne 0$), since then
we would have $\mb{x}^T \bar{\mb{H}} \mb{x} = a^{2} \lambda_{i} $; since
all eigenvalues are non-zero, the assumption $\mb{x}^T \bar{\mb{H}}
\mb{x} = 0$ for $\mb{x} \ne \mb{0}_{n}$ of the lemma would be violated.
We therefore have a least two non-zero projections $z_{i}$ and $z_{j}$
for $i \neq j$ (only two if $\mb{x}$ happens to lie in the sub-space
spanned by $\mb{v}_{i}$ and $\mb{v}_{j}$).

In order to fulfill $Q = 0$ for the given $\mb{x}$, the terms of the sum
in equation \eqref{eq_zero_dir_Qx} need to cancel out. Since we have at
least two non-zero terms, this means that there must exist indices $i$
and $j$ ($i \neq j$) with $\lambda_{i} > 0$ and $\lambda_{j} < 0$;
therefore the matrix $\bar{\mb{H}}$ is indefinite.
\end{proof}
\begin{proof}\label{proof_saddle}
We can now apply Lemma \ref{lemma_zero_dir} to the bordered Hessian at a
stationary point $\bar{\mb{x}}$
\begin{align}
\bar{\mb{H}}_b
&=
\mleft.{\begin{pmatrix}\frac{\partial}{\partial \mb{x}}\,\mleft(
\frac{\partial}{\partial \mb{x}}\,L\mleft( \mb{x},\bs{\lambda} \mright)
\mright)^T & \mleft( \frac{\partial}{\partial \mb{x}}\,\mb{g}\mleft(
\mb{x} \mright) \mright)^T\\\frac{\partial}{\partial
\mb{x}}\,\mb{g}\mleft( \mb{x} \mright) &
\mb{0}_{m,m}\end{pmatrix}}\mright|_{\mb{x} = \bar{\mb{x}}}.
\end{align}
We have to assume that the bordered Hessian is non-singular
(particularly in the stationary point) --- if that wouldn't be the case,
a Newton descent (see below) wouldn't work, since it requires the
inversion of the bordered Hessian.\footnote{There are objective
functions in unconstrained optimization where the Hessian is singular in
a stationary point, e.g. the Bazaraa-Shetty function
\citep[][p.~15]{nn_Alt11}, a fourth-order polynomial. Optimization
textbooks specify conditions for the singularity of the bordered
Hessian, e.g. \citet[p.~240]{nn_Geiger02} and \citet[p.~288]{nn_Alt11},
but I could not derive from these conditions whether singular bordered
Hessians in stationary points are just edge cases.} If we pick a
direction $\mb{z}$
\begin{align}
\mb{z}
&=
\begin{pmatrix}\mb{0}_{n}\\\mb{u}\end{pmatrix}
\end{align}
with $\mb{u} \ne \mb{0}_{m}$ (where $\mb{z}$: arbitrary, $(n + m) \times
1$; $\mb{u}$: arbitrary, $m \times 1$), we see that $\mb{z}^T \mb{H}_b
\mb{z} = 0$, so the assumptions of Lemma \ref{lemma_zero_dir} are
fulfilled and therefore the bordered Hessian is indefinite.\footnote{It
may be interesting to observe that we didn't have to insert the
stationary point into the equation of the bordered Hessian --- the
bordered Hessian is indefinite everywhere, except in singularities.}
\end{proof}
This implies that a gradient descent in the vector comprising both the
original variables and the Lagrange multipliers is not possible since
each fixed point is a saddle point; instead, a Newton descent needs to
be applied. Also, as mentioned above, a line search can't be done in the
Lagrangian, but only in merit functions, or with the novel line-search
criterion introduced here (combined with a special line-search
strategy).

\section{Experiments on 2D Test Functions}\label{sec_exp_twod}
%=============================================================

%
In this section, the results of numerical simulations for 2D test
functions are presented. We first provide plots of the singularity of
the Hessian $\opnl{det} \mleft\{ \mb{H} \mright\} = 0$ and of the
zero-crossings of the divergence criterion $\tau$ from \eqref{eq_tau} as
well as of contours of $\check{\tau}$ from \eqref{eq_tauv}
(section~\ref{sec_twod_sing_crit}) and then show trajectories for the
different versions of Newton's method and time courses of the
line-search criteria (section~\ref{sec_twod_newton_ls}).

All plots where produced with Python using additional packages for
numerical processing, symbolic processing, and plotting.\footnote{Python
3.11.13, NumPy 1.26.4, SciPy 1.12.0, SymPy 1.12, Matplotlib 3.8.3}

\subsection{Hessian Singularity and Divergence Criterion}\label{sec_twod_sing_crit}
%----------------------------------------------------------------------------------

\subsubsection{Plot Description}\label{sec_twod_sing_crit_plot}
%''''''''''''''''''''''''''''''''''''''''''''''''''''''''''''''

%
\textbf{Figure~\ref{fig_crit_rsnbrk_all}} shows a 2D plot related to the
test function Rosenbrock-wide (see section~\ref{sec_rosenbrock}). In
this and in the following plots of the same type, the name of the test
function is given in the title, with the parameters of the test function
(here $a$, $b$, $c$) added in brackets (or empty brackets, if the
function has no parameters). The contour lines of the test function are
shown with thin blue lines. The roughly logarithmic contour levels
appear in blue font. Solid contour lines correspond to positive levels,
dashed contour lines to negative levels. Gray arrows with wide heads
represent the negative gradients of the test function. Red arrows with
narrow heads are Newton vectors. All vectors are normalized to the same
length for the plot.

The Hessian singularity is plotted as a thick orange curve. For the
visualization of the line-search criterion $\check{\tau}$, we use three
plot elements. The zero crossing of $\tau$ (which captures some points
where $\check{\tau} = 0$) appear as a dark green curve. This curve is
surrounded by dotted lines (better visible in
\textbf{Figure~\ref{fig_crit_rsnbrk_aux}}) which indicate the transition
threshold in $\check{\tau}$ from the zig to the zag phase of the
line-search strategy (see section~\ref{sec_zigzag}). In addition, a
colored contour plot from dark green to white (with logarithmic scaling,
ranging from $10^{-7}$ to $2\cdot 10^{-1}$, the latter roughly
corresponding to the transition threshold) shows the values of
$\check{\tau}$.

The Python Matplot library has the property that zero-crossings of
$\tau$ are also shown at singularities where $\tau$ goes to positive
infinity on one side and to negative infinity on the other side.
Therefore, there may be a green zero-crossing curve hidden underneath
the orange curve of the Hessian singularity even though $\check{\tau}$
is non-zero at these points; the plotting order is chosen such that the
orange curve is on top. This singularity of $\tau$ also explains that
the green and orange curves cross over in some plots, e.g. visible in
\textbf{Figure~\ref{fig_crit_hmlbl}}. In these cross-over points, the
dotted, dark green threshold curves approach the dark green solid curve
of the zero crossing, indicating that $\tau$ gets increasingly larger
slopes when approaching the cross-over point and has a singularity at
the cross-over point.

\subsubsection{Experiments}\label{sec_twod_sing_crit_exp}
%''''''''''''''''''''''''''''''''''''''''''''''''''''''''

%
\textbf{Rosenbrock-wide} (see section~\ref{sec_rosenbrock}): We chose
this version with wider valley over the original Rosenbrock function
since the plots are easier to interpret. In
\textbf{Figure~\ref{fig_crit_rsnbrk_all}} we see that the singularity
curve of the Hessian runs close to the bottom of the narrow valley of
the function value (as computed in section~\ref{sec_rosenbrock}). The
zero crossing of $\tau$ runs exactly at the bottom of the valley,
judging from the fact that it passes through $(0,0)$ and the minimum at
$(1,1)$. In large regions above and below the valley, the Newton vectors
point towards the valley but are surprisingly almost aligned with the
$y$-axis.
\begin{figure}[tp]
\begin{center}
\includegraphics[width=12cm]{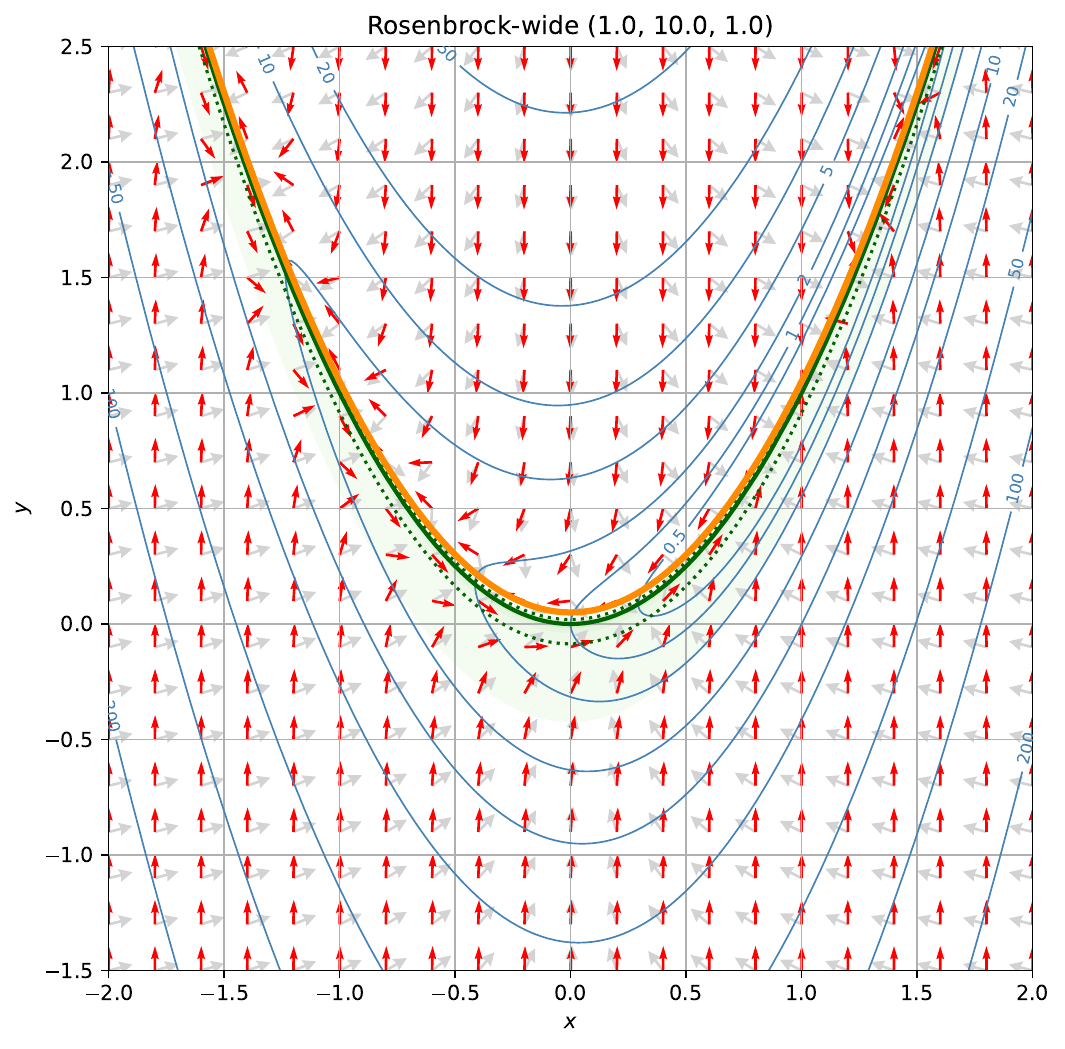}
\caption{}
\label{fig_crit_rsnbrk_all}
\end{center}
\end{figure}
\begin{figure}[tp]
\begin{center}
\includegraphics[height=0.47\textwidth]{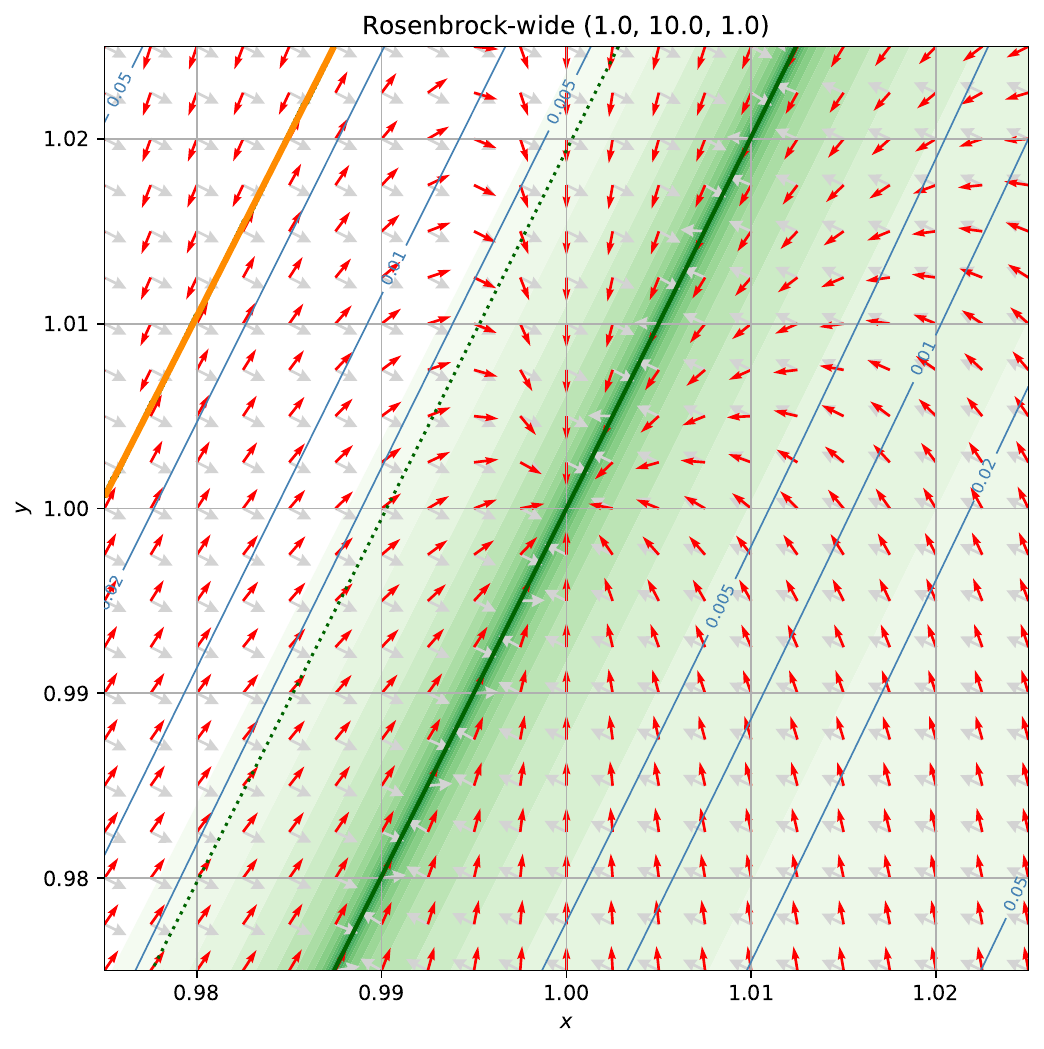}
\hspace*{1mm}
\includegraphics[height=0.47\textwidth]{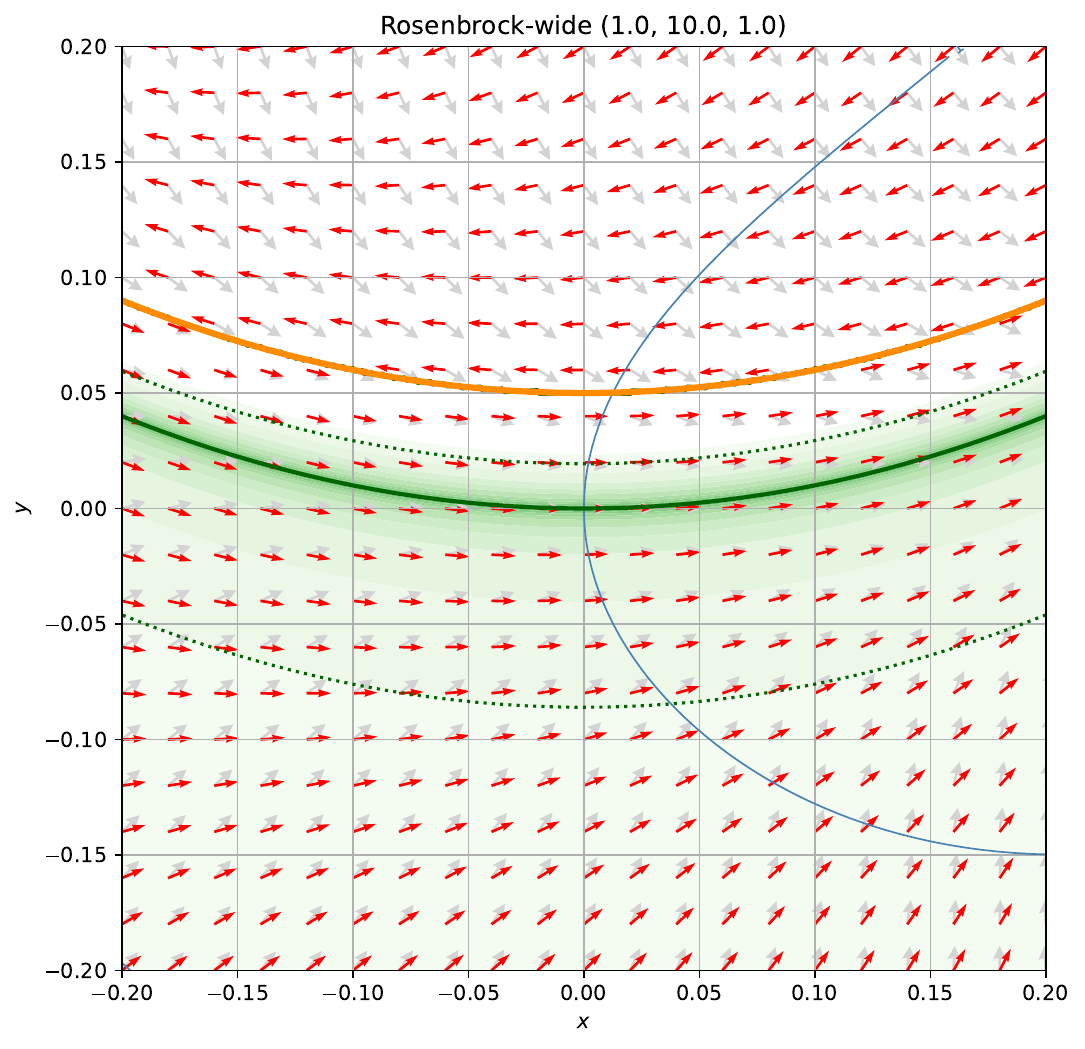}
\caption{}
\label{fig_crit_rsnbrk_aux}
\end{center}
\end{figure}
\textbf{Figure~\ref{fig_crit_rsnbrk_aux}} shows close-ups of the regions
around the minimum at $(1,1)$ (left) and around $(0,0)$ (right). We see
that the Hessian singularity is surrounded by Newton vectors pointing
{\em along} the singularity curve, but into opposite directions. This
seems to be a special type of singularity curve which is related to
winding valleys (or bent valley slopes). In the following, we refer to
this type as ``countercurrent singularities''. In contrast, Hessian
singularities separating the regions corresponding to different
stationary points are surrounded by Newton vectors with roughly opposite
directions, but these point in directions roughly {\em perpendicular} to
the singularity curve, as can be seen for Himmelblau's function in
\textbf{Figure~\ref{fig_crit_hmlbl}}. In the following, we refer to the
latter type as ``inflection singularities''.

Far from the minimum, the Newton vectors point along the valley
(\textbf{Figure~\ref{fig_crit_rsnbrk_aux}}, right), and in the vicinity
of the minimum, all Newton vectors point towards the stationary point,
as can be expected from Newton's method
(\textbf{Figure~\ref{fig_crit_rsnbrk_aux}}, left). Thus, once the method
has reached the valley bottom, it can approach the minimum by following
the ``ravine'' of the criterion $\check{\tau}$.
\textbf{Rosenbrock-wide-saddle}: If the parameter $b$ of Rosenbrock's
function is negative, the stationary point is a saddle. In this case, we
can't speak of a ``valley'' in the objective function. As
\textbf{Figure~\ref{fig_crit_rsnbrk_saddle_all}} shows, the Hessian
singularity curve now appears on the other side of the zero-crossing
curve of $\tau$. Also note the gradient vectors vectors which in this
case point away from the zero-crossing curve in larger distances. The
Newton vectors show the same overall behavior as in the minimum version
of the function. \textbf{Figure~\ref{fig_crit_rsnbrk_aux}} provides
close-ups around $(1,1)$ and $(0,0)$. Again, we see the ``countercurrent
flow'' of the Newton vectors adjacent to the Hessian singularity. Close
to the zero-crossing curve of $\tau$, these vectors point towards the
stationary point. This is promising, as the line-search strategy can
follow the Newton vectors once it has reached the $\check{\tau}$-ravine,
even in the case of a saddle point.
\begin{figure}[tp]
\begin{center}
\includegraphics[width=12cm]{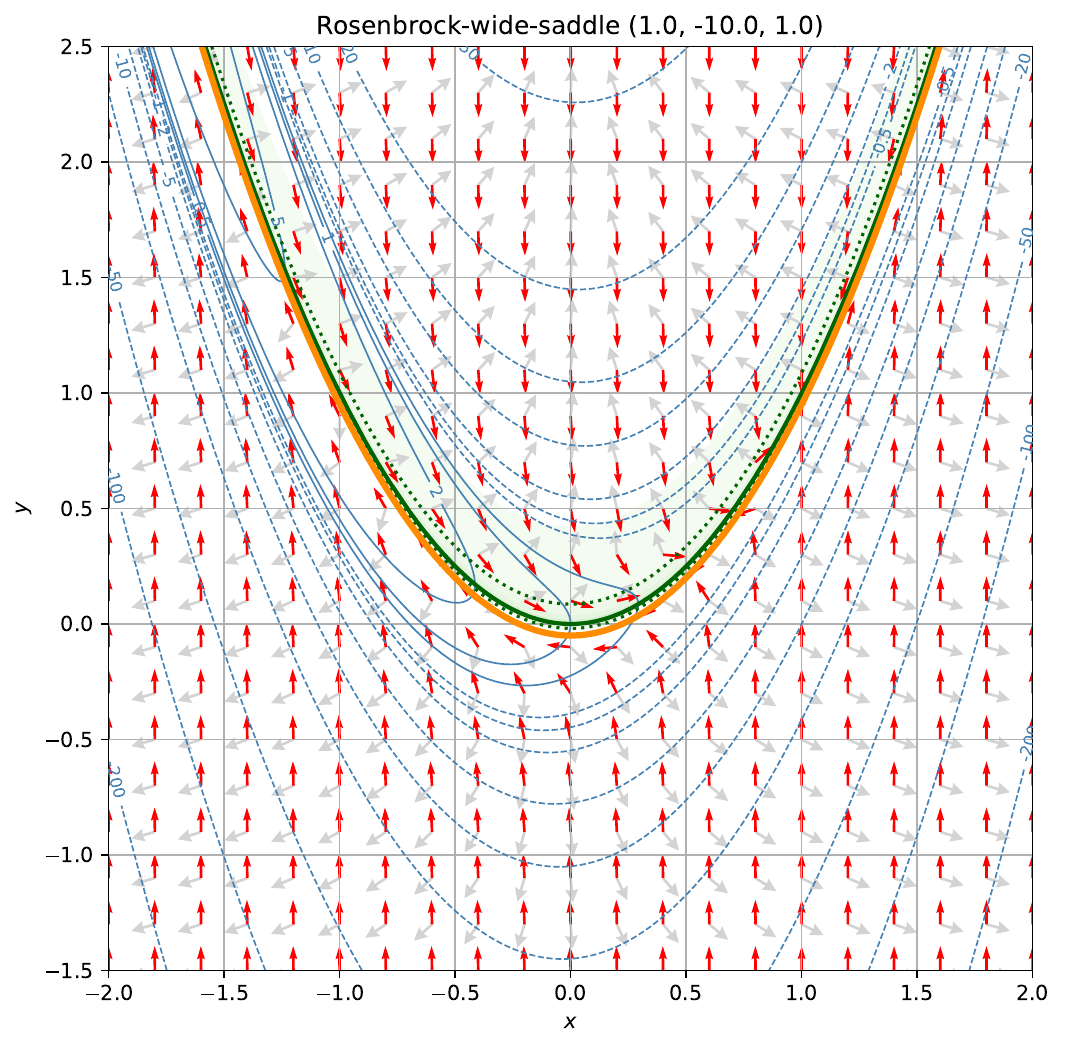}
\caption{}
\label{fig_crit_rsnbrk_saddle_all}
\end{center}
\end{figure}
\begin{figure}[tp]
\begin{center}
\includegraphics[height=0.47\textwidth]{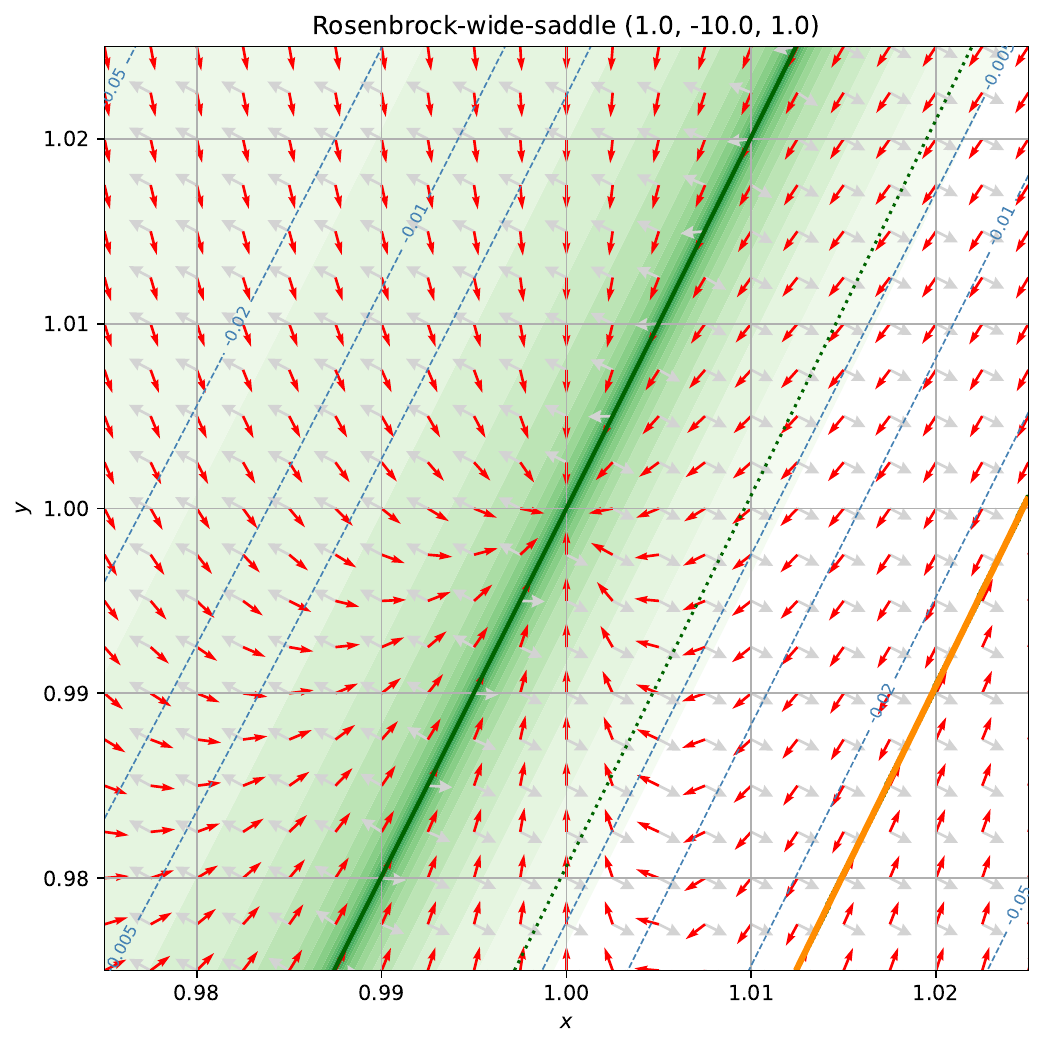}
\hspace*{1mm}
\includegraphics[height=0.47\textwidth]{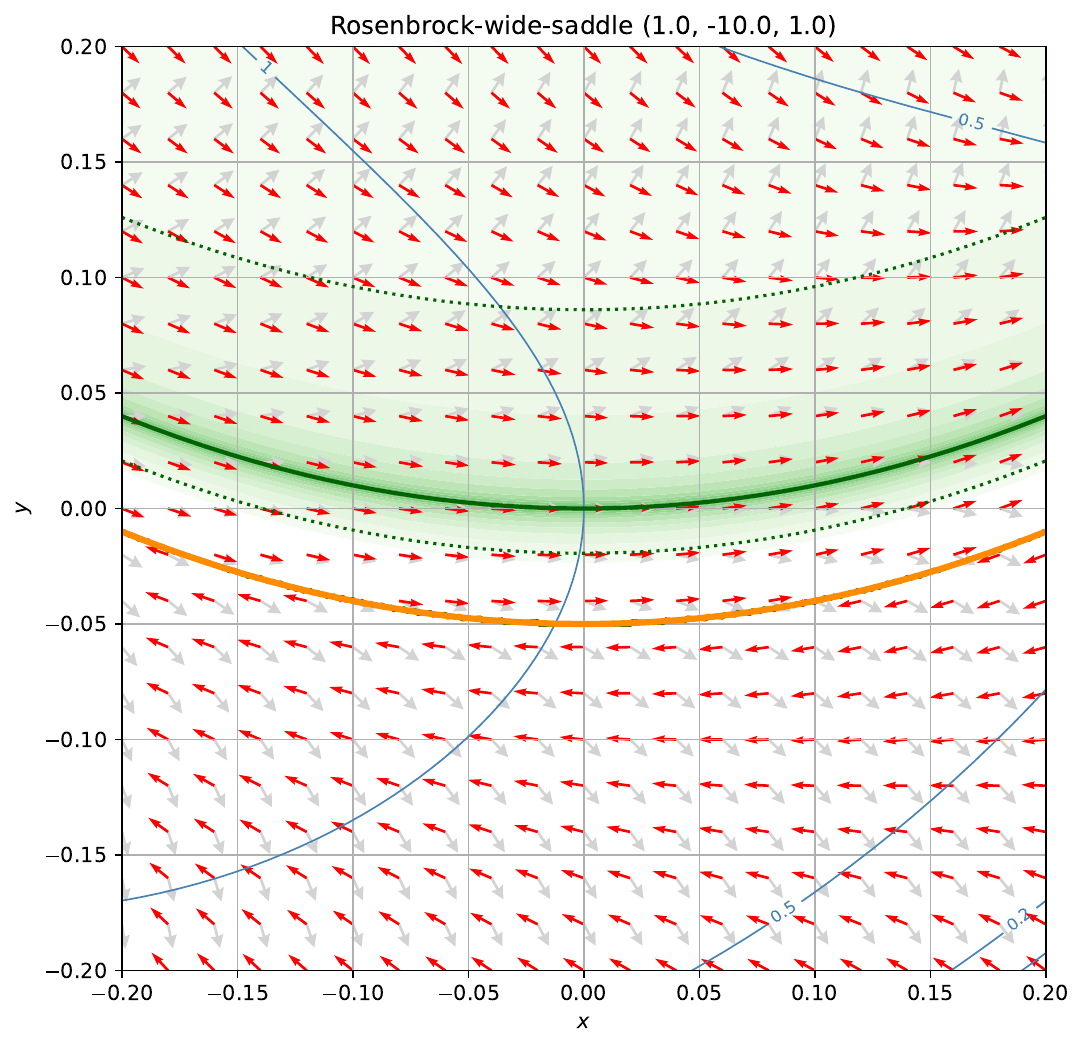}
\caption{}
\label{fig_crit_rsnbrk_saddle_aux}
\end{center}
\end{figure}
\textbf{Rosenbrock-ditch-wide}: However, the simple situation seen in
the Rosenbrock functions (with minimum or saddle) breaks down even with
an apparently minor modification. \textbf{Figure~\ref{fig_crit_ditch}}
shows the Rosenbrock-ditch-wide function (introduced in this work; see
section~\ref{sec_ditch}). In contrast to the Rosenbrock function, it is
obtained by a bending deformation applied to a function which is still
quadratic in $x$-direction but now ``ditch-shaped'' in $y$-direction. In
the 1D ditch shape, the function is quadratic close to the valley
bottom, but approaches a constant value in some distance; it has an
inflection point in between.

Instead of a single ``countercurrent'' singularity and a single
$\tau$-zero-crossing at the bottom of the valley in the Rosenbrock
functions, \textbf{Figure~\ref{fig_crit_ditch}} reveals a total of {\em
four} singularities (two of them of the ``countercurrent'' type) and
even {\em six} $\tau$-zero-crossings. Two of the additional
singularities are caused by the inflection curves of the ditch shape; it
is not clear to me why there is a fourth singularity (of the
countercurrent type, seen at the top in the figure).

There is still a $\tau$-zero-crossing ($\check{\tau}$-ravine) at the
bottom of the valley (running through $(0,0)$ and through the minimum at
$(1,1)$, as in the Rosenbrock functions), but there is an additional
$\check{\tau}$-ravine running adjacent to it. Beyond the
inflection-point singularities, two more $\check{\tau}$-ravines appear,
and further two which are supposedly related to the countercurrent
singularity seen at the top of the figure.

This example reveals the limits of following a $\check{\tau}$-ravine as
suggested in this work. If inflection-point singularities are present,
the increasing number of $\check{\tau}$-ravines may misdirect the line
search. However, is is also visible in the figure that even the plain
Newton method likely fails for a large portion of the starting points in
the range of the figure, namely trajectories starting outside the region
delimited by the two inflection-point singularities.
\begin{figure}[tp]
\begin{center}
\includegraphics[width=12cm]{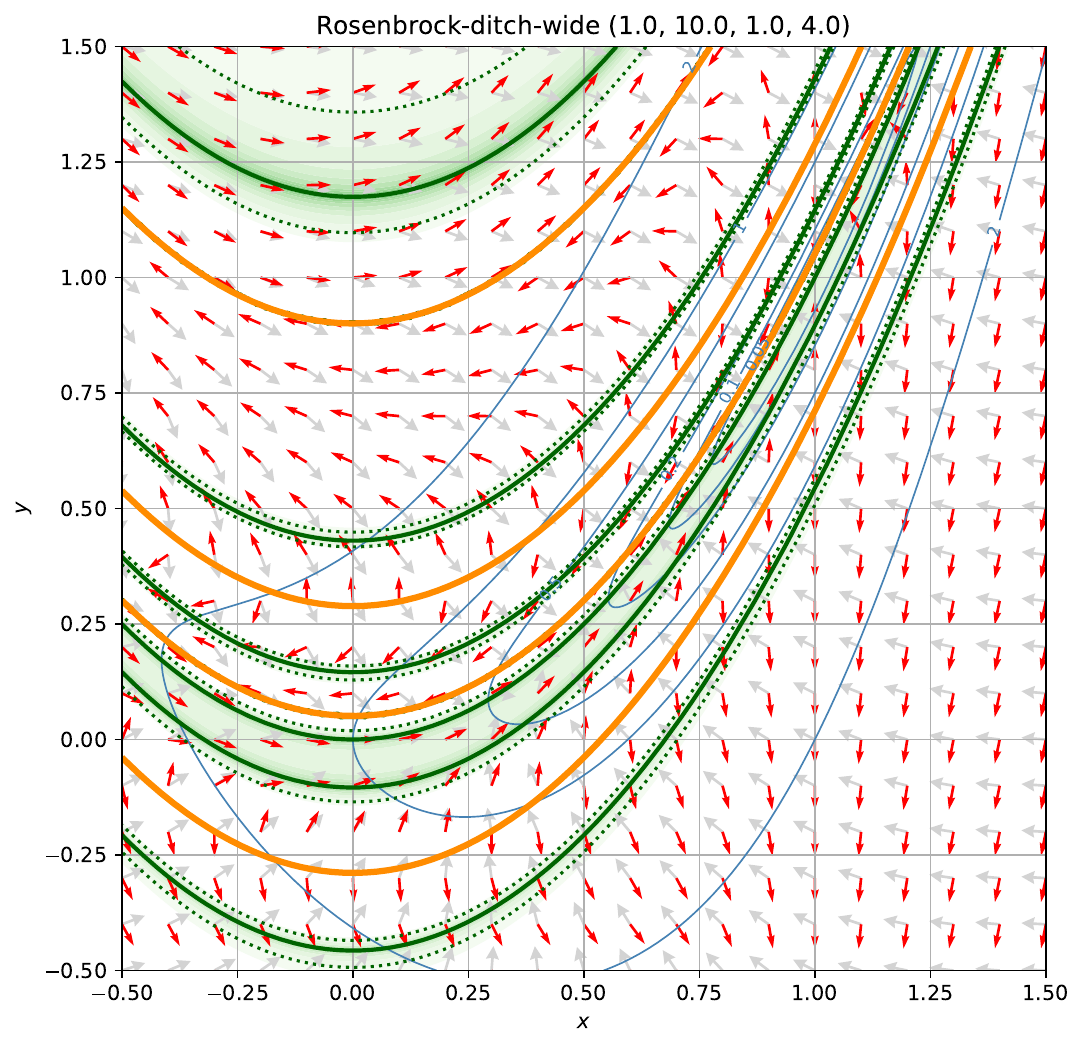}
\caption{}
\label{fig_crit_ditch}
\end{center}
\end{figure}
\textbf{Rosenbrock-ditch-wide-straight}: For comparison,
\textbf{Figure~\ref{fig_crit_ditch_straight}} shows the Rosenbrock ditch
function for $c = 0$. In this case, the ditch is straight. As can be
expected, two inflection singularities are present. The two
countercurrent singularities disappear for $c = 0$ (indicating that
these are due to the bending). There are three $\check{\tau}$-ravines,
one of them running at the bottom of the valley of the objective
function. In contrast, if $c = 0$ is used for the original Rosenbrock
function, we obtain $\check{\tau} = 0$ everywhere since the third-order
terms in \eqref{eq_tau} are zero for purely quadratic functions.
\begin{figure}[t]
\begin{center}
\includegraphics[width=12cm]{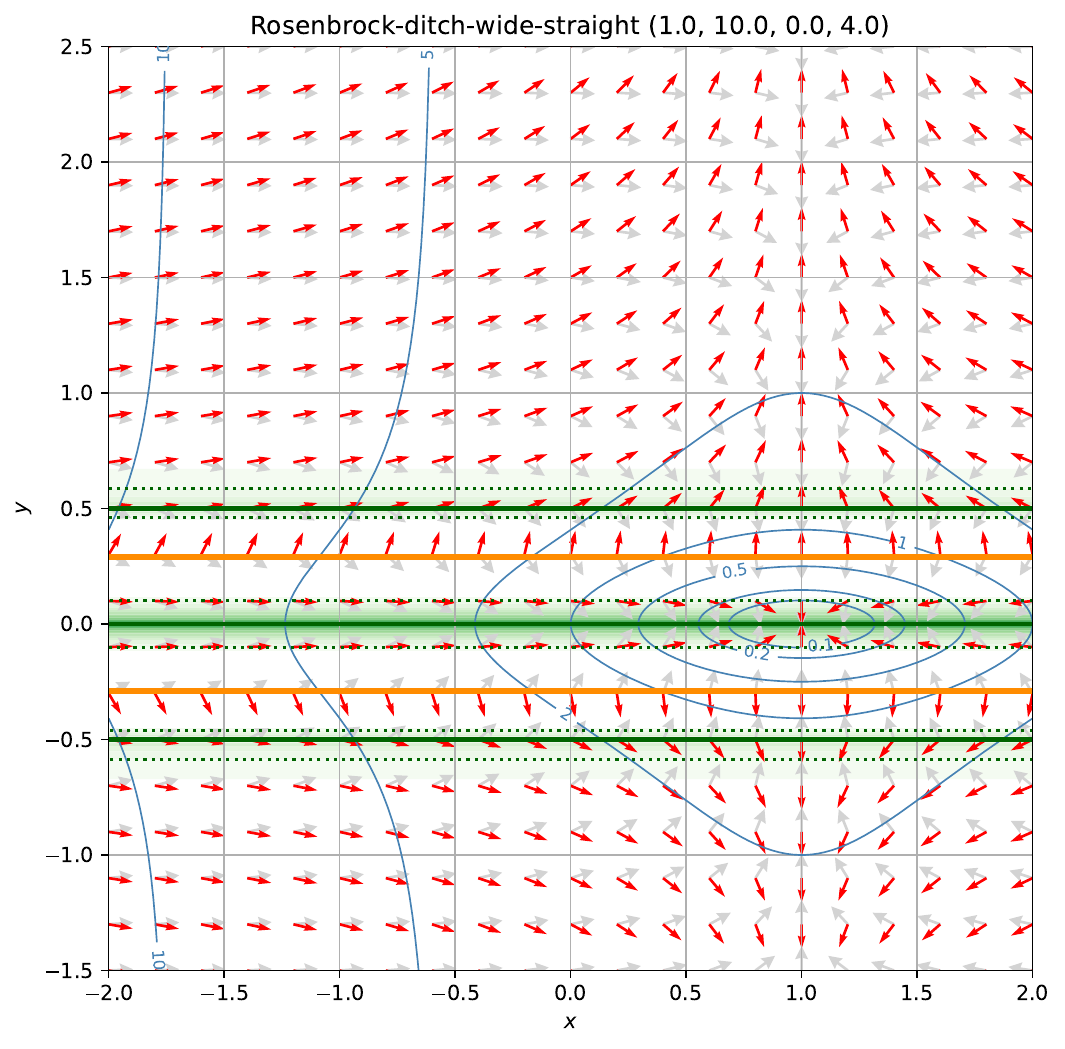}
\caption{}
\label{fig_crit_ditch_straight}
\end{center}
\end{figure}
\textbf{Himmelblau's function} (see section~\ref{sec_himmelblau}) was
included in the tests as it provides multiple stationary points of
different types, albeit without obvious valleys in the objective
function. The results are shown in \textbf{Figure~\ref{fig_crit_hmlbl}}.
We see that the Hessian singularity completely cuts off the four minima
and the maximum (assuming that the curves leaving the diagram run to
infinity). Interestingly, the four saddle points in between are not
perfectly separated from each other; instead there are narrow
``straits'' connecting these stationary points.\footnote{I assume that
in a 2D function all saddle points are connected by straits since they
all have the same set of eigenvalue signs: one value negative, the other
positive. In higher-dimensional functions, the ``inertia'' of the
Hessian, i.e. the tuple with the number of positive, negative, and zero
eigenvalues, can differ for different saddle points. In this case,
saddle points with different inertia are probably separated by a
singularity of the Hessian. It is, however, not clear to me, whether the
eigenvalues alone are sufficient to describe the separation of saddle
points, or whether the eigenvectors also need to be considered.} The
$\check{\tau}$-ravines run through all stationary points. In contrast to
Rosenbrock's function, the Newton vectors in the ravine are not exactly
pointing along the ravine's bottom; however, their overall direction
would still lead the method to the stationary points.
\begin{figure}[t]
\begin{center}
\includegraphics[width=12cm]{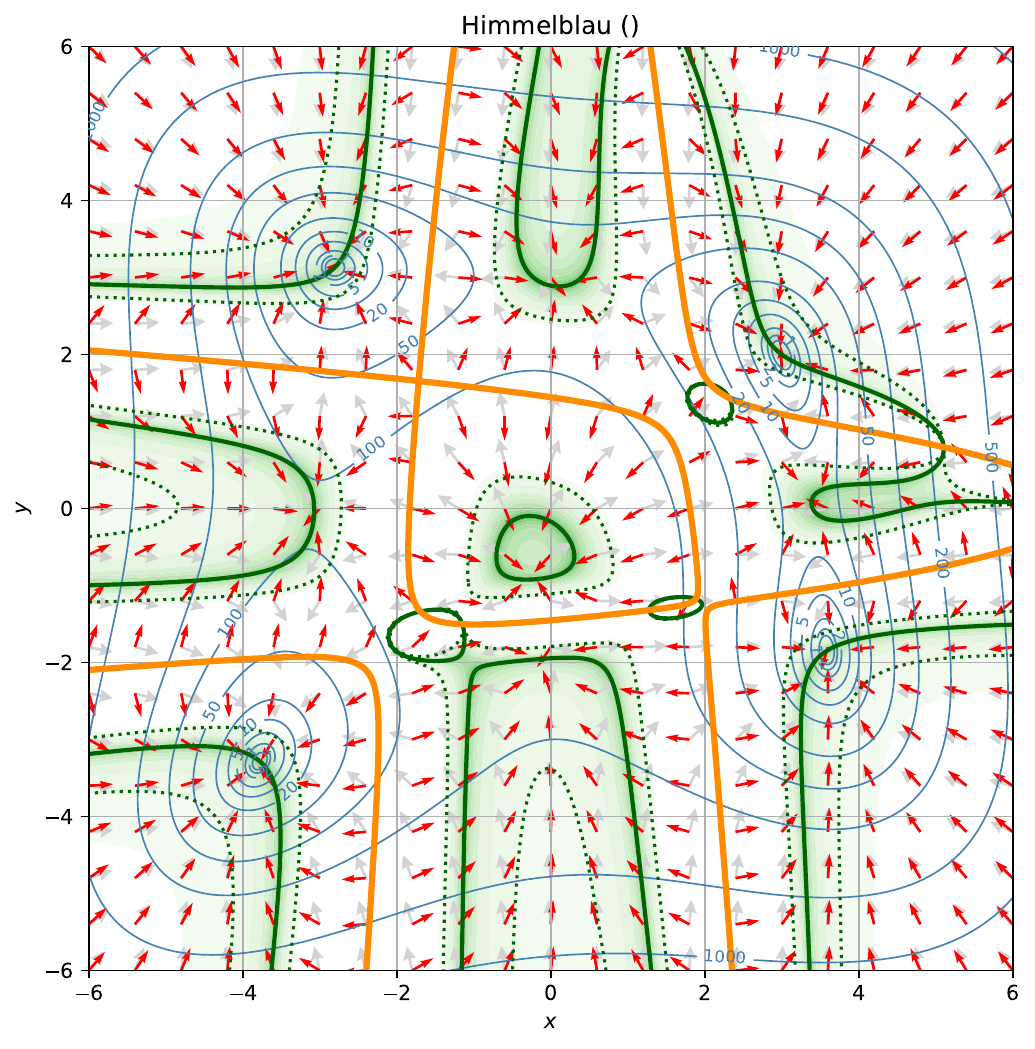}
\caption{}
\label{fig_crit_hmlbl}
\end{center}
\end{figure}
Potentially problematic for the method suggested in this work are the
three small, closed zero-crossing curves near the straits connecting the
saddle point regions. As these appear to be unrelated to any stationary
point, they may capture the method and prevent a convergence towards a
stationary point. It is visible from the missing color contours around
these three closed loops that these ravines are also steep, in contrast
the other ravines in the diagram.

The maximum in the center is also part of a small, closed ravine. This
may only be problematic if the method arrives in the looped ravine on
the opposite side from the maximum, as it could slow down the progress
(a special parallelity check in the line-search strategy described below
may prevent the method from getting stuck).

A plot for the \textbf{H\'enon-Heiles} system (see
section~\ref{sec_henonheiles}) is shown in
\textbf{Figure~\ref{fig_crit_hnnhls_all}}. The minimum at $(0,0)$ is
separated from the rest of the plane by a Hessian singularity. A closed
zero-crossing curve of $\tau$ runs through the three saddle points and
touches the singularity at three places (with a narrowing
$\check{\tau}$-ravine as described above). The Newton vectors would lead
a method that is traveling through the ravine to one of the saddle
points (unless is starts at the points where the ravine touches the
singularity). However, the minimum point is not connected to the ravine
network, but an isolated minimum in $\check{\tau}$; this is a property
not seen in the previous examples.
\begin{figure}[t]
\begin{center}
\includegraphics[width=12cm]{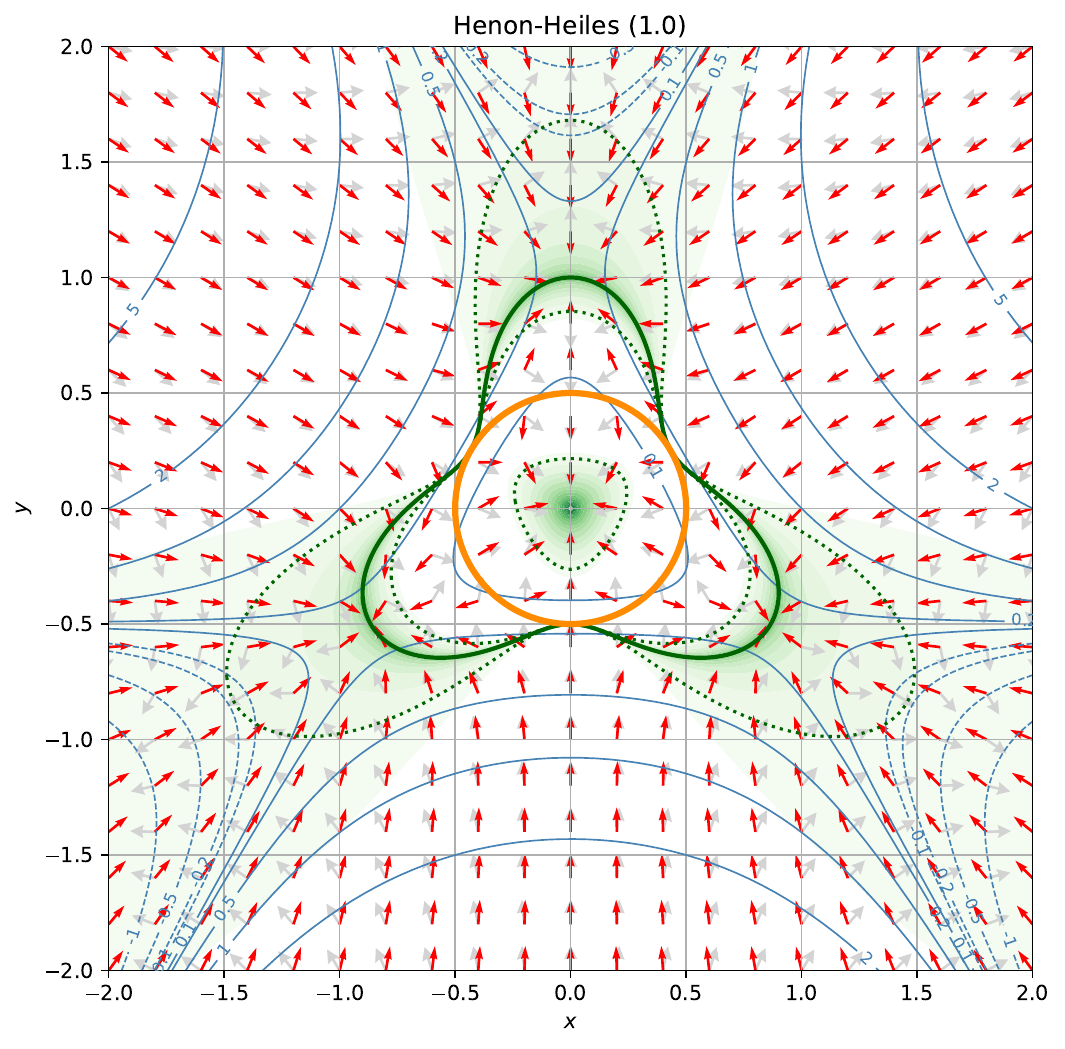}
\caption{}
\label{fig_crit_hnnhls_all}
\end{center}
\end{figure}
We now continue with four considerably more complex functions.

\textbf{Junction1}: This function, shown in
\textbf{Figure~\ref{fig_crit_junction1_all}}, is constructed from two
crossing ditches superimposed on a quadratic function (see
section~\ref{sec_junction}). This function was designed to explore
whether $\check{\tau}$-ravines can cross in a minimum point.

We see from \textbf{Figure~\ref{fig_crit_junction1_all}} that the number
of Hessian singularies and $\check{\tau}$-ravines is quite large. At
least three singularity curves run along the valleys, as do up to five
$\check{\tau}$-ravines. However, the $\check{\tau}$-ravines don't cross
in the minimum point at $(0,0)$. Instead, as visible in
\textbf{Figure~\ref{fig_crit_junction1_aux}} (left), the
$\check{\tau}$-ravines ``shy away'' from the minimum, and the minimum of
the objective function corresponds to an isolated minimum in
$\check{\tau}$ (surrounded by four small, looped zero-crossings of
$\tau$). At least in the arms of the vertical ditch, there are
zero-crossings of $\tau$ close to the valley bottom, but they are
folding back and peter out before they reach the minimum; see
\textbf{Figure~\ref{fig_crit_junction1_aux}} (right).

Overall, this complex situation obviously doesn't offer itself to an
application of the line-search strategy suggested in this work. The
reasons are similar to the ones discussed for the Rosenbrock ditch
function.
\begin{figure}[tp]
\begin{center}
\includegraphics[width=12cm]{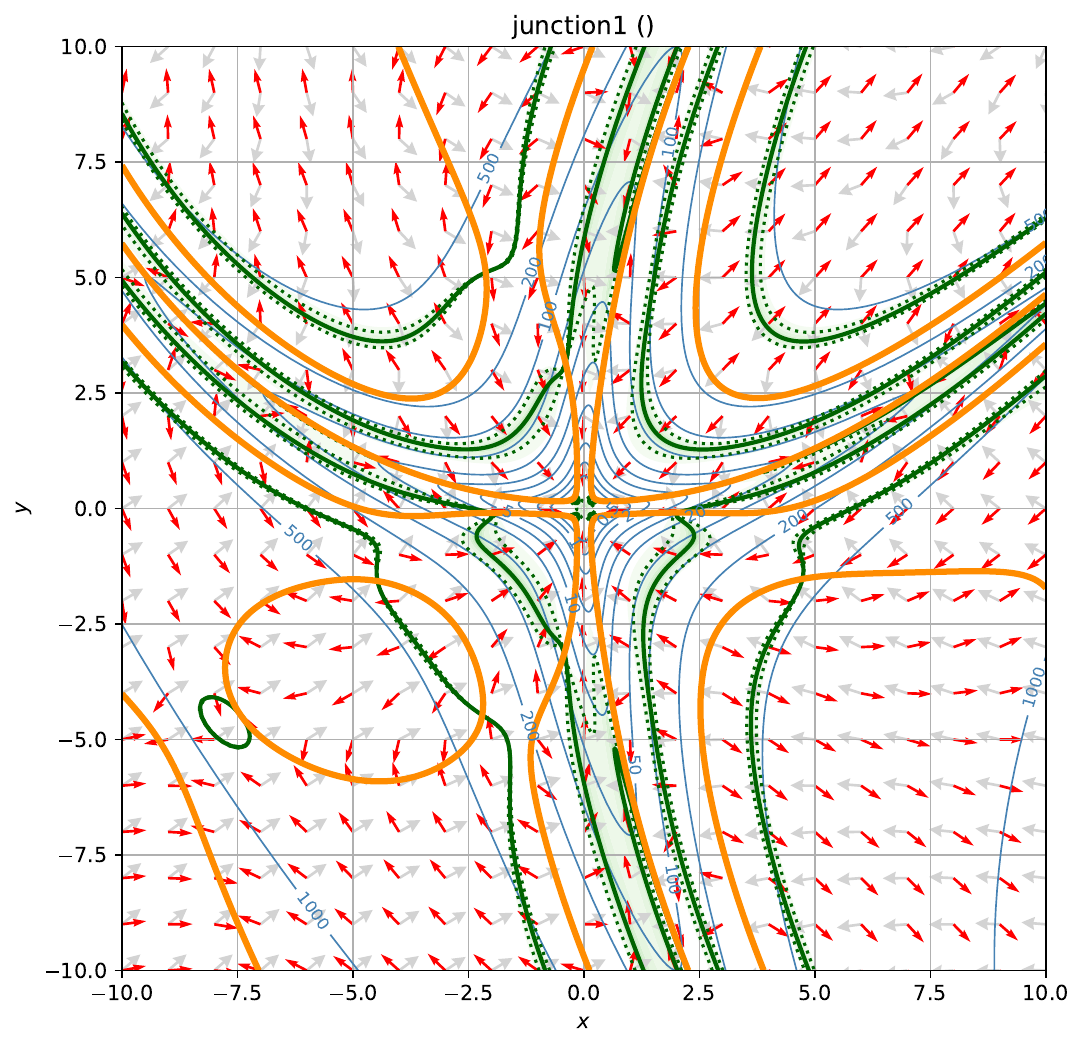}
\caption{}
\label{fig_crit_junction1_all}
\end{center}
\end{figure}
\begin{figure}[tp]
\begin{center}
\includegraphics[height=0.47\textwidth]{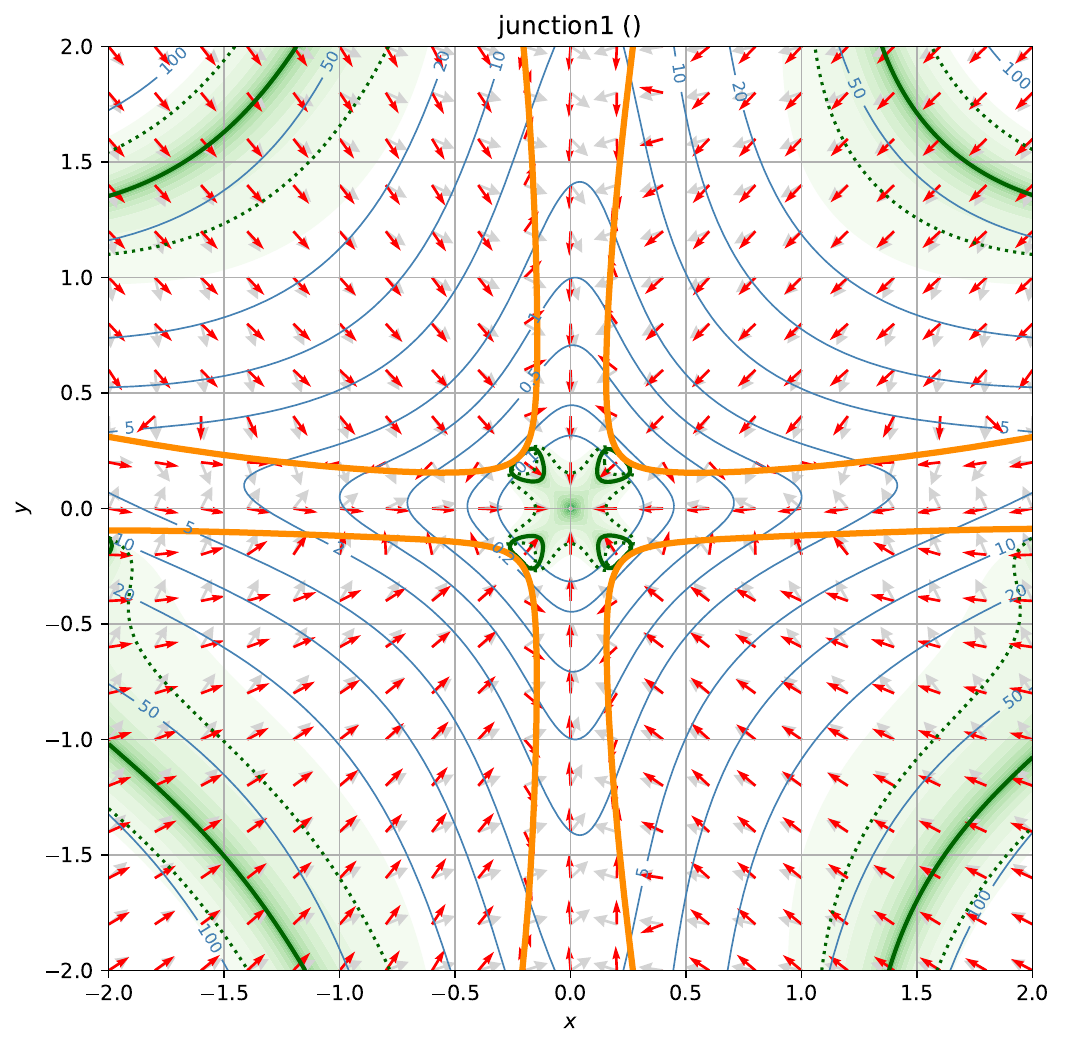}
\hspace*{1mm}
\includegraphics[height=0.47\textwidth]{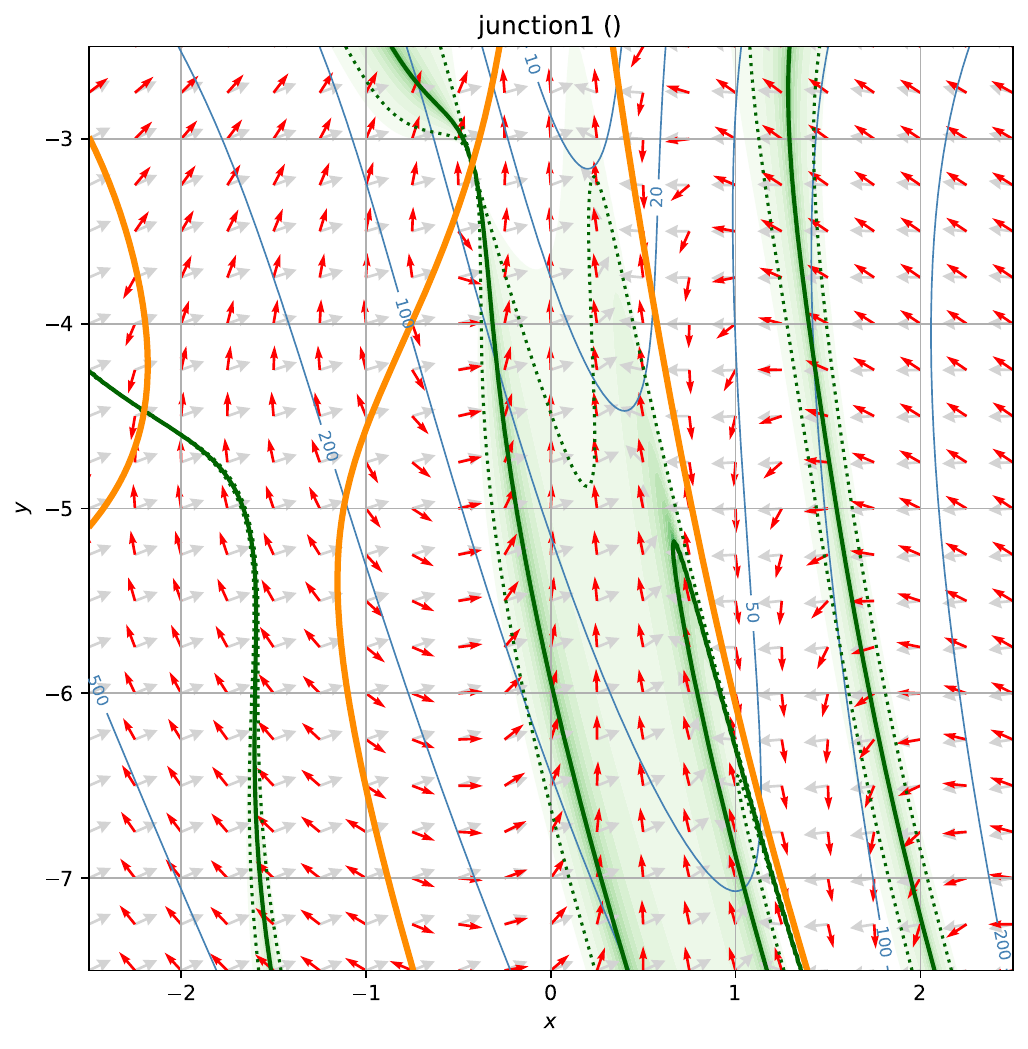}
\caption{}
\label{fig_crit_junction1_aux}
\end{center}
\end{figure}
\textbf{Junction2}: This objective function, shown in
\textbf{Figure~\ref{fig_crit_junction2_all}}, has one straight and one
bent ditch. Still, there is an isolated $\check{\tau}$-minimum at the
stationary point $(0,0)$. The overall situation is quite complex as for
junction1.
\begin{figure}[tp]
\begin{center}
\includegraphics[width=12cm]{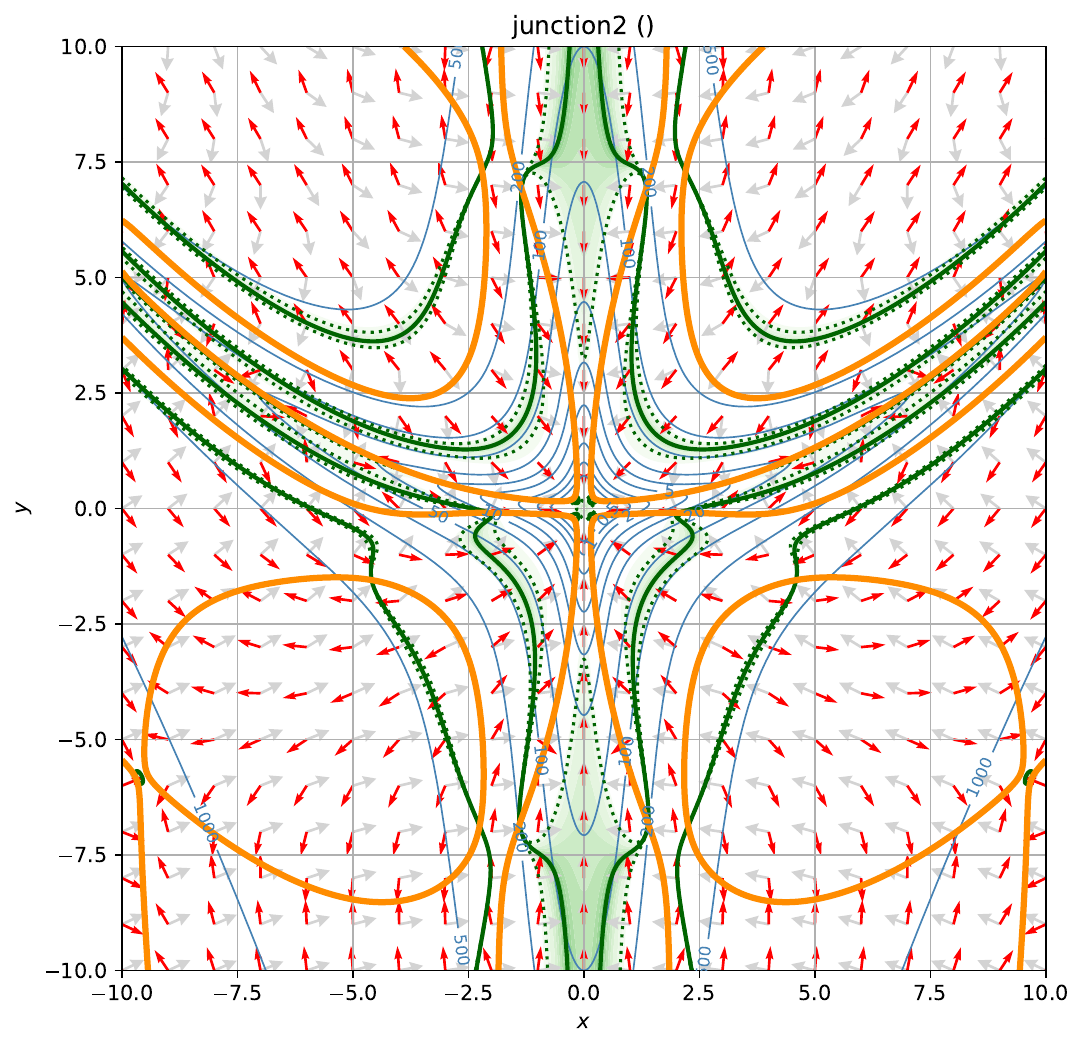}
\caption{}
\label{fig_crit_junction2_all}
\end{center}
\end{figure}
\begin{figure}[tp]
\begin{center}
\includegraphics[height=0.47\textwidth]{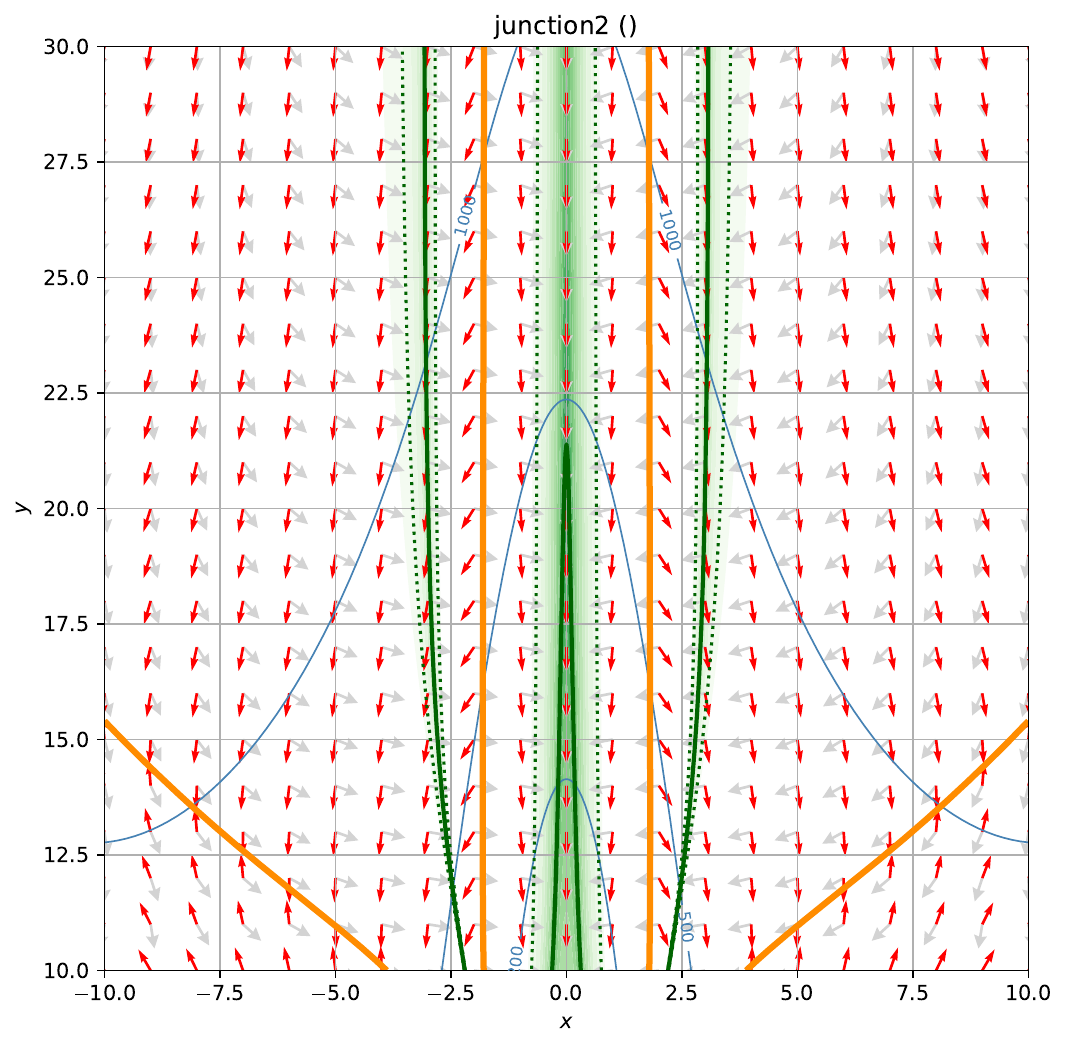}
\hspace*{1mm}
\includegraphics[height=0.47\textwidth]{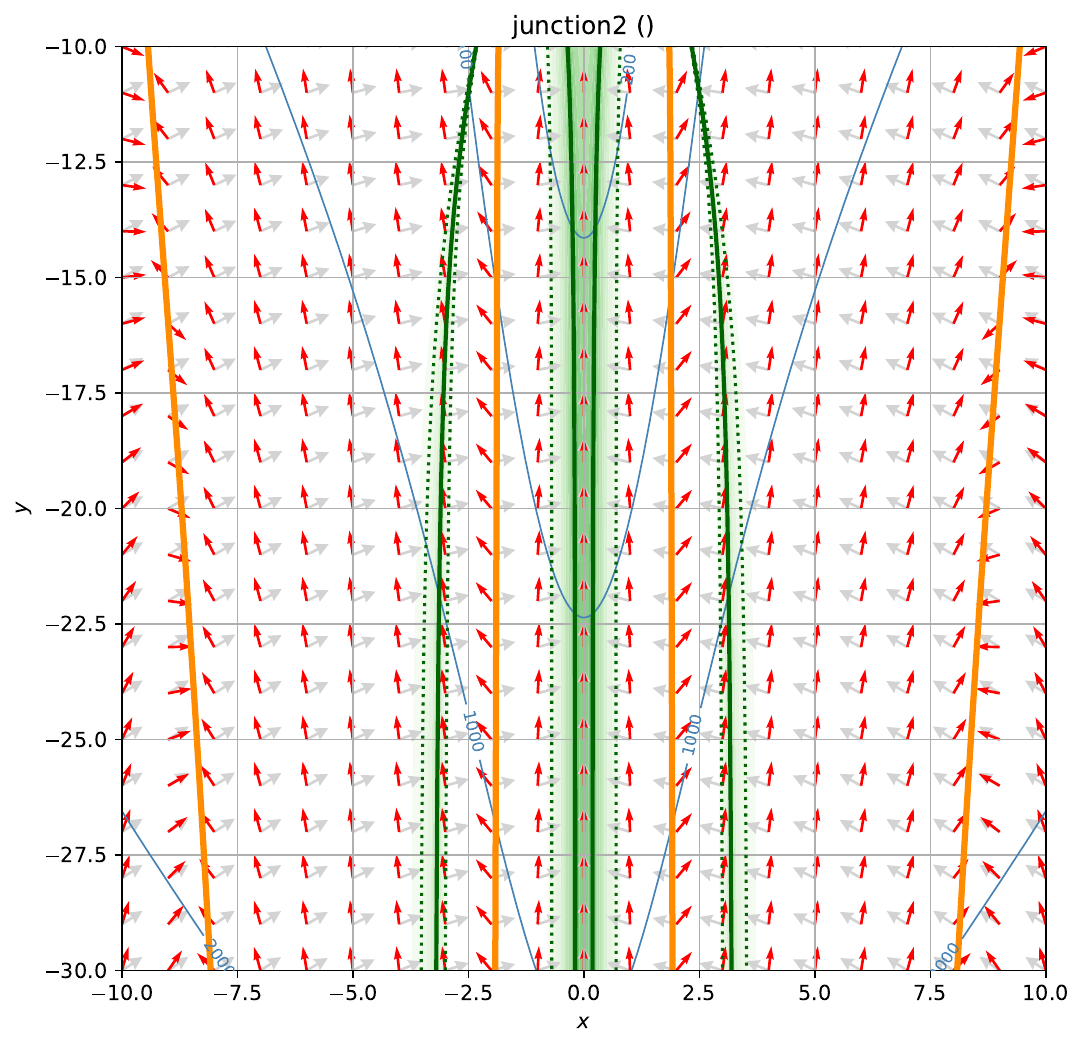}
\caption{}
\label{fig_crit_junction2_aux}
\end{center}
\end{figure}
\textbf{Figure~\ref{fig_crit_junction2_aux}} shows how the straight
vertical valley continues upwards (left) and downwards (right) from the
range in \textbf{Figure~\ref{fig_crit_junction2_all}}. It is not clear
to me (particulary when compared with the Rosenbrock-ditch-wide-straight
function) why there is a {\em double} $\check{\tau}$-ravine running at
the valley bottom. It is also not clear why this double ravine peters
out into a single ravine at the upper continuation (left), but not at
the lower one (right). Also, the two singularities run outwards at the
upper continuation, but stay in the vicinity of the valley in the lower
one.

\textbf{Goldstein-Price}: This function
(section~\ref{sec_gldstnprc_beale}) also exhibits two crossing valleys,
in this case running in diagonal directions; both valleys appear to be
straight. \textbf{Figure~\ref{fig_crit_gldstnprc}} shows this function
on the left side, with a closeup around the minimum at $(0,-1)$ on the
right. The overall situation, albeit still being complex, markedly
differs from the junction functions, though. The minimum at $(0,-1)$ is
not isolated, but contained in a $\check{\tau}$-ravine. There are no
$\check{\tau}$-ravines running at they bottom of the valleys; instead,
some looped ravines appear. This is probably related to a different
cross section of the valleys compared to the junction functions, and to
the fact that the valleys are straight, but this has not been explored
further so far. Unfortunately, we have to conclude that the relation
between $\check{\tau}$-ravines and valley bottoms in the objective
function doesn't seem to be stable.
\begin{figure}[tp]
\begin{center}
\includegraphics[height=0.47\textwidth]{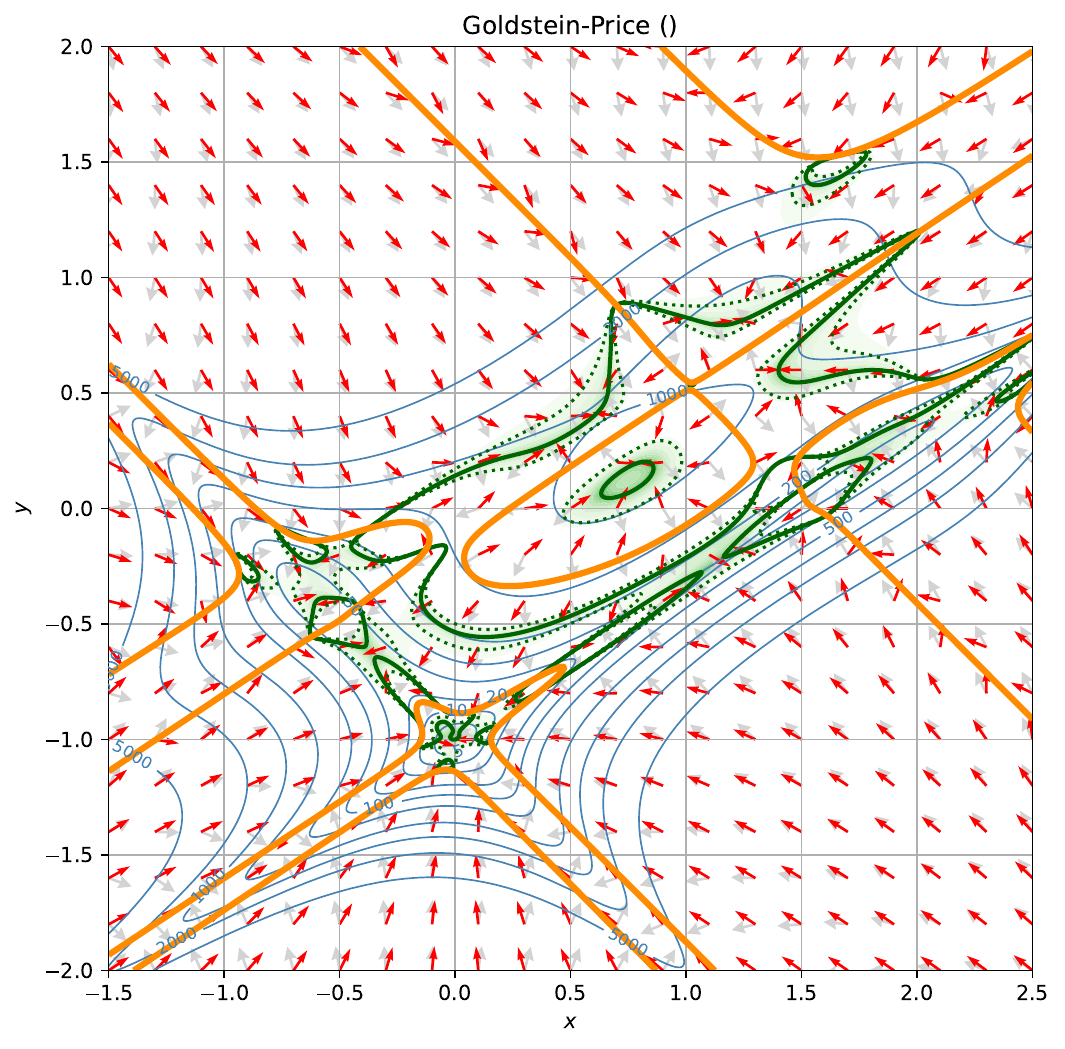}
\hspace*{1mm}
\includegraphics[height=0.47\textwidth]{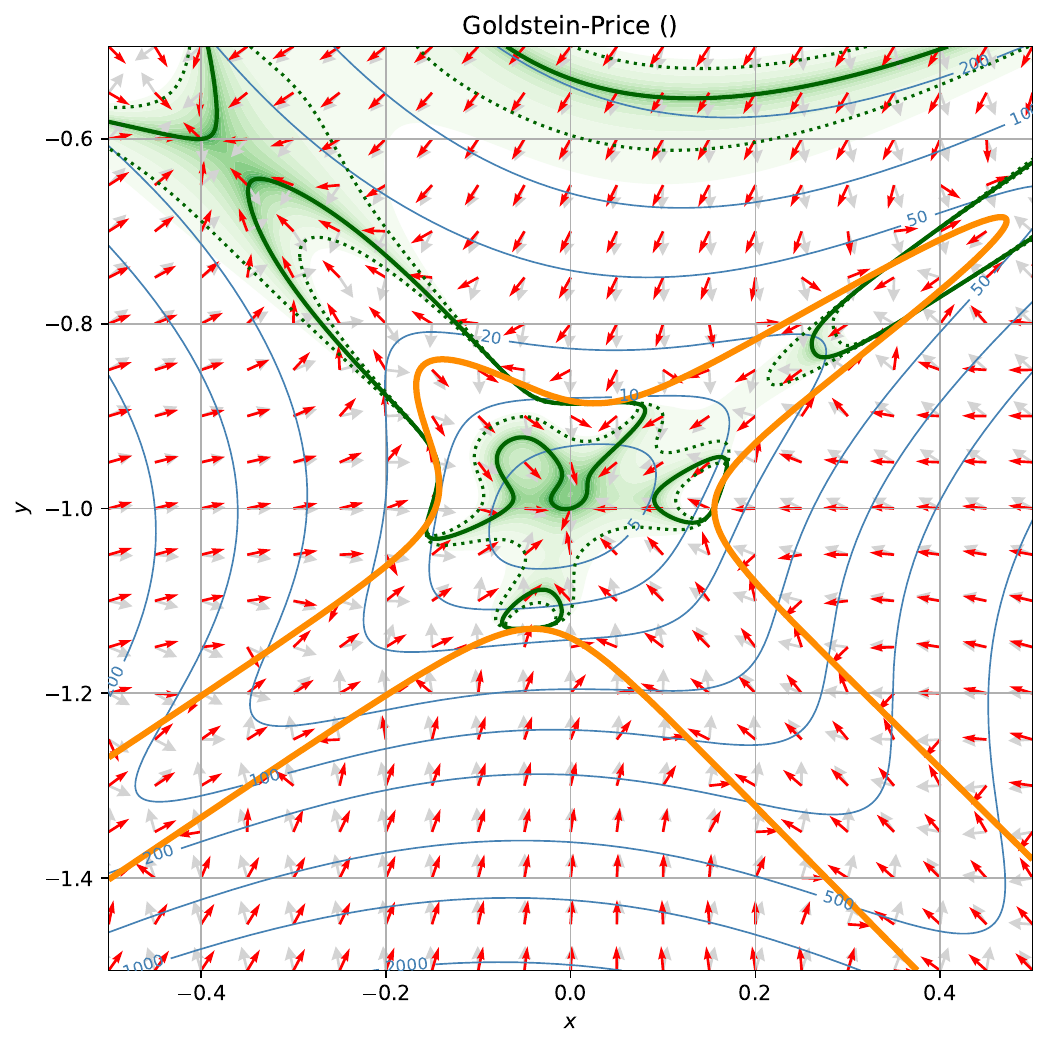}
\caption{}
\label{fig_crit_gldstnprc}
\end{center}
\end{figure}
\begin{figure}[tp]
\begin{center}
\includegraphics[height=0.47\textwidth]{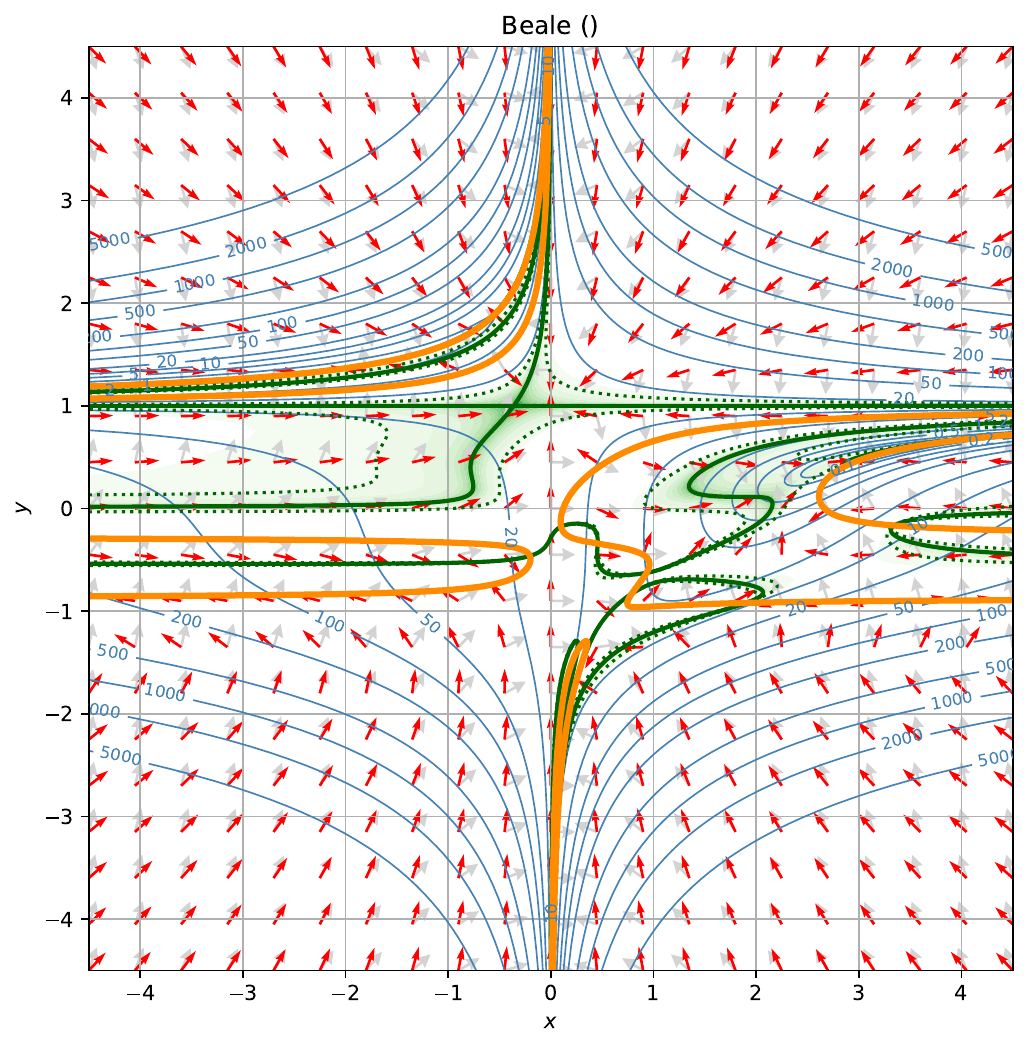}
\hspace*{1mm}
\includegraphics[height=0.47\textwidth]{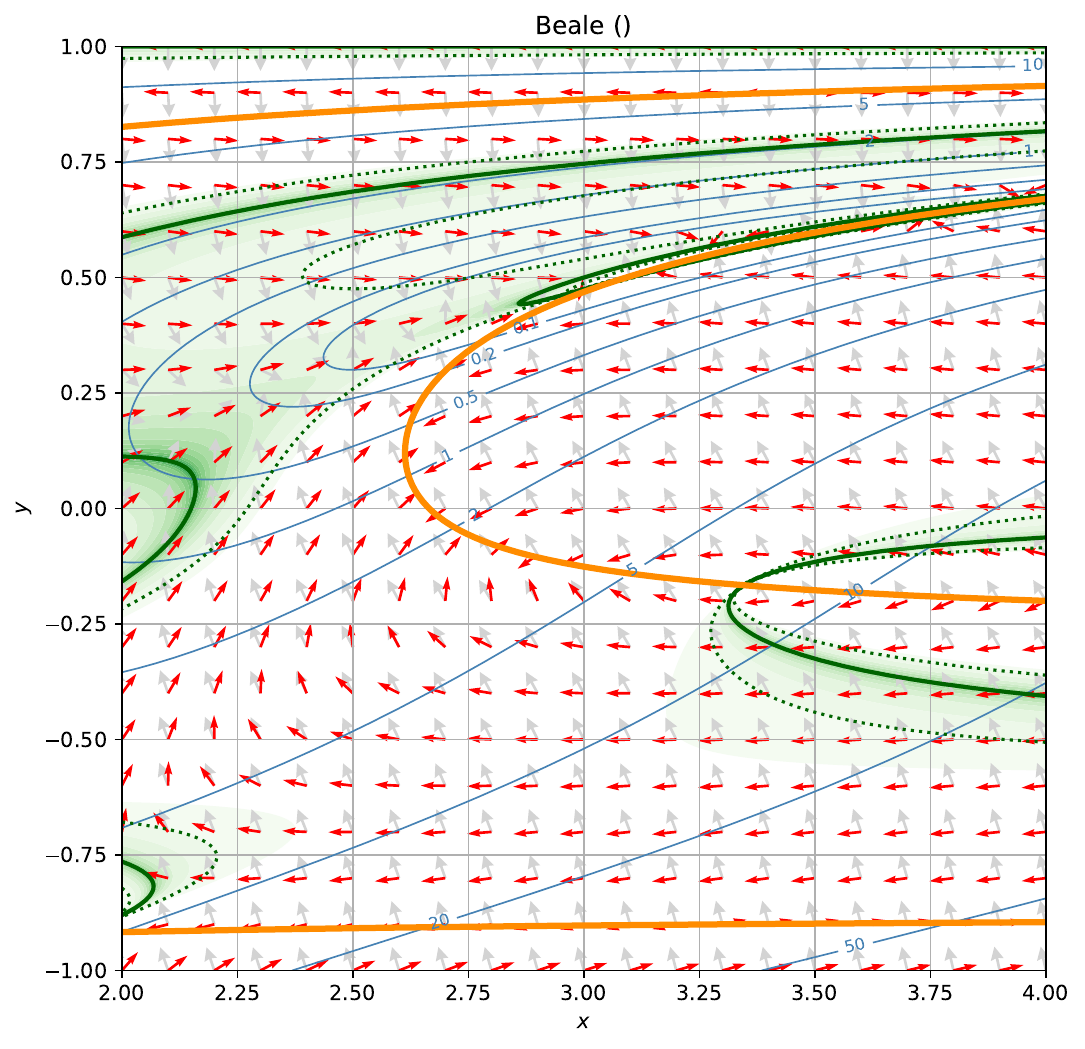}
\caption{}
\label{fig_crit_beale}
\end{center}
\end{figure}
\textbf{Beale}: Also the Beale function
(section~\ref{sec_gldstnprc_beale}) shown in
\textbf{Figure~\ref{fig_crit_beale}} (left) exhibits a complex situation
with respect to singularities and $\check{\tau}$-ravines. Here the
minimum at $(3,0.5)$ is not at the crossing of the two valleys (running
vertically and horizontally through $(0,0)$), but at another valley-like
structure (magnified in the right diagram). There we have a double,
looped $\check{\tau}$-ravine including the stationary point. It is also
interesting that there is a crossing point of $\check{\tau}$-ravines
(left diagram, at around $(0.5,1)$), though a stationary point seems to
lie a little distance away from the crossing at $(0,1)$. This feature
didn't appear in the previous plots.

\subsection{Application of Newton's Method and Line Search}\label{sec_twod_newton_ls}
%------------------------------------------------------------------------------------

%
In the following, we apply different versions of Newton's method to the
2D test functions studied above.

\subsubsection{Nomenclature of Methods}\label{sec_twod_newton_ls_nom}
%''''''''''''''''''''''''''''''''''''''''''''''''''''''''''''''''''''

%
We use a nomenclature of the different methods comprising three
components (joined by a hyphen): letter \texttt{S} followed by the
line-search strategy, letter \texttt{M} followed by the line-search
method, and letter \texttt{C} followed by the line-search criterion. The
distinction between line-search strategy and line-search method had to
be introduced for the zigzag method suggested in this paper, since
traversing the $\check{\tau}$-ravines requires a combination of multiple
line-search methods in a relatively complex, overarching line-search
strategy.

\subsubsection{Methods}\label{sec_twod_newton_ls_meth}
%'''''''''''''''''''''''''''''''''''''''''''''''''''''

%
The implementation of Newton's method computes Hessian and gradient for
the current position, determines the Newton step, and for a non-zero
Newton step invokes the selected line-search strategy (otherwise it
stops). If the Euclidean distance between the last point and the new
point determined by the line-search strategy falls below a threshold
($10^{-5}$) or if more than $100$ steps were performed, the method
terminates.

The method \texttt{Sno-Mno-Cval2} is the plain Newton method without any
line search (no strategy, no method). The criterion \texttt{Cval2}
(relating to the value of the objective function) is only specified in
the program for visualization purposes.

The method \texttt{Sno-Mex-Cval2} is a straight-forward explicit line
search on the value of the objective function (criterion
\texttt{Cval2}). The method can therefore only successfully be applied
to approach minima, but not saddle points. The line-search strategy
\texttt{Sno} falls through to the explicit line-search method
\texttt{Mex}. Explicit line search splits the $\alpha$ range $[0,1]$
into $100$ steps and determines the minimum of the criterion for the
corresponding points.\footnote{Note that $\alpha = 0$ is included. This
may prevent any progress in some cases.}

The methods \texttt{Szz-Mlm-Ctau} and \texttt{Szzp-Mlm-Ctau} are the
ones introduced in this work. They apply the zigzag strategy outlined in
section~\ref{sec_zigzag}. The two strategies \texttt{Szz} and
\texttt{Szzp} just differ with respect to whether the parallelity check
(see below) is invoked (indicated by \texttt{p}) or not (the
implementation is the same, the parallelity check is just deactivated).
In the following, the line-search strategy (\texttt{Szz}/\texttt{Szzp})
and the line-search method (\texttt{Mlm}) are described in detail.

\subsubsection{Implementation}\label{sec_twod_newton_ls_impl}
%''''''''''''''''''''''''''''''''''''''''''''''''''''''''''''

%
\textbf{Zigzag strategy:} The zigzag strategy
(\texttt{Szz}/\texttt{Szzp}, overview in section~\ref{sec_zigzag}) has
two thresholds with respect to the criterion: The first threshold
($10^{-3}$) defines if the bottom of a $\check{\tau}$-ravine has been
reached, the second threshold ($10^{-1}$) defines the transition point
between the zig and the zag phase. We refer to these parameters as
``entry threshold'' and ``escape threshold'', respectively.

At the beginning, the zigzag strategy determines the current value of
the criterion ($\check{\tau}$ from \eqref{eq_tauv}, in the nomenclature
\texttt{Ctau}), and then proceeds in two different cases:

{\em Case 1:} If the current value of the criterion is larger than the
entry threshold, we are not at the bottom of a $\check{\tau}$-ravine,
and therefore a specific line-search method (\texttt{Clm}, see below) is
invoked. This is referred to as the \textbf{down phase}, as we attempt
to move down into a $\check{\tau}$-ravine. If the criterion at the
resulting point is below the entry threshold, we have reached the bottom
of a ravine and the step is accepted; otherwise a full Newton step is
performed (in this way the $\check{\tau}$-ravine may be found in a
sequence of multiple Newton steps).

{\em Case 2:} If the current value of the criterion is smaller than the
entry threshold (or equal), we already are at the bottom of a
$\check{\tau}$-ravine. We therefore start the \textbf{zig phase},
proceeding along the ravine. An explicit line search along the Newton
step is performed as described above, but instead of returning the
minimum, it determines the point at the first $\alpha$ value in the
range $[0,1]$ where the criterion exceeds the escape threshold. If all
samples of the criterion stay below the escape threshold (or are equal),
the point corresponding to $\alpha=1$ is returned, ensuring maximum
progress.

Case 2 continues with the zag phase where we try to return to the bottom
of the ravine from the escape point. We first determine the pullback
direction (see section~\ref{sec_pullback}) at the escape point. Then the
\textbf{parallelity check} is applied. If the angle\footnote{Note that
the sign of the pullback direction is undefined; this is considered in
the computation of the angle.} between the zig and the zag (pullback)
direction is smaller than a threshold ($0.2\,\mbox{rad}$, approximately
$11\,\mbox{deg}$), there is a danger of protracted oscillations in the
ravine without making noticeable progress. Therefore, we perform a full
Newton step starting at the escape point of the zag phase.

Case 2 proceeds with the \textbf{zag phase} if the parallelity check
failed. The zag phase is implemented as a Golden Section search starting
with a tight bracket ($[-10^{-5}, 10^{-5}]$), maximal $100$ steps, and a
tolerance of $10^{-3}$. If that fails, we stay at the escape point,
otherwise the point found by the Golden Section search is the next
point.

Note that the zigzag strategy is ``state-free'', so in the iteration by
Newton's method, each new step checks the criterion at the current point
and then decides how to proceed, without considering what type of action
was taken in the last step. However, the zig and zag phases are part of
the same Newton step.

\textbf{Line-search method for zigzag strategy:} The line-search method
(nomenclature \texttt{Mlm} for ``local minimum'') is used in case 1 of
the zigzag strategy (\textbf{down phase}) to reach the bottom of a
$\check{\tau}$-ravine. This line-search method is specifically tuned to
the zigzag problem, particularly to the requirement of finding an often
tight valley in the $\check{\tau}$ criterion. The method first starts
with an explicit line-search with $100$ steps.\footnote{As the value of
$\check{\tau}$ can vary drastically over the search range, particularly
if Hessian singularies are passed, this appears to be the only feasible
way.} It determines all local minima, including a check for the last
point covered by the explicit search (which only has a single neighbor),
but excluding the first point.

Then a Golden Section search is performed at all local minima to refine
the result of the explicit search by trying to descend deeper into
possible ravines. The starting point of each refinement search is the
coarse local minimum; the search direction and range is defined by the
Newton step. Results where the absolute value of the refined $\alpha$
exceeds $10^{-1}$ are rejected (as this deviates too much from the
coarse minimum point).\footnote{Here a value of $10^{-2}$ would actually
fit better to the $100$ steps of size $10^{-2}$ of the explicit line
search. The larger value $10^{-1}$ may however be beneficial if the
refinement starts from the last point in the explicit search; in this
case the Newton step could be extended by $10$\%.} Also steps running
into a negative overall $\alpha$ (coarse and fine search combined) are
rejected to prevent backward movements. This results in a (possibly
reduced) list of refined local minima. Finally the first refined local
minimum (that with the smallest $\alpha$) smaller than the entry
threshold is accepted --- this leads to the bottom of the nearest
$\check{\tau}$-ravine. If no sub-threshold local minimum exists, the
deepest refined local minimum is returned. If all refinement steps fail,
the best result of the explicit search (the absolute minimum) is
returned.

Note that the descriptions above omit the handling of cases where
computations fail, e.g. due to a singular Hessian (please consult the
code for details).

\textbf{Strategy string}: Since the zigzag strategy
(\texttt{Szz}/\texttt{Szzp}) is relatively complex, each invocation
returns a short identifier which describes the actions taken. These
identifiers are concatenated by the implementation of the Newton method
to a ``strategy string''; this string is shown in some of the diagrams
below. In the \textbf{down phase}, the identifier is either \texttt{D}
(only coarse search) or \texttt{D-} (coarse search followed by fine
search), if the search ended at the bottom of a $\check{\tau}$-ravine.
If not, a full step is taken, identified by \texttt{F}. The \textbf{zig
phase} delivers the identifier \verb@^@ if the escape threshold was
exceeded, otherwise \texttt{A}. If the computation of the pullback
direction failed, the zig phase returns identifier \texttt{U} instead.
If the parallelity check was raised, the identifier \texttt{P} is
returned instead. If the \textbf{zag phase} was successful, identifier
\texttt{v} is appended to the zig identifier. So, for a successful
zigzag movement where the escape threshold was reached, the combined
identifier would be \verb@^v@, otherwise \texttt{Av}. An example of an
entire strategy string would be \verb@D-^v^v^v^v^vAvAvAvAv@. For the
strategy \texttt{Sno}, the identifier is always \texttt{N} (for ``Newton
step'').

\subsubsection{Plot Description}\label{sec_twod_newton_ls_plot}
%''''''''''''''''''''''''''''''''''''''''''''''''''''''''''''''

%
\textbf{Trajectories}: Trajectories are drawn into the plots described
in section~\ref{sec_twod_sing_crit_plot}, see e.g.
\textbf{Figure~\ref{fig_newton_rsnbrk_no-no}}, top left, for an example.
The version of Newton's method used for generating the plot is added in
a second header line. For the experiments, around $10$ starting points
for Newton's method were chosen manually, attempting to cover all
regions in the plot (sometimes excluding regions where Newton's method
can be expected to diverge in all tested versions). Each trajectory
appears with lines and symbols in an individual color. The starting
point is marked by an open circle, all intermediate points by a filled
circle, and the final point indicates either a failure (cross symbol,
length of final gradient vector above $10^{-5}$), a minimum (downward
triangle), maximum (upward triangle), or saddle point (diamond); note
that for the successful runs, the final symbol for one of the
trajectories usually hides that of all others. The type of the final
point is determined from the signs of the eigenvalues of the final
Hessian. At the bottom right corner of the plot, the number of
successful runs in relation to the total number of runs is displayed
(however, this should not be misinterpreted as a measure that could be
used to compare the different methods, particularly due to the small
number of trajectories with manually chosen starting points).

\textbf{Criterion-over-$\alpha$ plots}: A special type of diagram was
designed to visualize the effect of the strategy steps on the criterion
for a single trajectory, see e.g.
\textbf{Figure~\ref{fig_newton_rsnbrk_no-ex}}, top right. The test
function is indicated in the first header line together with the number
and color of the trajectory in the related trajectory plot. The second
header line specifies the method. In one of the corners of the diagram,
the strategy string described in section~\ref{sec_twod_newton_ls_impl}
is shown (see e.g. \textbf{Figure~\ref{fig_newton_rsnbrk_zzp}}, top
right).

In this diagram, the progress of the method is indicated on the
horizontal axis by {\em concatenating} the $\alpha$ values of each step.
Each section of the curve corresponds to a sub-step of the method, e.g.
in the zigzag strategy, the coarse search of the down phase, the
refinement search of the coarse phase, the zig phase, or the zag phase
are shown as separate sections. Each section is shown in a different
color (though the sequence of colors\footnote{blue, orange, green, red,
purple, brown, pink, gray, olive, cyan} repeats after 10 steps), with
the ``best'' $\alpha$ value returned from the sub-step marked by a black
dot; at this point, the next section starts. If such a sub-step has a
length smaller than $1$ we see that is was damped (shortened) in
comparison to the original Newton step, so the line search had an
effect. The logarithmic vertical axis relates to the criterion used in
the method, as specified in the header (around the final steps, effects
which are probably due to the limits of the floating point
representation may result in ragged or noisy sections).

There are two special considerations for the zag phase in the zigzag
strategy: The step vector used for the Golden Section search is a unit
vector (a pullback direction, see section~\ref{sec_pullback}) multiplied
by the Euclidean length of the Newton step (of the zig phase) in order
to give the same scale to the $\alpha$ values of both zig and zag phase.
Moreover, since the sign of the pullback directions is undefined, the
sign of the $\alpha$ value returned from the zag phase is inverted if it
was negative, so $\alpha$ always increases in these diagrams (the method
itself is not affected).

Entry and escape threshold are visualized as dashed blue lines for
methods where these thresholds are relevant.

\subsubsection{Experiments}\label{sec_twod_newton_ls_exp}
%''''''''''''''''''''''''''''''''''''''''''''''''''''''''

%
\textbf{Rosenbrock-wide}: \textbf{Figure~\ref{fig_newton_rsnbrk_no-no}}
visualizes the results for an application of the \textbf{plain Newton
method}. In the top left diagram it is visible that Newton's method
converges from all chosen starting points, even without line search.
Most trajectories approach the valley in the first step, but then
perform a long jump out before returning into the valley and approaching
the minimum in a few more steps. These outward jumps may be long due to
the vicinity of the singularity (but the jumps into the valley are also
long, so this argument may not hold).

As can be seen in the top right diagram for the black trajectory, the
method converges in a small number of steps. The bottom left diagram
shows the final steps of all trajectories in the vinicity of the minimum
point $(1,1)$. The bottom right diagram visualizes the behavior of the
method around point $(0,0)$; here we also see that the countercurrent
singularity is crossed. Note that one of the trajectories (dark cyan)
starts at $(-10,0)$ and apparently performs a huge jump outside the
figure range before approaching the minimum from below.
\begin{figure}[t]
\begin{center}
\includegraphics[height=0.47\textwidth]{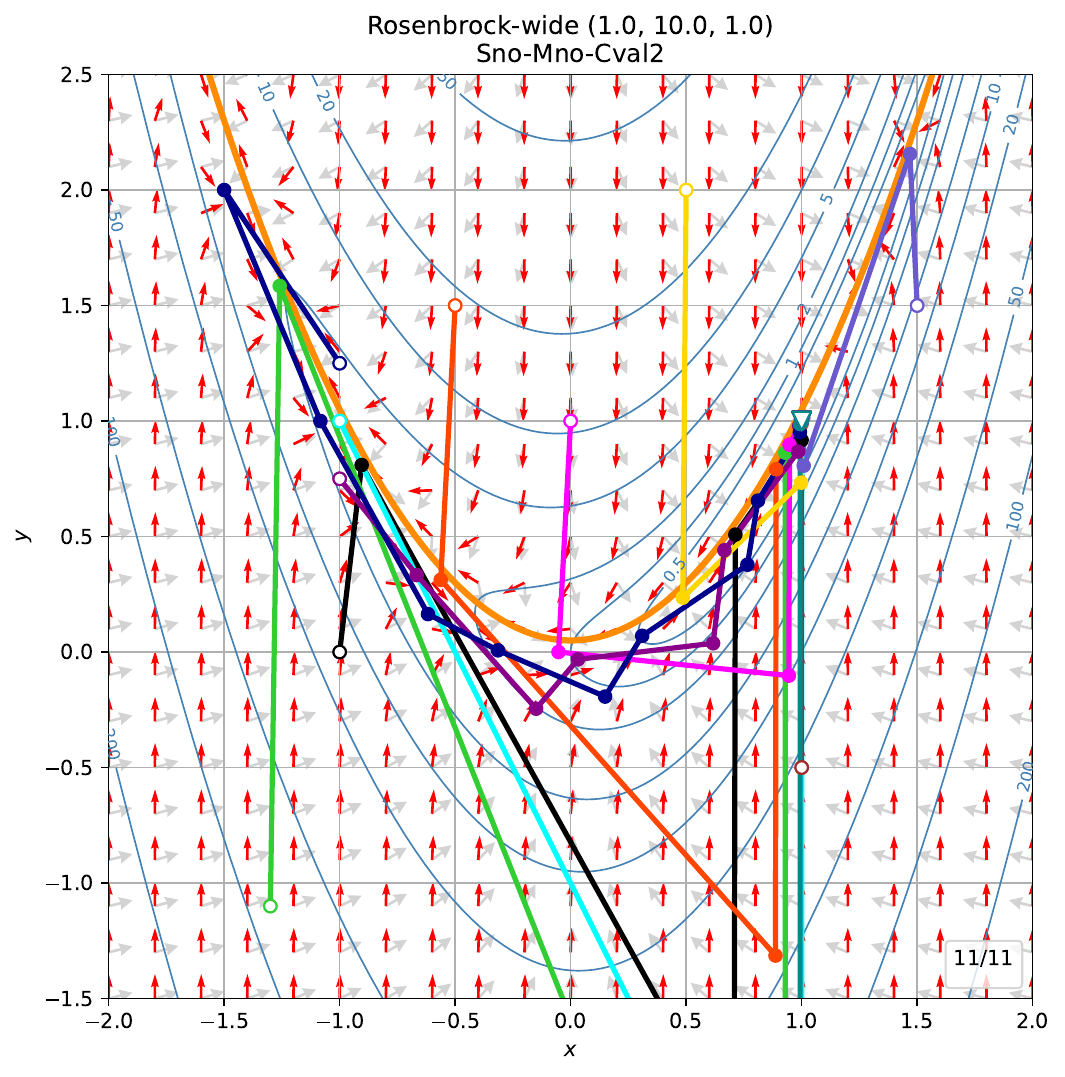}
\hspace*{1mm}
\includegraphics[height=0.47\textwidth]{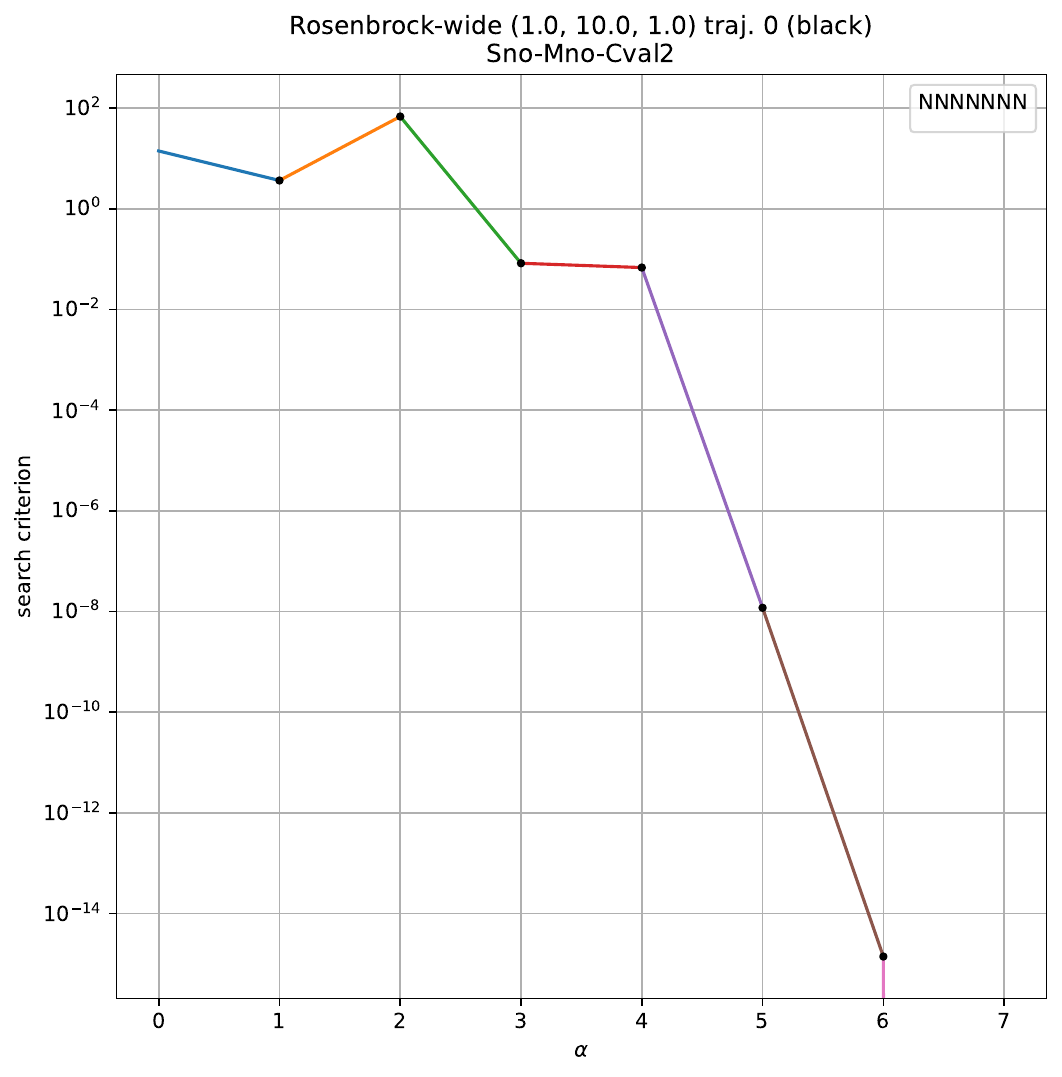}\\
\includegraphics[height=0.47\textwidth]{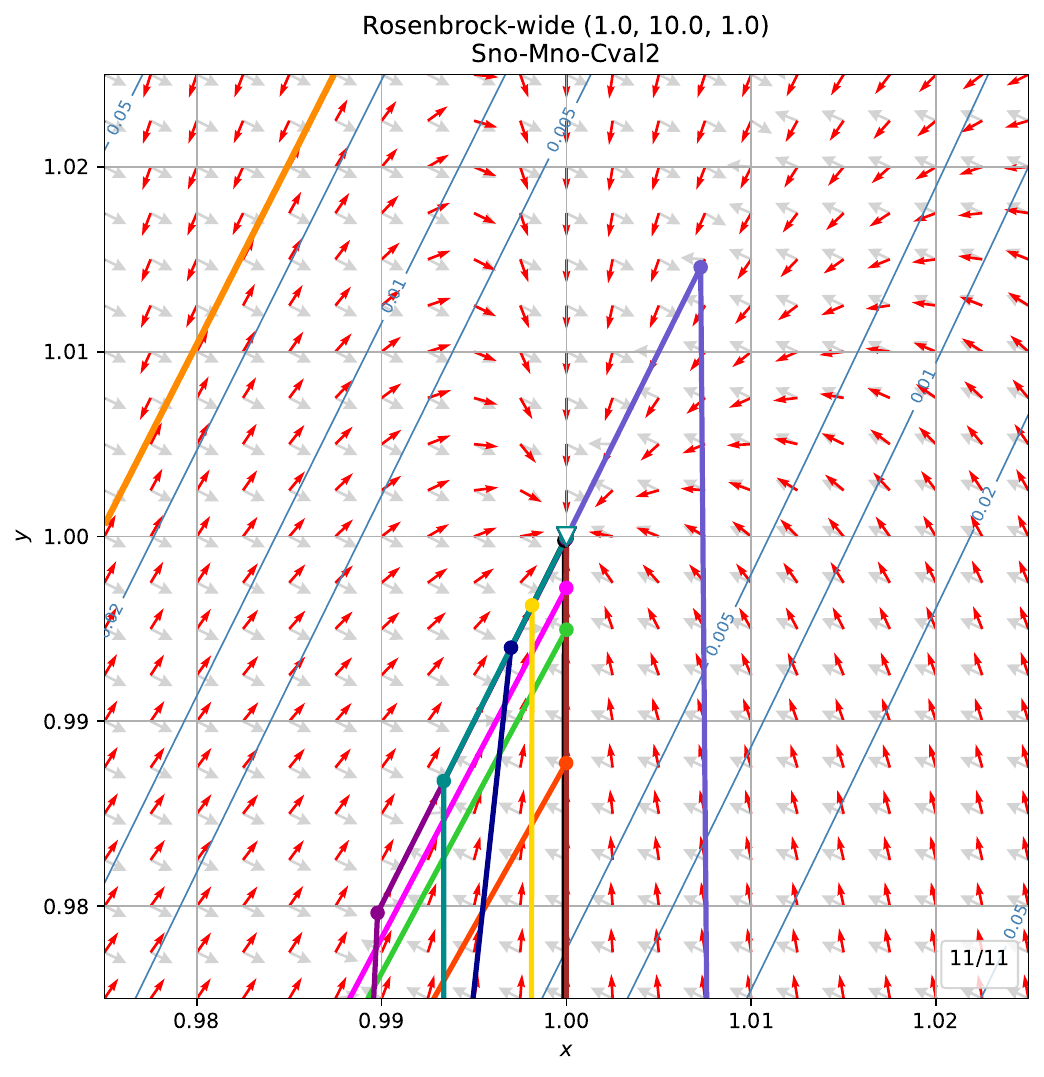}
\hspace*{1mm}
\includegraphics[height=0.47\textwidth]{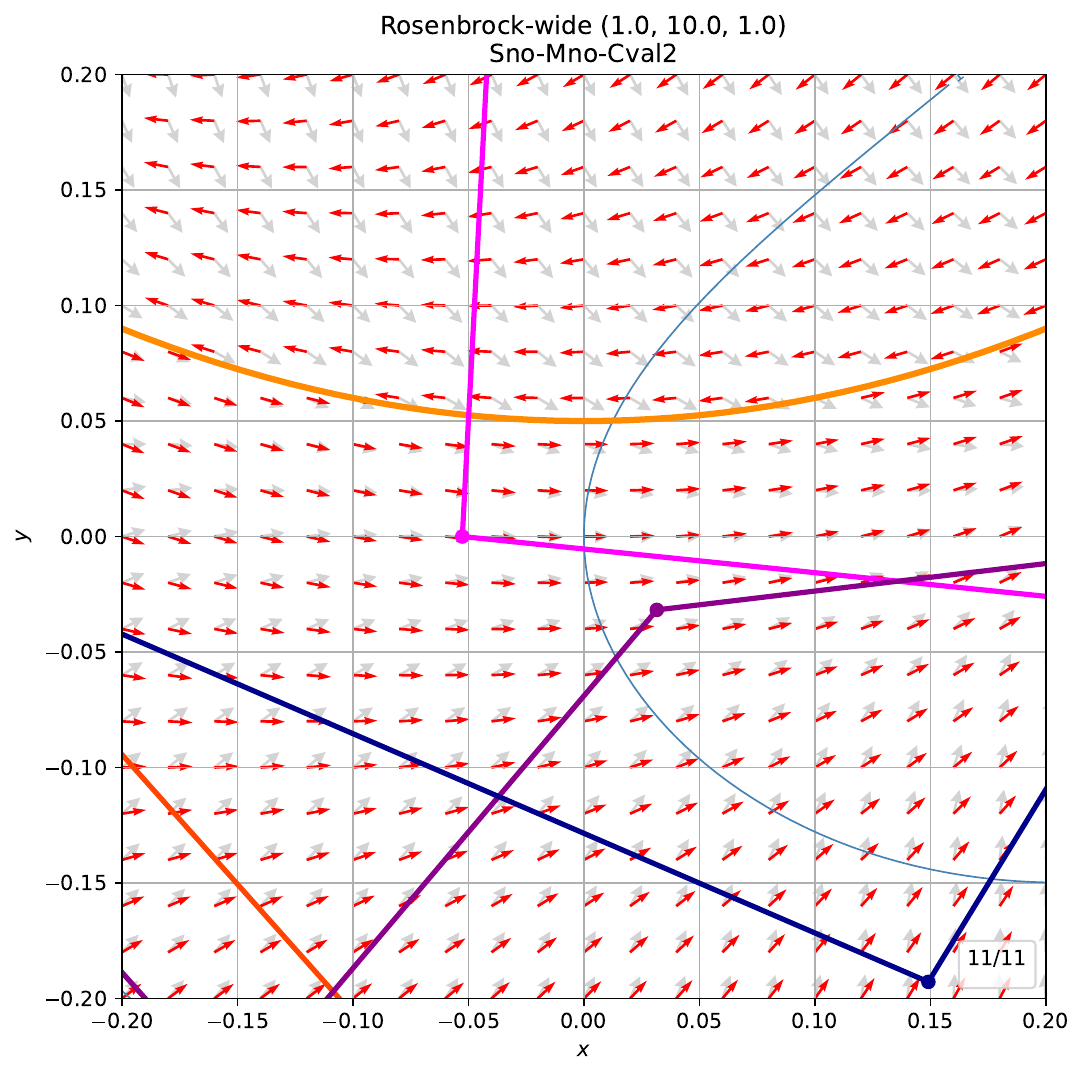}\\      
\caption{}
\label{fig_newton_rsnbrk_no-no}
\end{center}
\end{figure}
\textbf{Figure~\ref{fig_newton_rsnbrk_no-ex}} visualizes the results
when \textbf{Newton's method with a value-based line search} is applied.
We see in the top left diagram that all trajectories (except one) enter
the valley with the first step and then approach the minimum while
staying in the vicinity of the valley bottom. Only the blue trajectory
--- which starts close to the countercurrent singularity --- immediately
gets stuck (indicating that the line search found the best value for
$\alpha = 0$). The dark cyan trajectory starting at $(-10,0)$ has
presumably followed the valley over a long distance before entering the
figure range at the top left.

The diagram at the top right reveals that, for the black trajectory,
there are only few steps where the Newton step is damped. This is
somewhat surprising as the trajectories neatly follow the valley. An
explanation can be extracted from the bottom right diagram: most
trajectories seem to run along one wall of the valley, in some distance
from the valley bottom (which passes through $(0,0)$). At this distance,
the Newton steps seem to have the right length such that they end in a
point from where the next Newton step can also be undamped. In the
vicinity of the minimum, where the valley is relatively straight, most
trajectories run close to the valley bottom (bottom right diagram) and
seem to have the proper length as well (which can be expected close to
the minimum).
\begin{figure}[t]
\begin{center}
\includegraphics[height=0.47\textwidth]{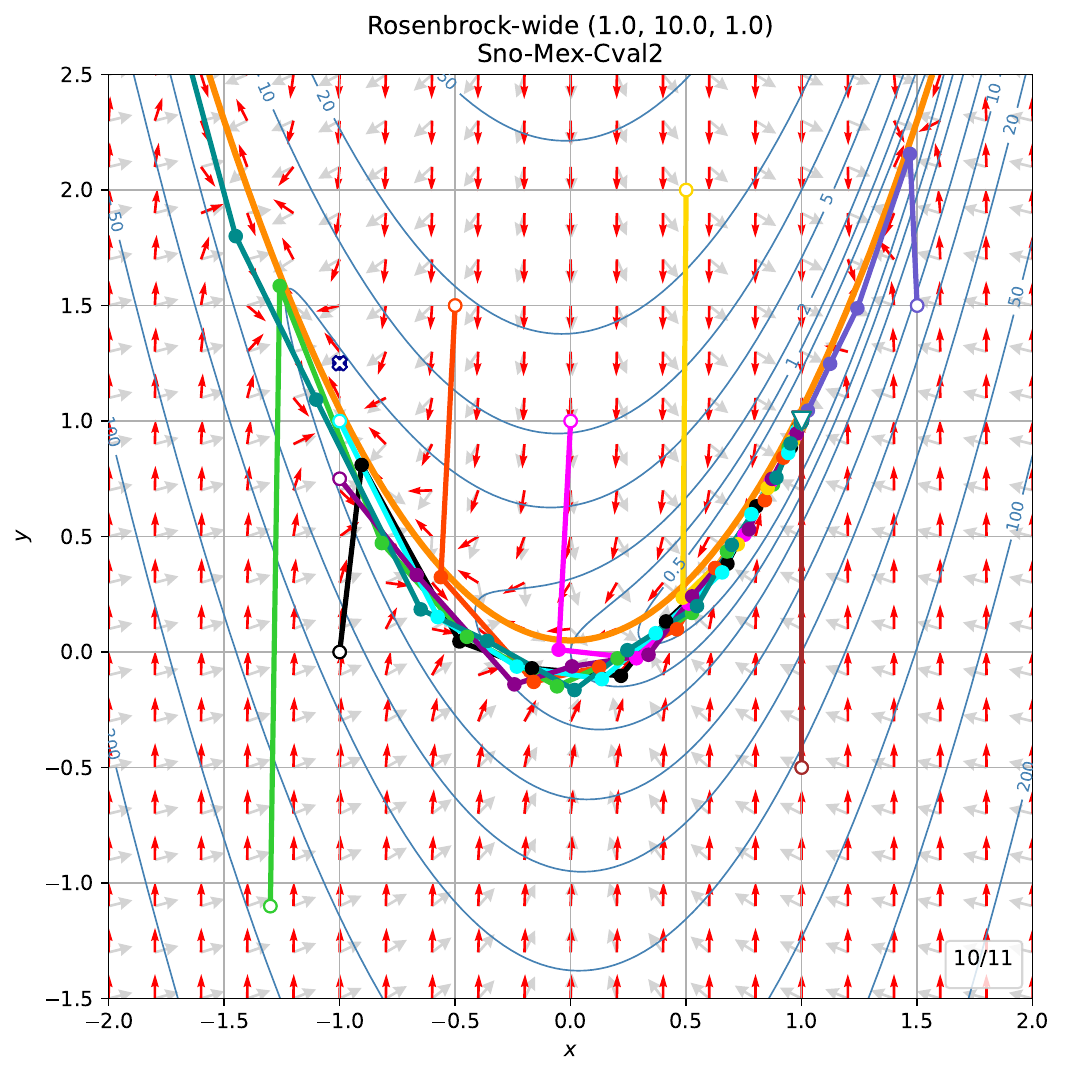}
\hspace*{1mm}
\includegraphics[height=0.47\textwidth]{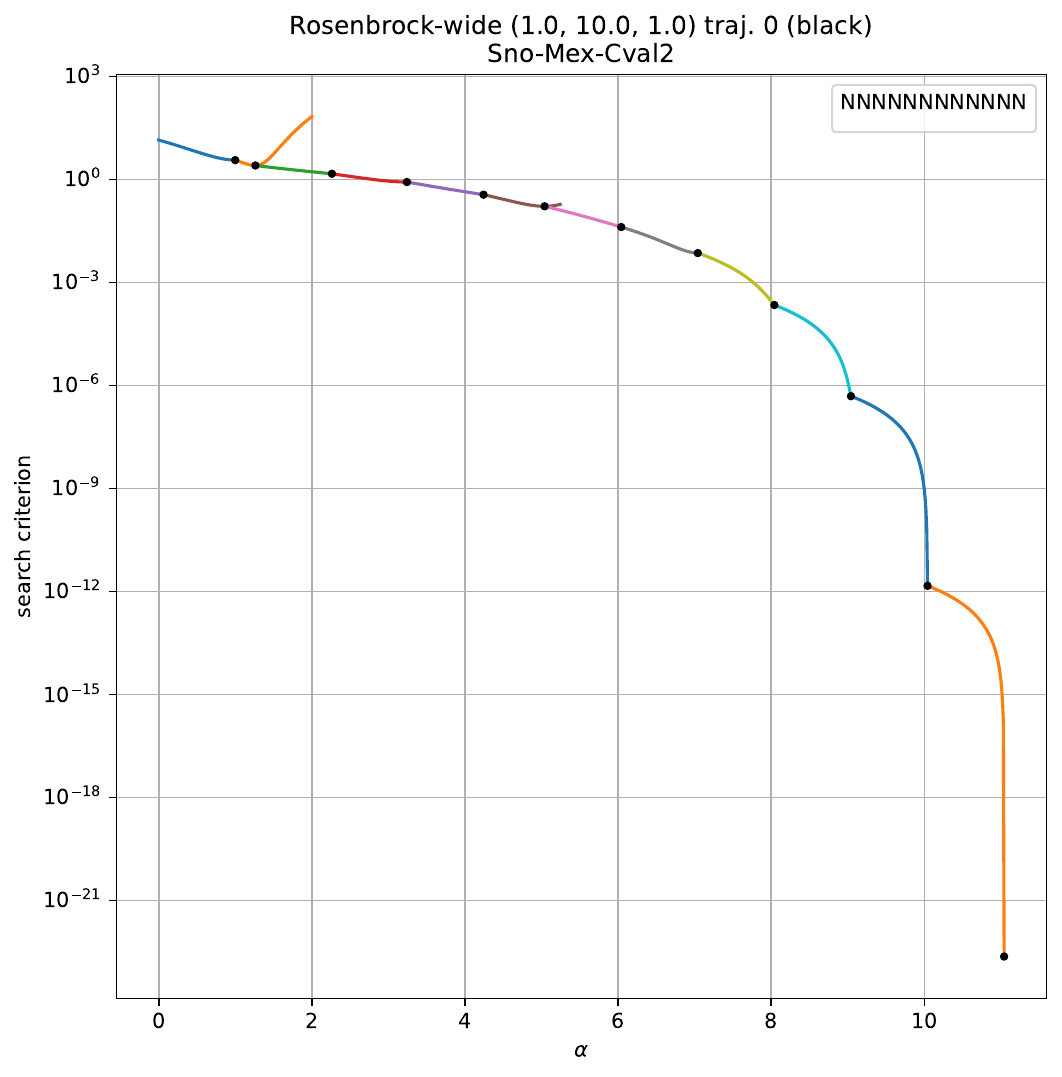}\\
\includegraphics[height=0.47\textwidth]{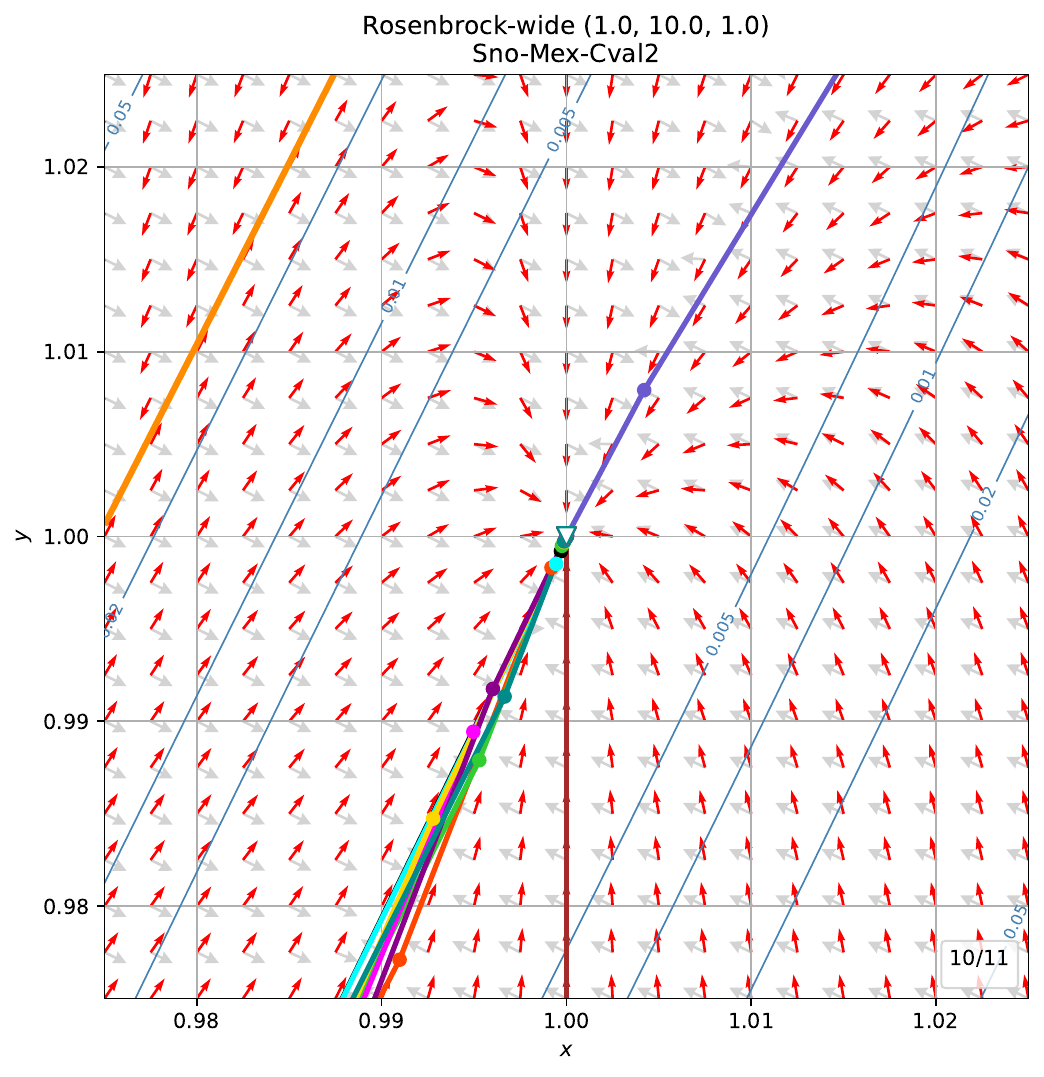}
\hspace*{1mm}
\includegraphics[height=0.47\textwidth]{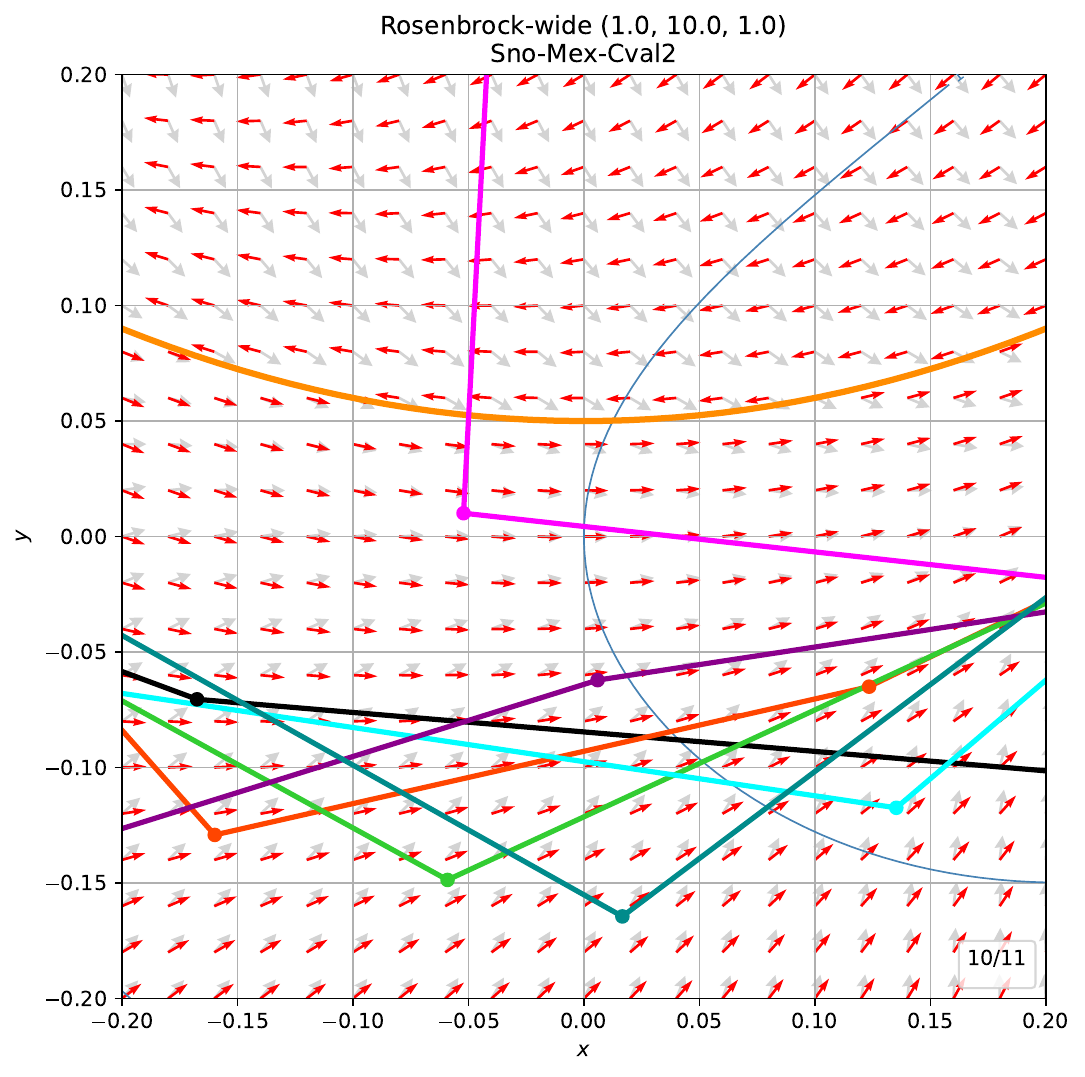}\\      
\caption{}
\label{fig_newton_rsnbrk_no-ex}
\end{center}
\end{figure}
The results for the \textbf{zigzag strategy} are shown in
\textbf{Figure~\ref{fig_newton_rsnbrk_zzp}}. As can be seen in the top
left diagram, most trajectories approach the valley bottom in the down
phase and then zigzag along the $\check{\tau}$ ravine. As the top right
diagram shows for the black trajectory, the down phase in the first
sub-step apparently performs a full Newton step (using explicit search),
and then, starting from the end point, descends deeper into the ravine
in the refinement (using Golden Section). A number of zigzag movements
follow where the step is damped. As soon as the zag phases don't reach
the escape threshold any longer, the steps become undamped.
\begin{figure}[t]
\begin{center}
\includegraphics[height=0.47\textwidth]{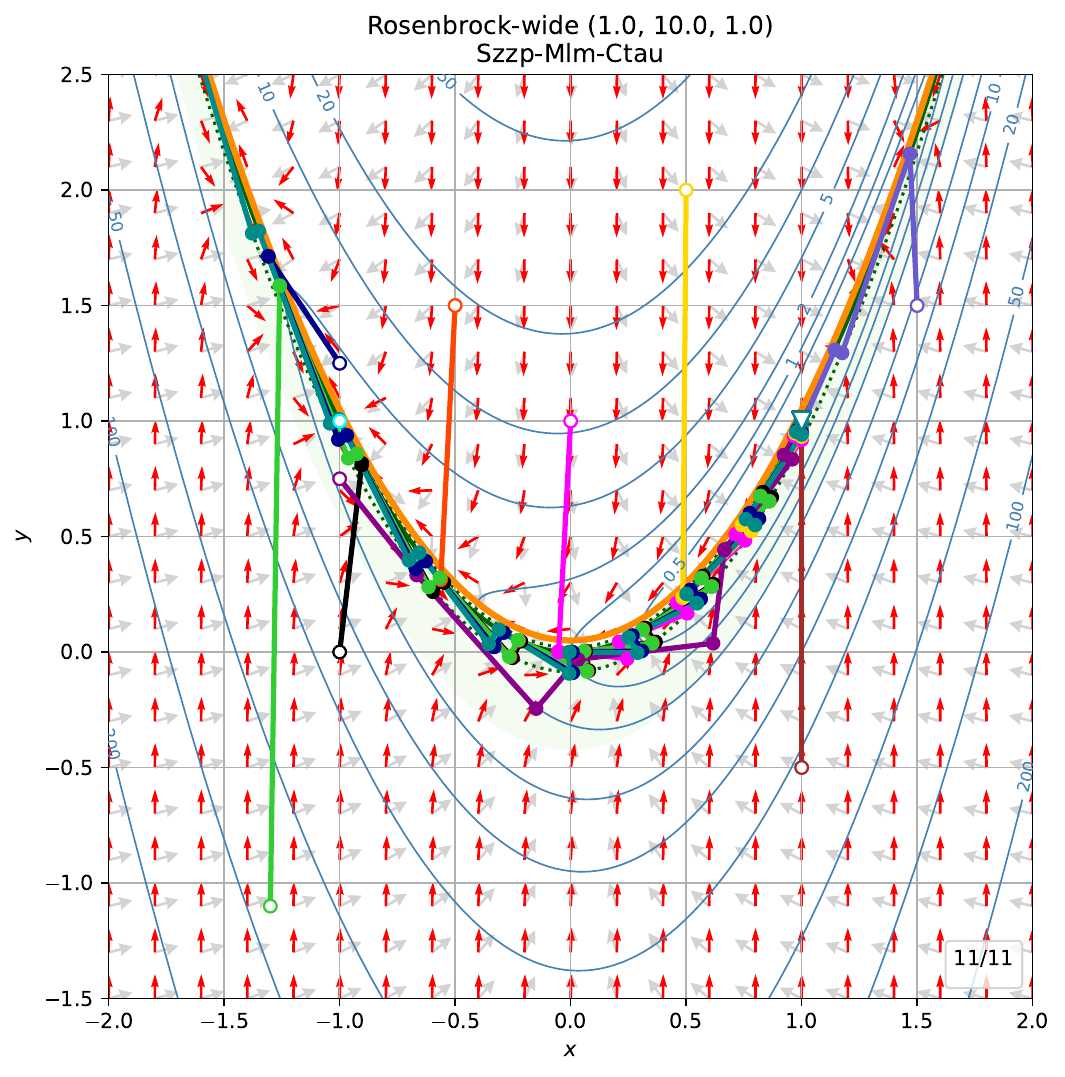}
\hspace*{1mm}
\includegraphics[height=0.47\textwidth]{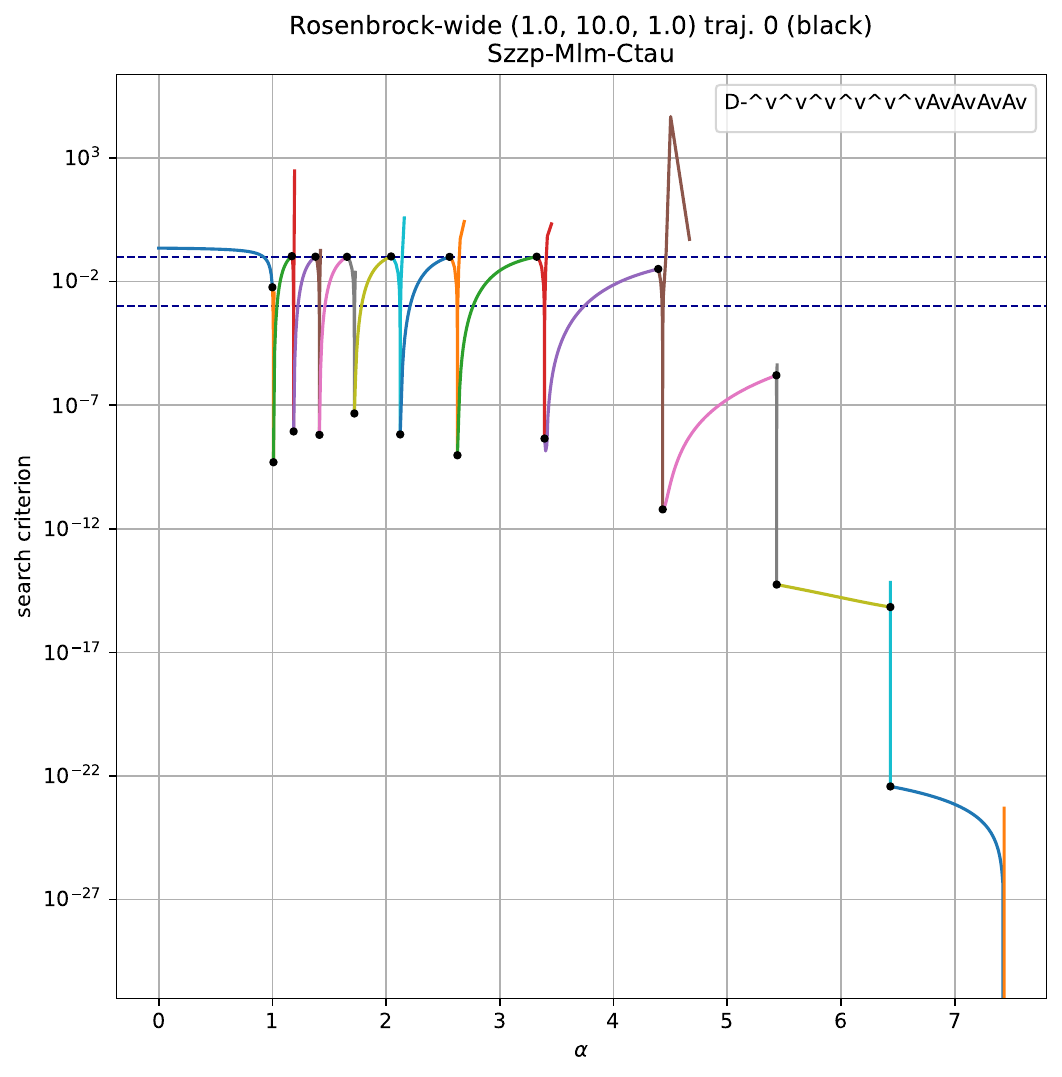}\\
\includegraphics[height=0.47\textwidth]{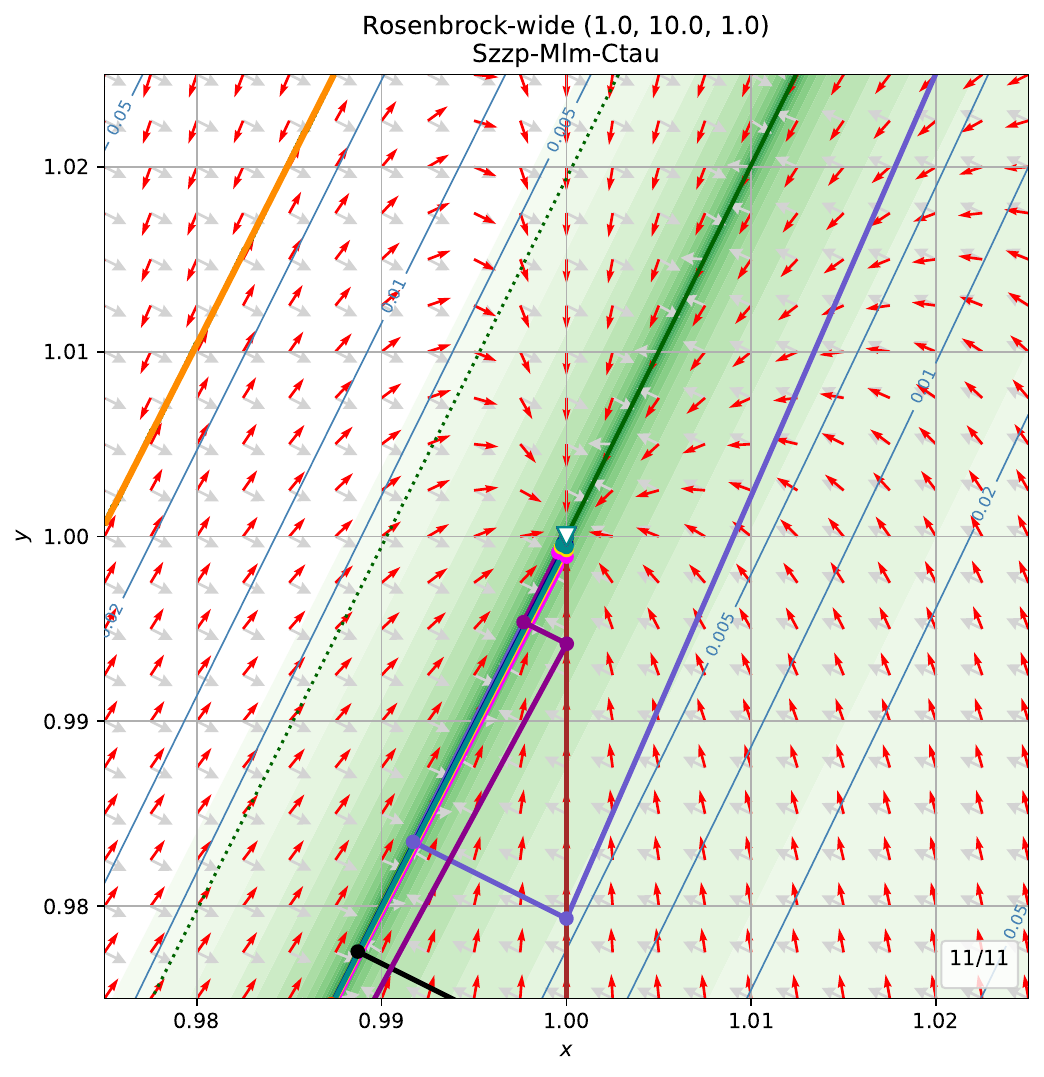}
\hspace*{1mm}
\includegraphics[height=0.47\textwidth]{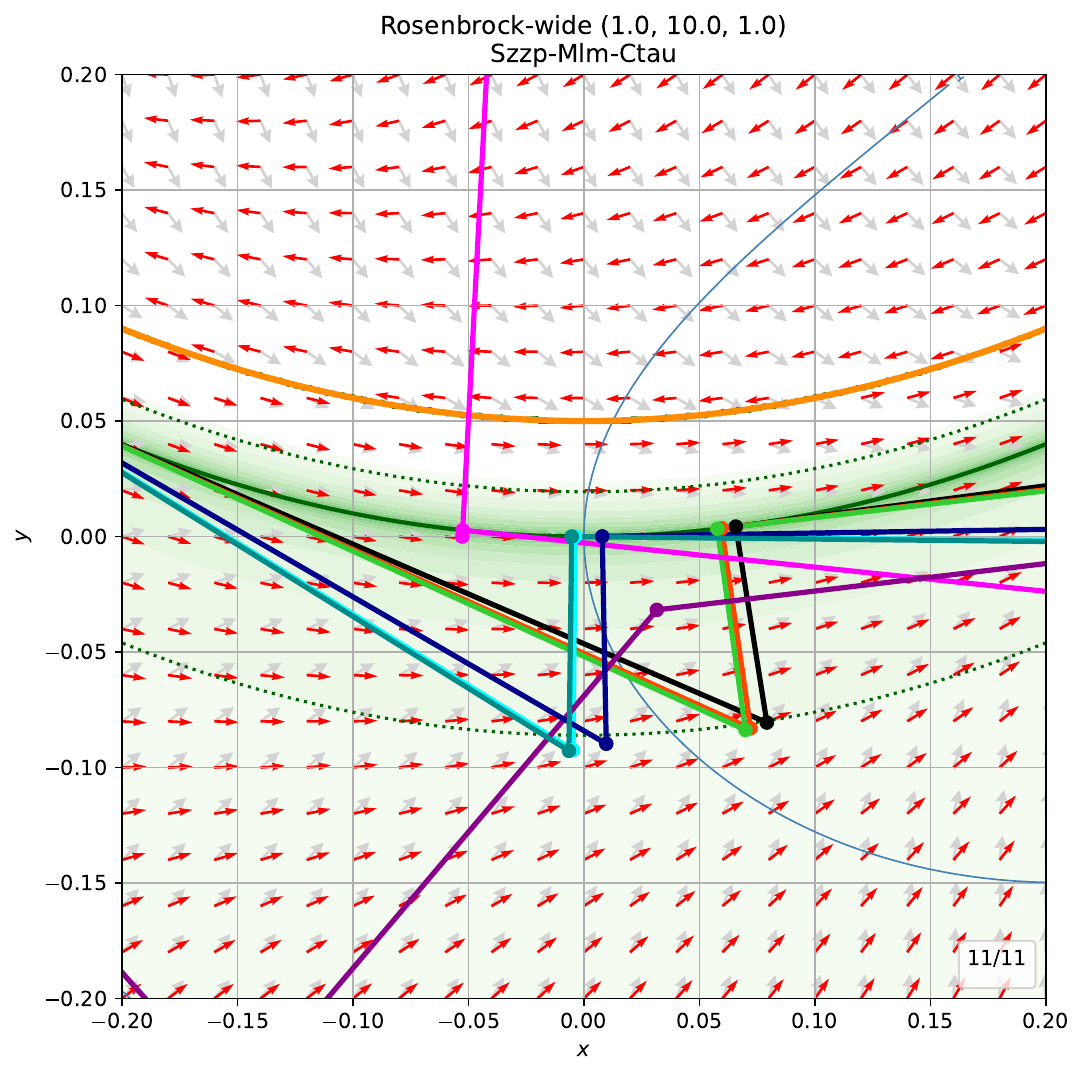}\\      
\caption{}
\label{fig_newton_rsnbrk_zzp}
\end{center}
\end{figure}
In the bottom right diagram it can be seen how the method zigzags
between dotted line belonging to the escape threshold (end point of the
zig phase) and the dark green line representing the ravine (end point of
the zag phase). This behavior differs from the one observed for Newton's
method with value-based line search. Also in the approach to the minimum
(bottom left diagram) we see differences in the behavior between these
two methods. Returning to the top left diagram we see that methods
starting in the vinicity of the countercurrent singularity exhibit a
somewhat different behavior, running ``backwards'' (dark blue) and
staying in larger distance from the ravine (dark magenta).
\textbf{Rosenbrock-wide-saddle}: This is the decisive experiment
demonstrating that the zigzag strategy suggested in this work can also
approach saddle points with damped steps. For comparison,
\textbf{Figure~\ref{fig_newton_rsnbrk_saddle_no}} shows that the
\textbf{plain Newton method} without line search doesn't have a problem
to approach the saddle point (albeit through large jumps; left diagram),
whereas \textbf{Newton's method with explicit line search} on the value
of the objective function fails from all starting points (right
diagram).
\begin{figure}[tp]
\begin{center}
\includegraphics[height=0.47\textwidth]{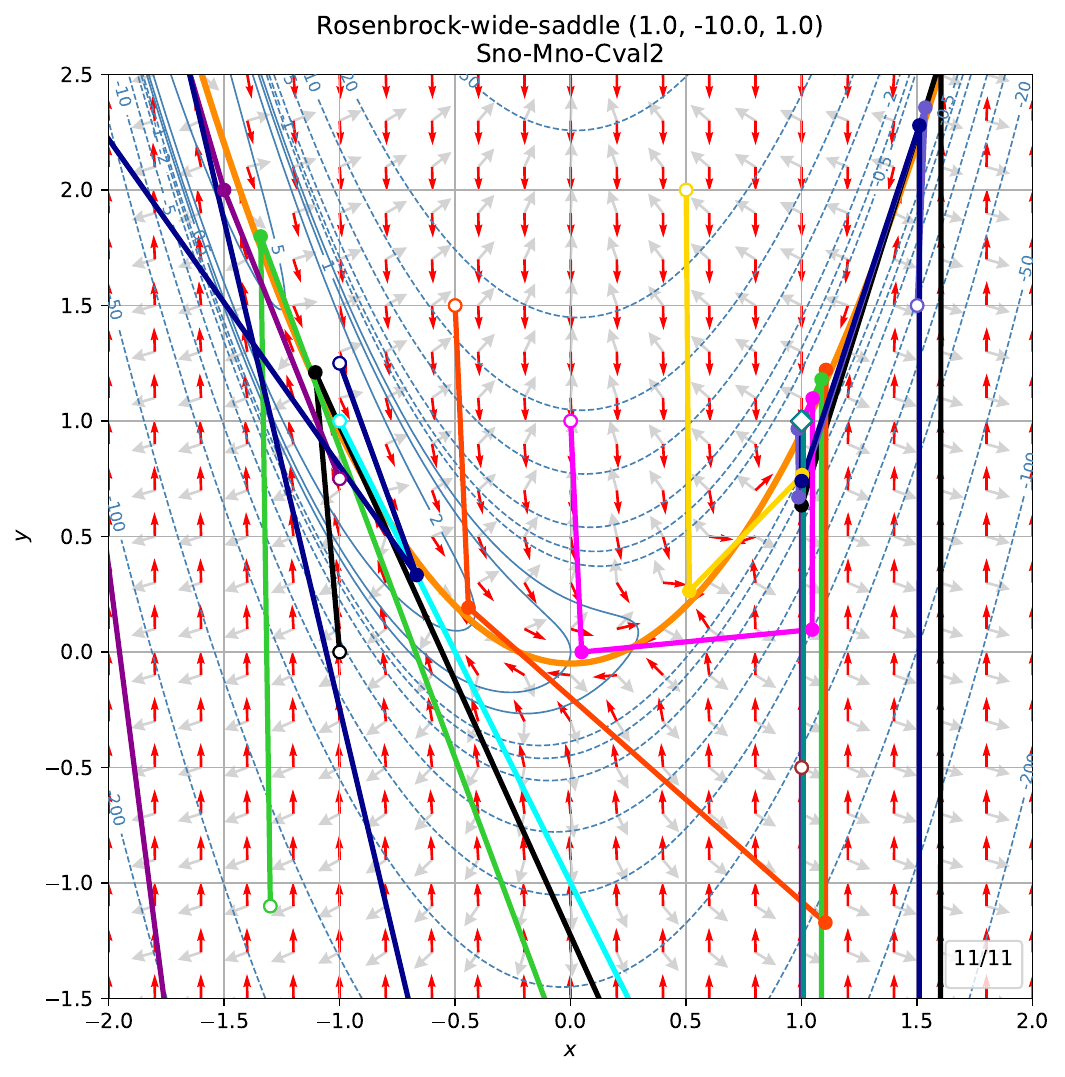}
\hspace*{1mm}
\includegraphics[height=0.47\textwidth]{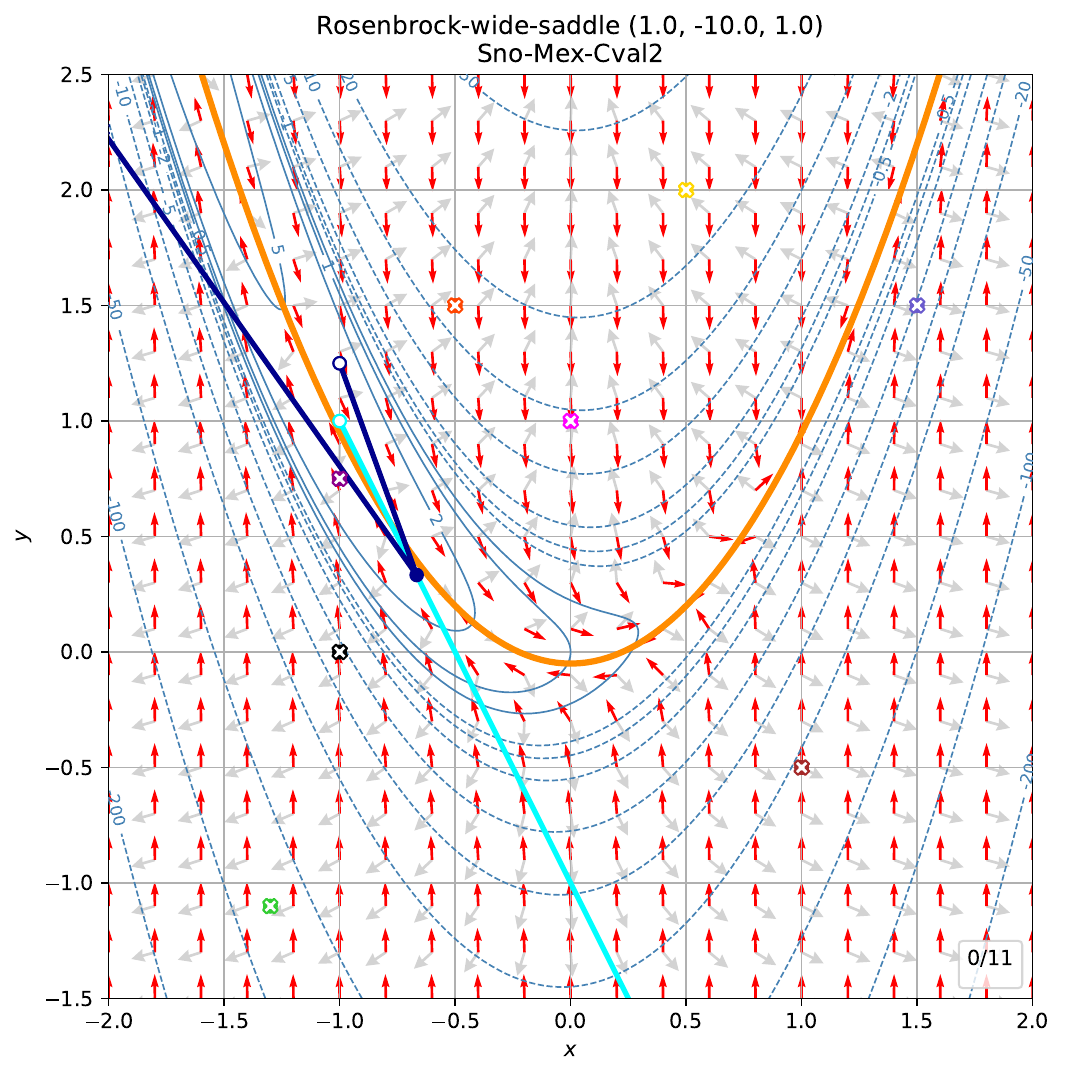}
\caption{}
\label{fig_newton_rsnbrk_saddle_no}
\end{center}
\end{figure}
The results for the \textbf{zigzag method} on the saddle-point problem
are shown in figure \textbf{Figure~\ref{fig_newton_rsnbrk_saddle_zzp}}.
In the top left diagram it is visible that all trajectories run along
the ``valley'' bottom and successfully reach the saddle point, but
perform relatively short steps. The dark magenta trajectory starting
close to the countercurrent singularity at $(-1,0.25)$, however, first
runs in the opposite direction, but seems to channel into the ``valley''
somewhere outside the figure range.

The top right diagram visualizes the criterion $\check{\tau}$ for the
black trajectory. The trajectory starts with a descent into the ravine,
then a number of short zigzag steps follow where the escape threshold is
reached; these steps are damped. At the end of the trajectory there are
a few zigzag steps where the escape threshold is not reached and which
therefore are undamped.

In the bottom right diagram we see the tight zigzags between the dotted
curve (escape threshold) and the ravine. Compared to the minimum version
of Rosenbrock's function, the distance between the escape-threshold
curve and the ravine is smaller, leading to shorter steps. The approach
to the saddle point (bottom right diagram) is similar in both versions
of Rosenbrock's function.
\begin{figure}[t]
\begin{center}
\includegraphics[height=0.47\textwidth]{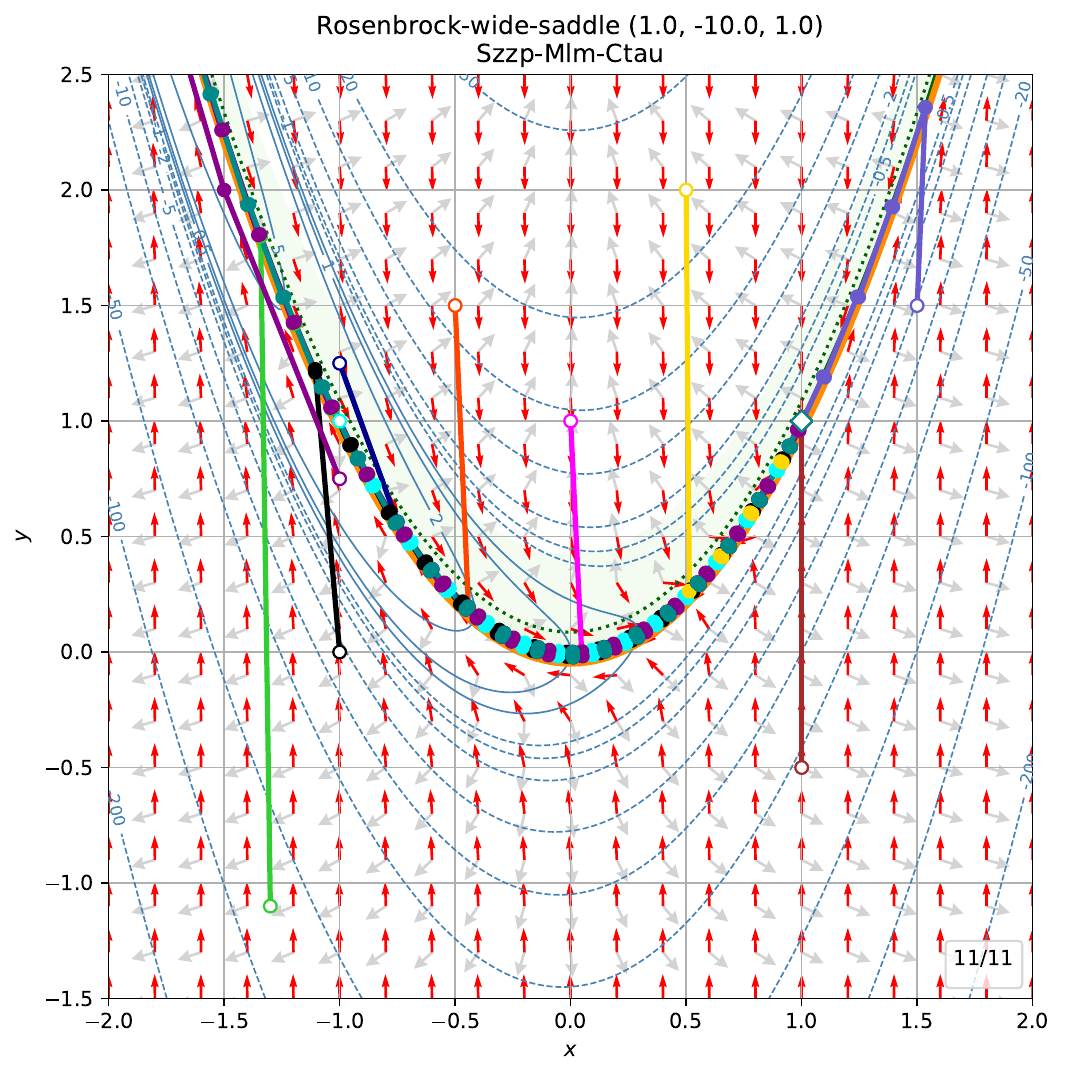}
\hspace*{1mm}
\includegraphics[height=0.47\textwidth]{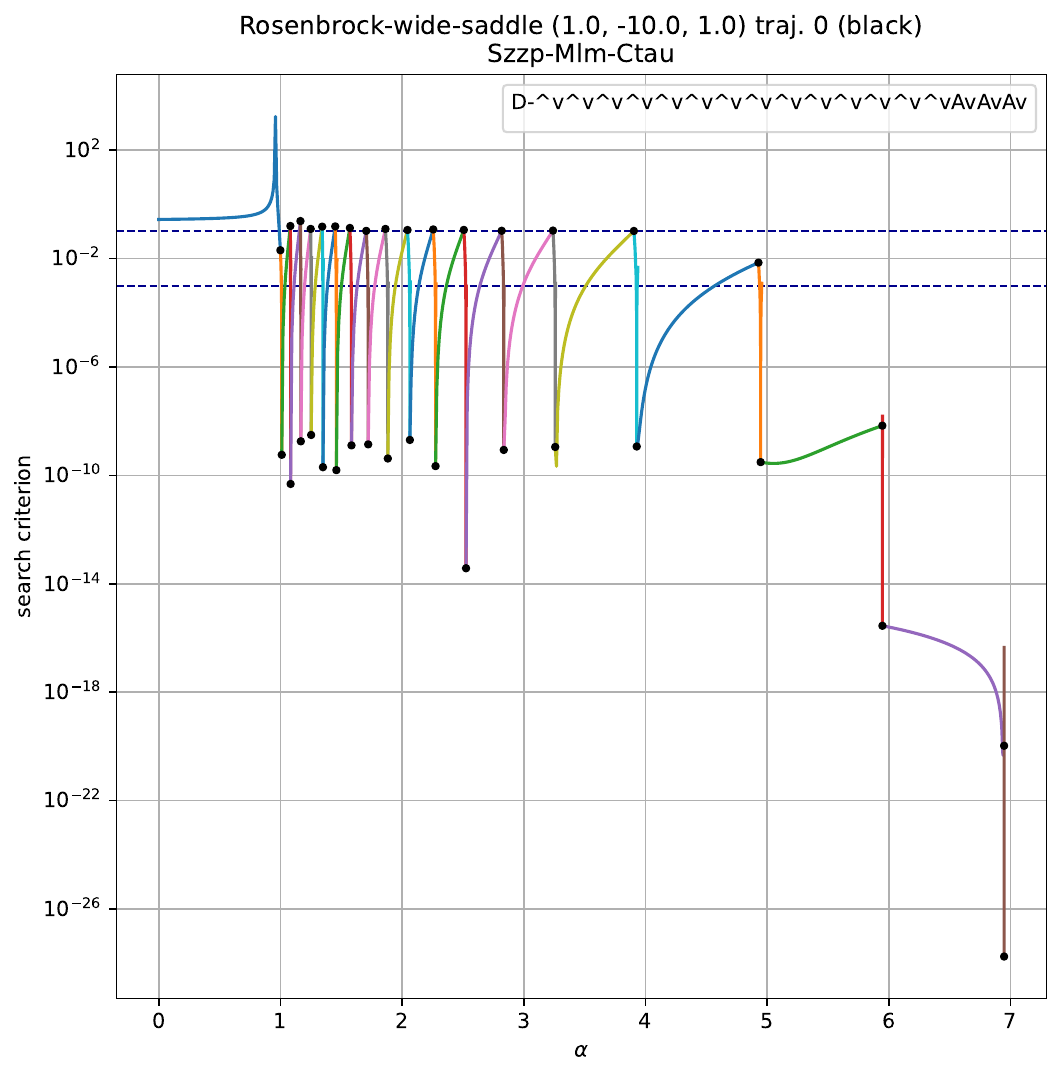}\\
\includegraphics[height=0.47\textwidth]{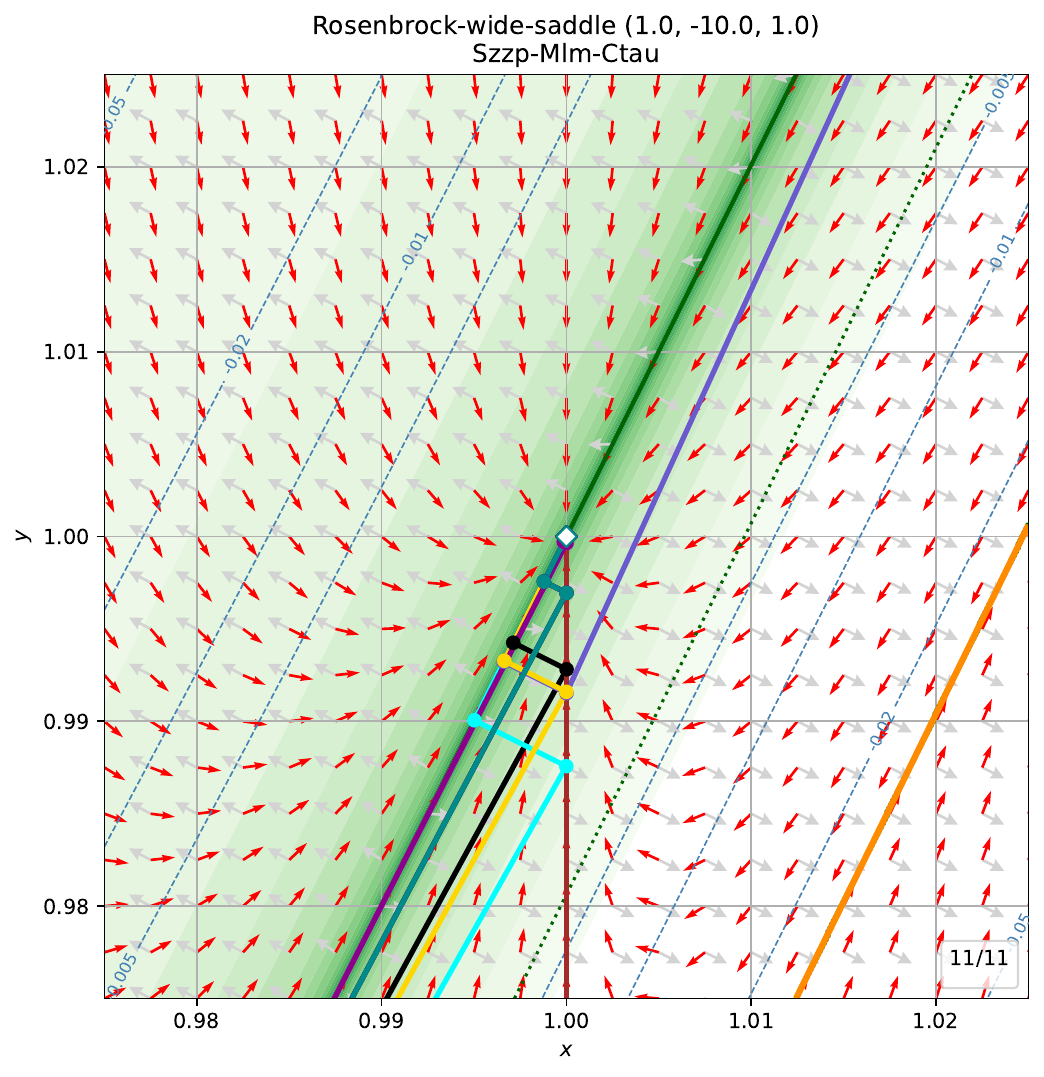}
\hspace*{1mm}
\includegraphics[height=0.47\textwidth]{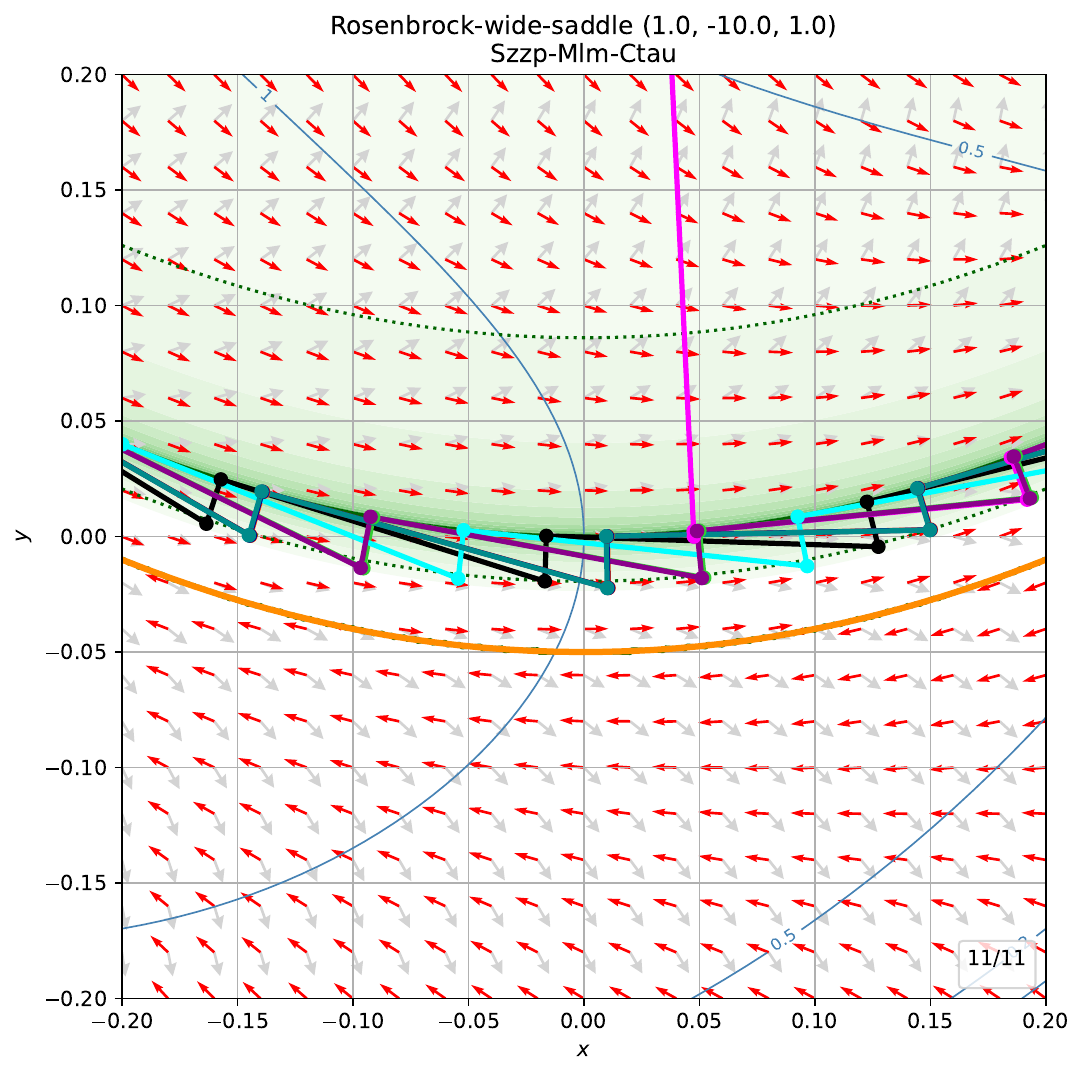}\\      
\caption{}
\label{fig_newton_rsnbrk_saddle_zzp}
\end{center}
\end{figure}
\textbf{Himmelblau}: In \textbf{Figure~\ref{fig_newton_hmlbl}}, the
behavior of the plain Newton method and of the zigzag method is compared
(Newton's method with line search cannot approach saddle points and is
therefore not tested). We see that both methods are successful from all
starting points. An advantage of the zigzag method (bottom left diagram)
seems to be that there is a tendency to approach the {\em closest}
stationary point, whereas Newton's method (top left diagram) sometimes
performs large jumps (in particular when starting close a singularity)
and therefore ends in distant stationary points (orangered and black
trajectories). A disadvantage of the zigzag method is visible for the
black trajectory in the two bottom diagrams: In the first steps, the
progress is impeded by tiny zigzags along a tight ravine (but without
triggering the parallelity check, see the strategy string in the bottom
right diagram).
\begin{figure}[t]
\begin{center}
\includegraphics[height=0.47\textwidth]{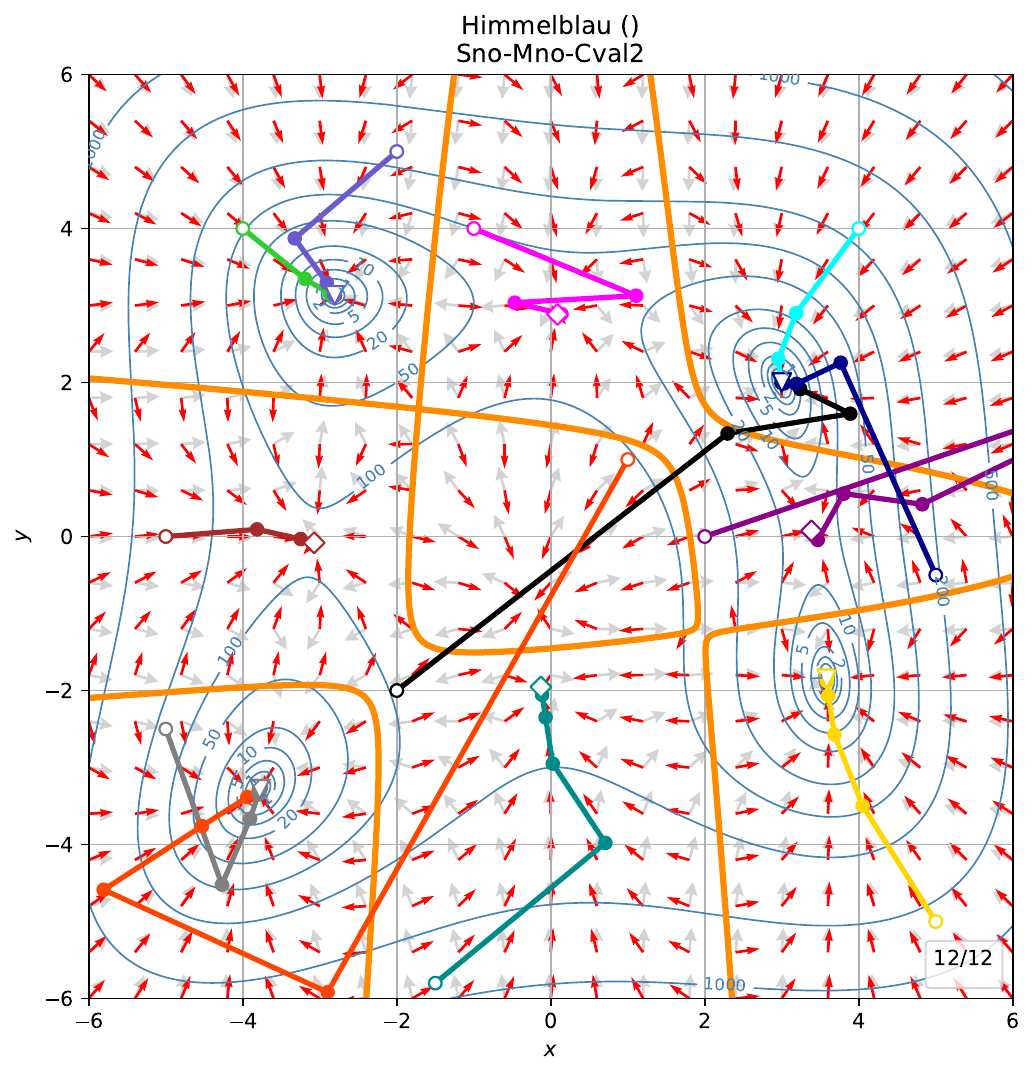}
\hspace*{1mm}
\includegraphics[height=0.47\textwidth]{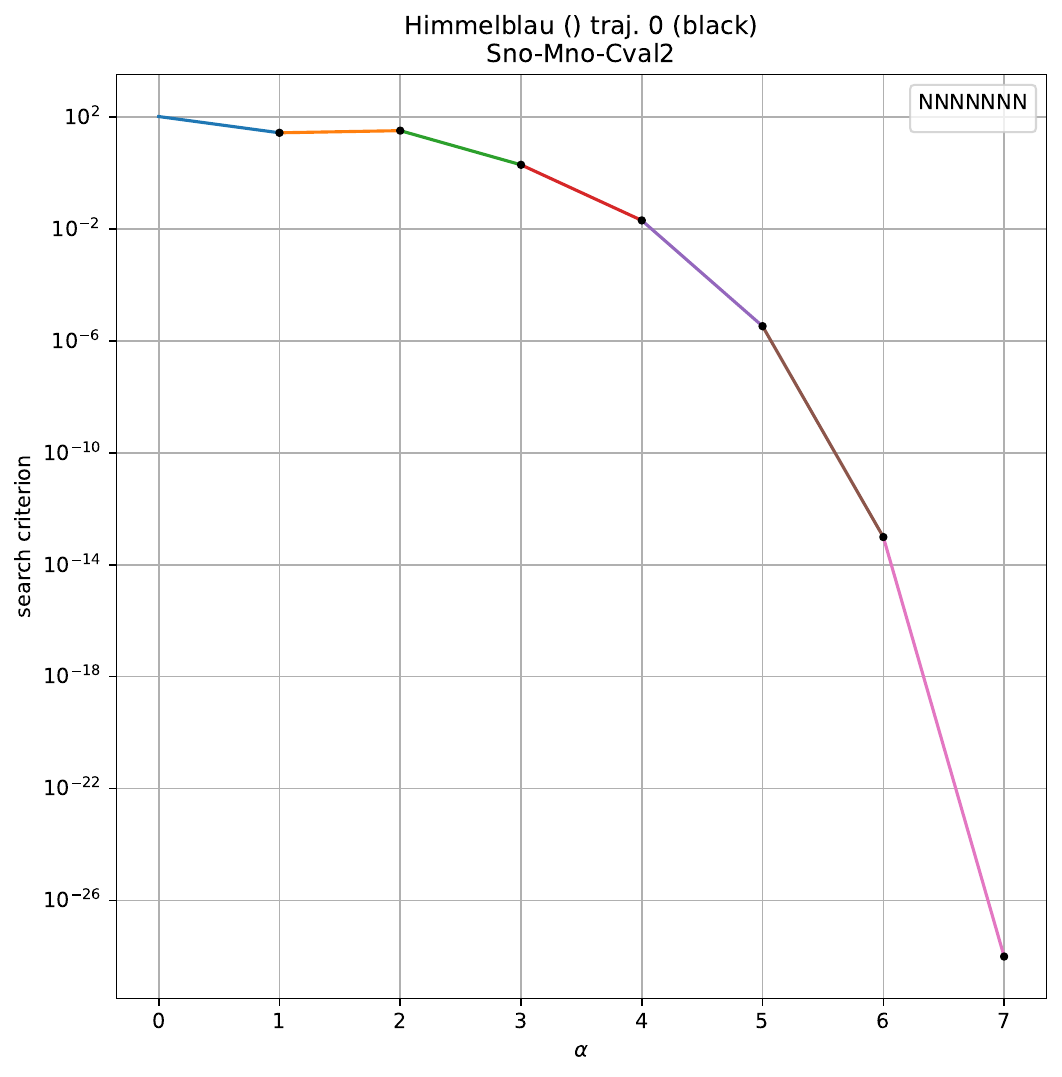}\\
\includegraphics[height=0.47\textwidth]{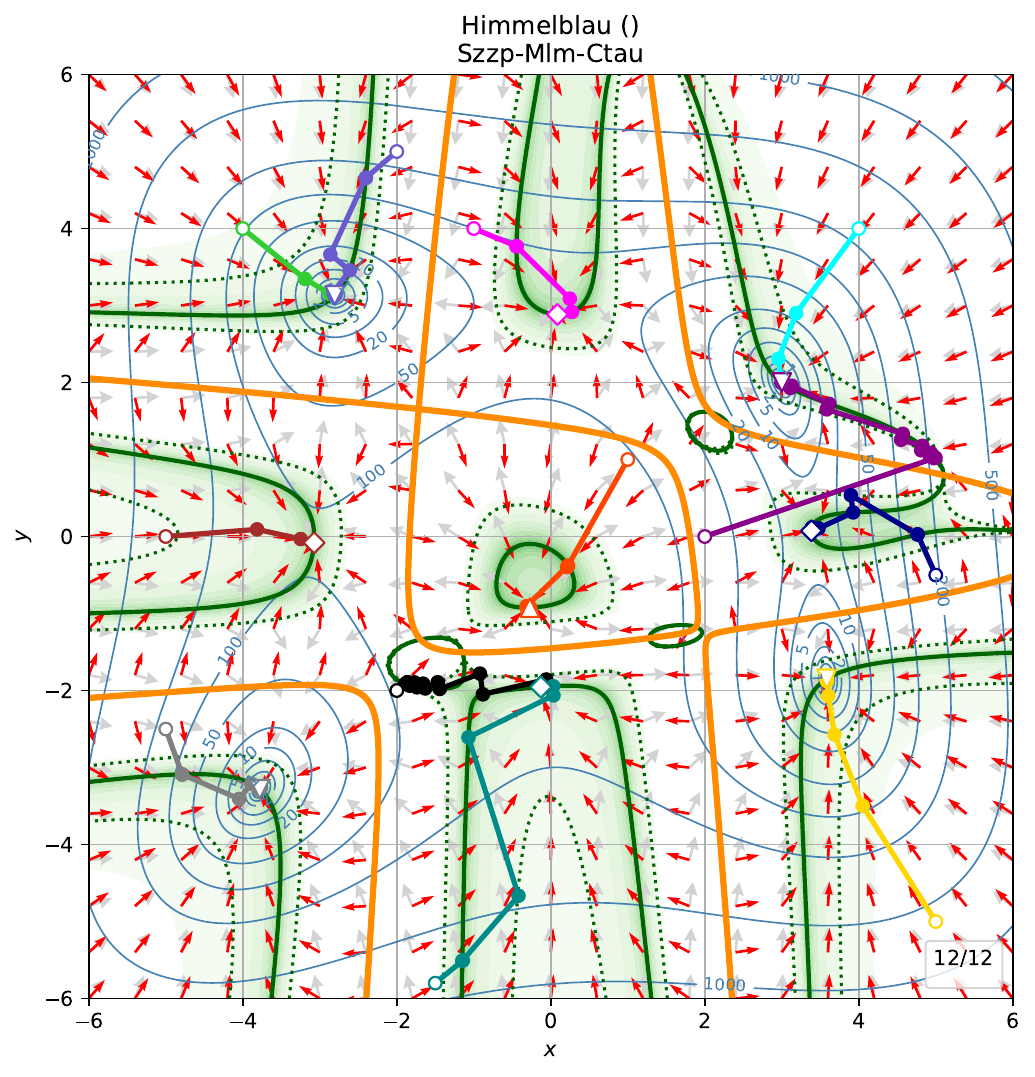}
\hspace*{1mm}
\includegraphics[height=0.47\textwidth]{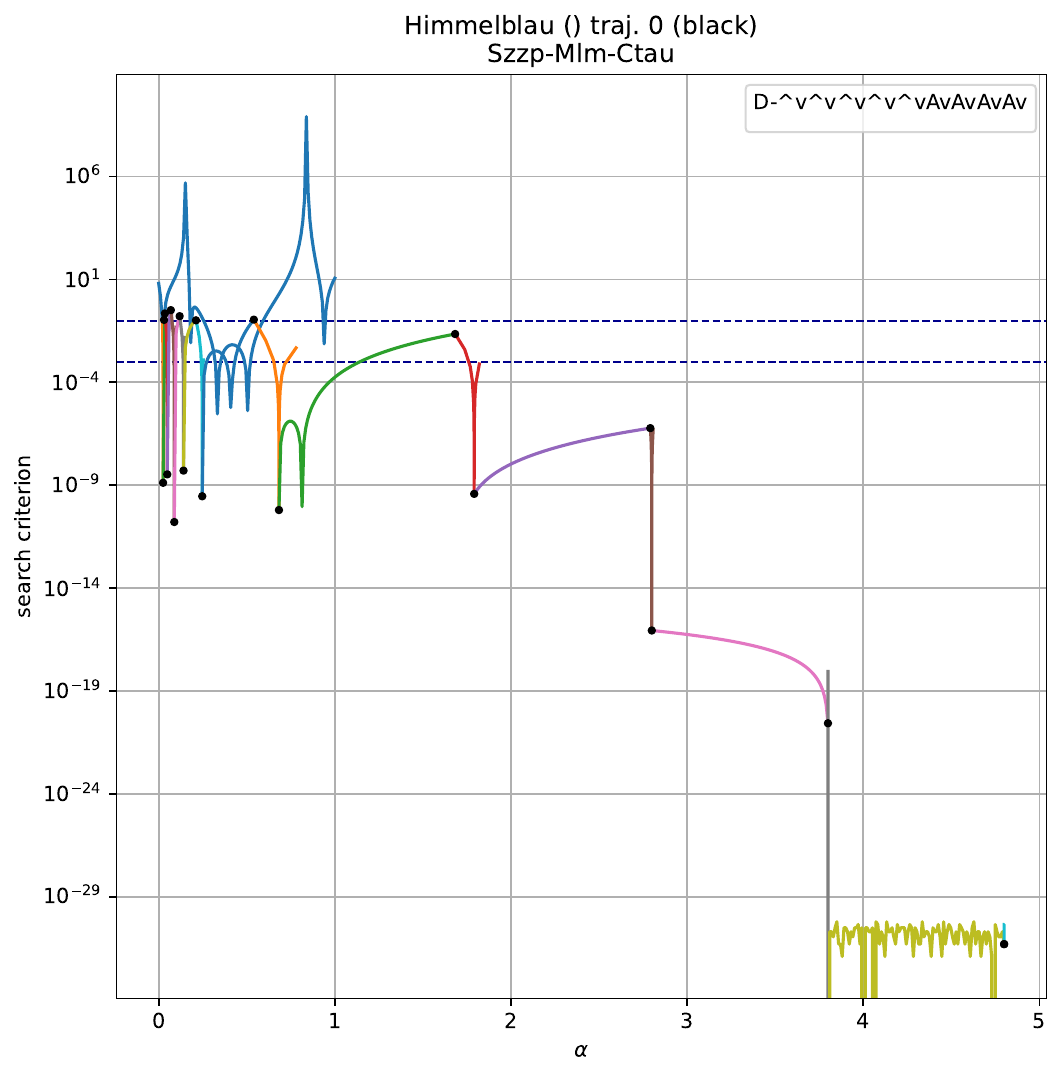}\\      
\caption{}
\label{fig_newton_hmlbl}
\end{center}
\end{figure}
\textbf{Henon-Heiles}: In the results for the Henon-Heiles function in
\textbf{Figure~\ref{fir_newton_hnnhls}} we see that both the plain
Newton method (two top diagrams) and the zigzag method (two bottom
diagrams) succeed from all starting points. A crucial different between
the two method becomes apparent when looking at the slateblue
trajectory: While in Newton's method the trajectory can reach the
minimum in the center, the trajectory is blocked by the closed ravine
from doing so and has to zigzag its way to the upper saddle point
instead (also visible in the bottom right diagram). Only the dark
magenta trajectory which approaches exactly from below can reach the
minimum point in the zigzag strategy (bottom left diagram), presumably
because the $\check{\tau}$-ravine coincides with the singularity on this
path and is therefore missed in the explicit search of the down phase.
\begin{figure}[t]
\begin{center}
\includegraphics[height=0.47\textwidth]{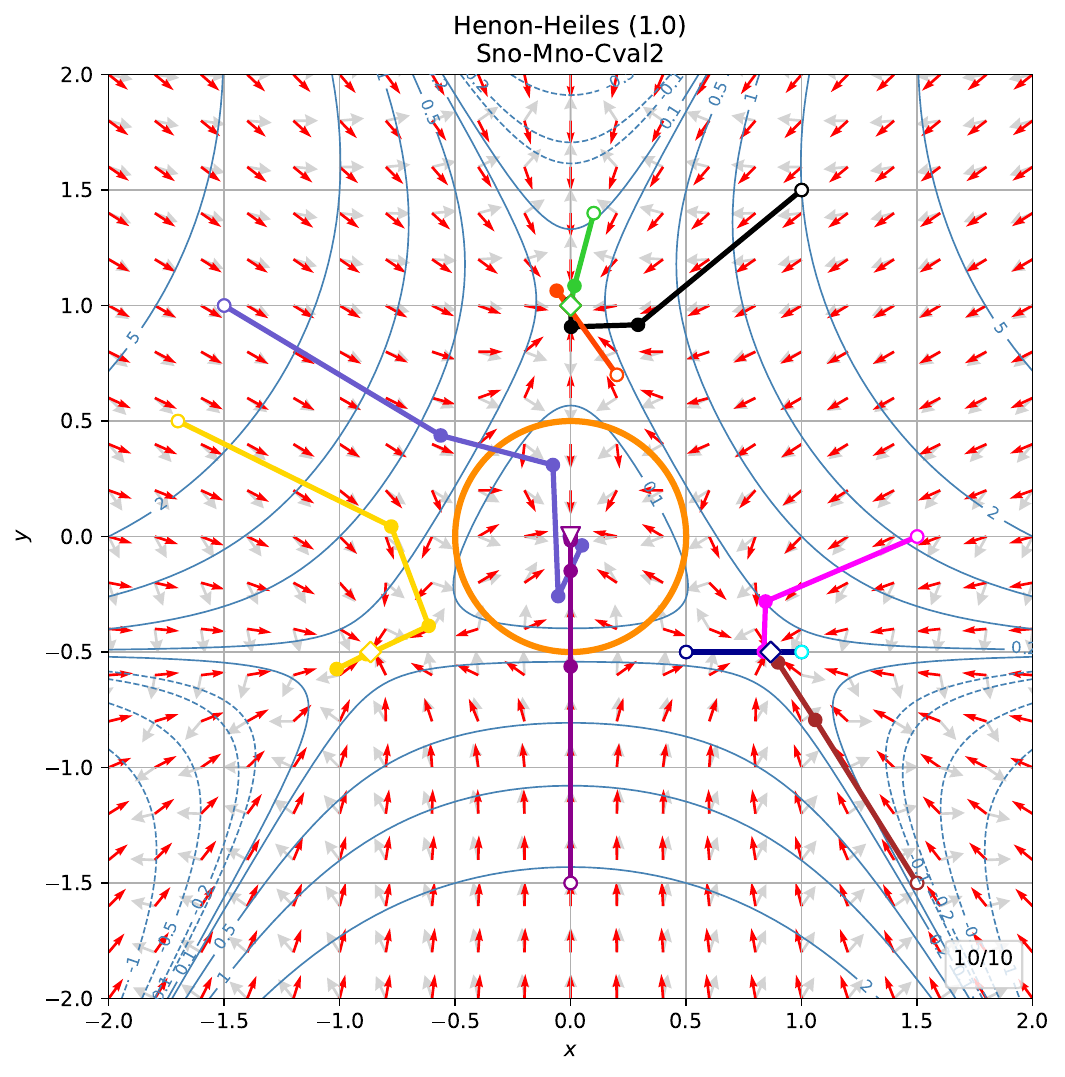}
\hspace*{1mm}
\includegraphics[height=0.47\textwidth]{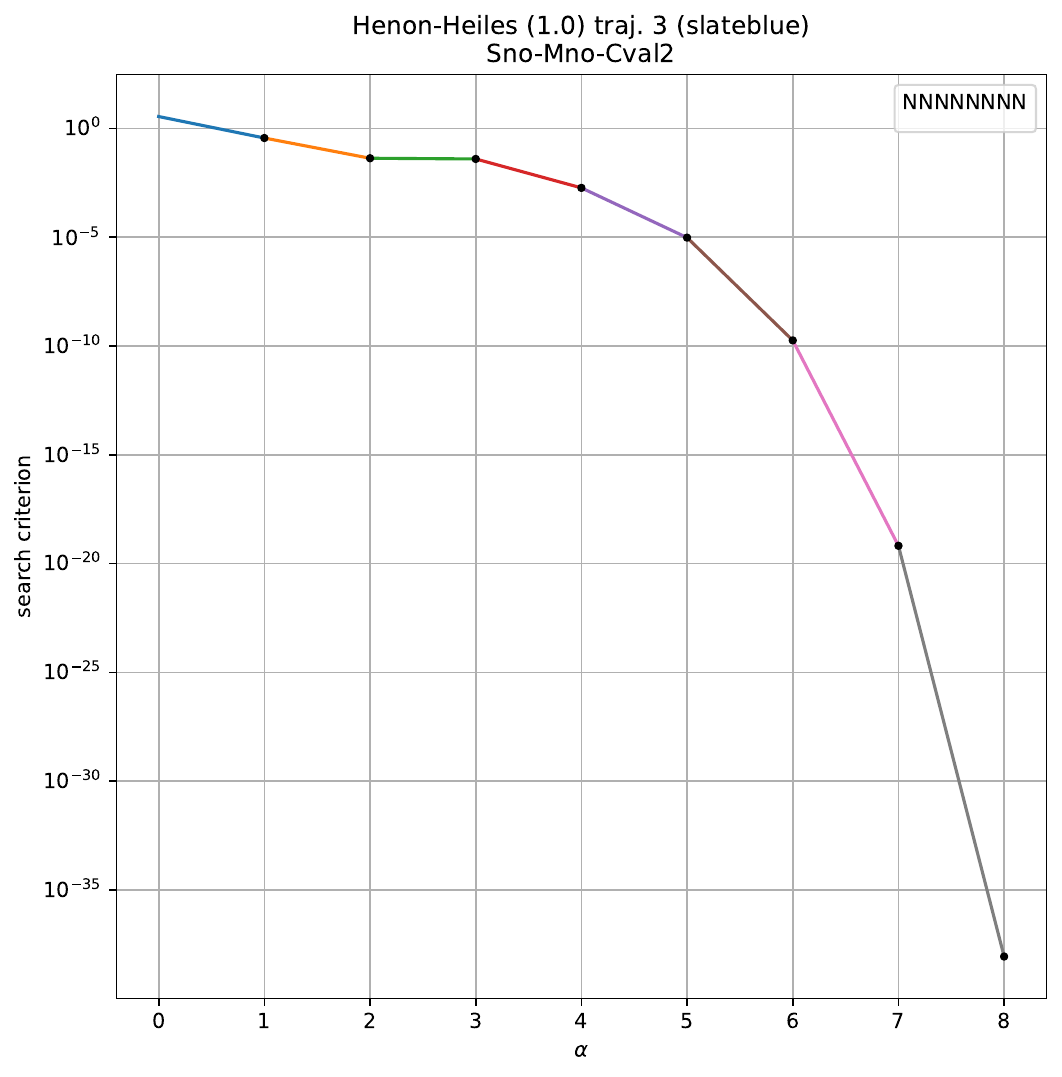}\\
\includegraphics[height=0.47\textwidth]{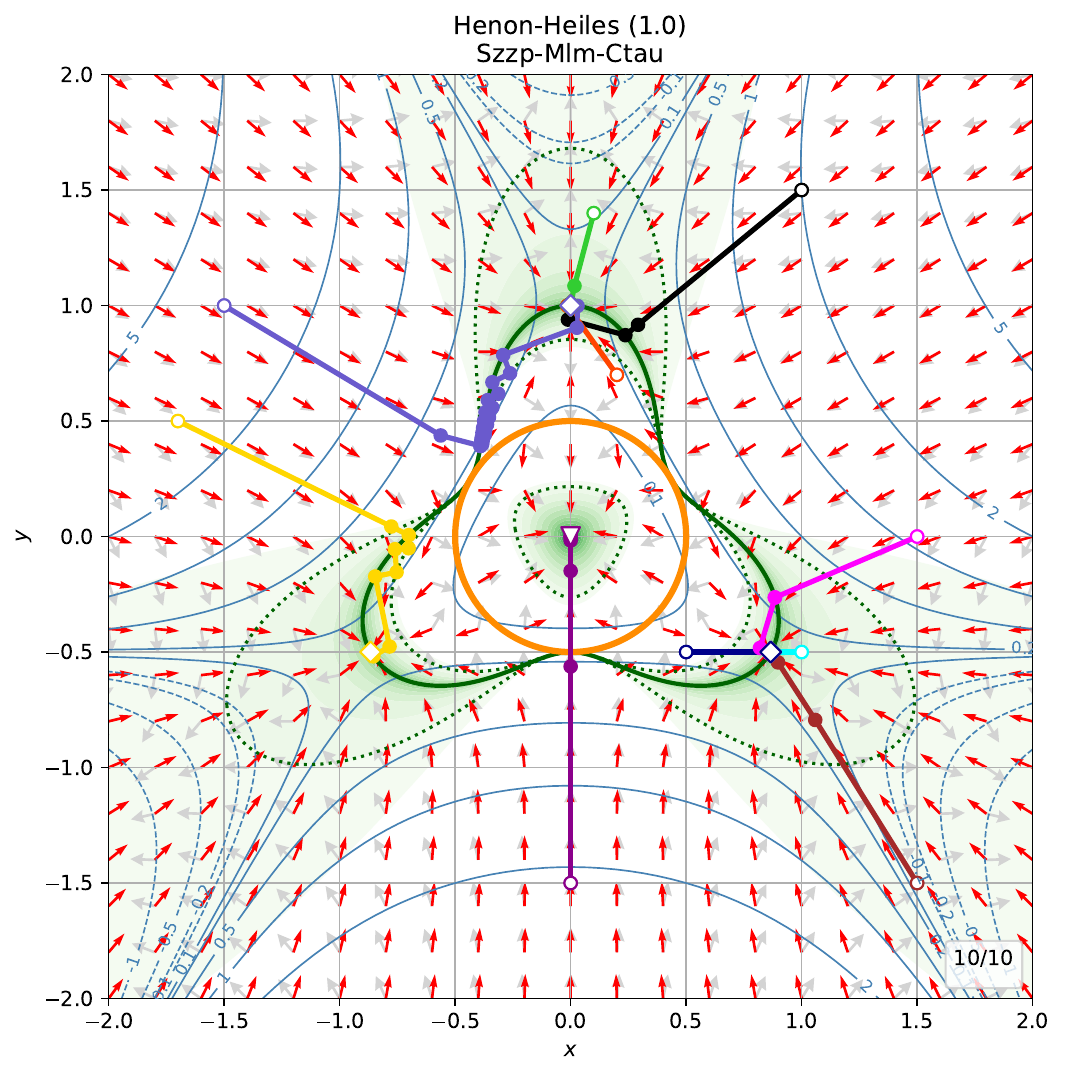}
\hspace*{1mm}
\includegraphics[height=0.47\textwidth]{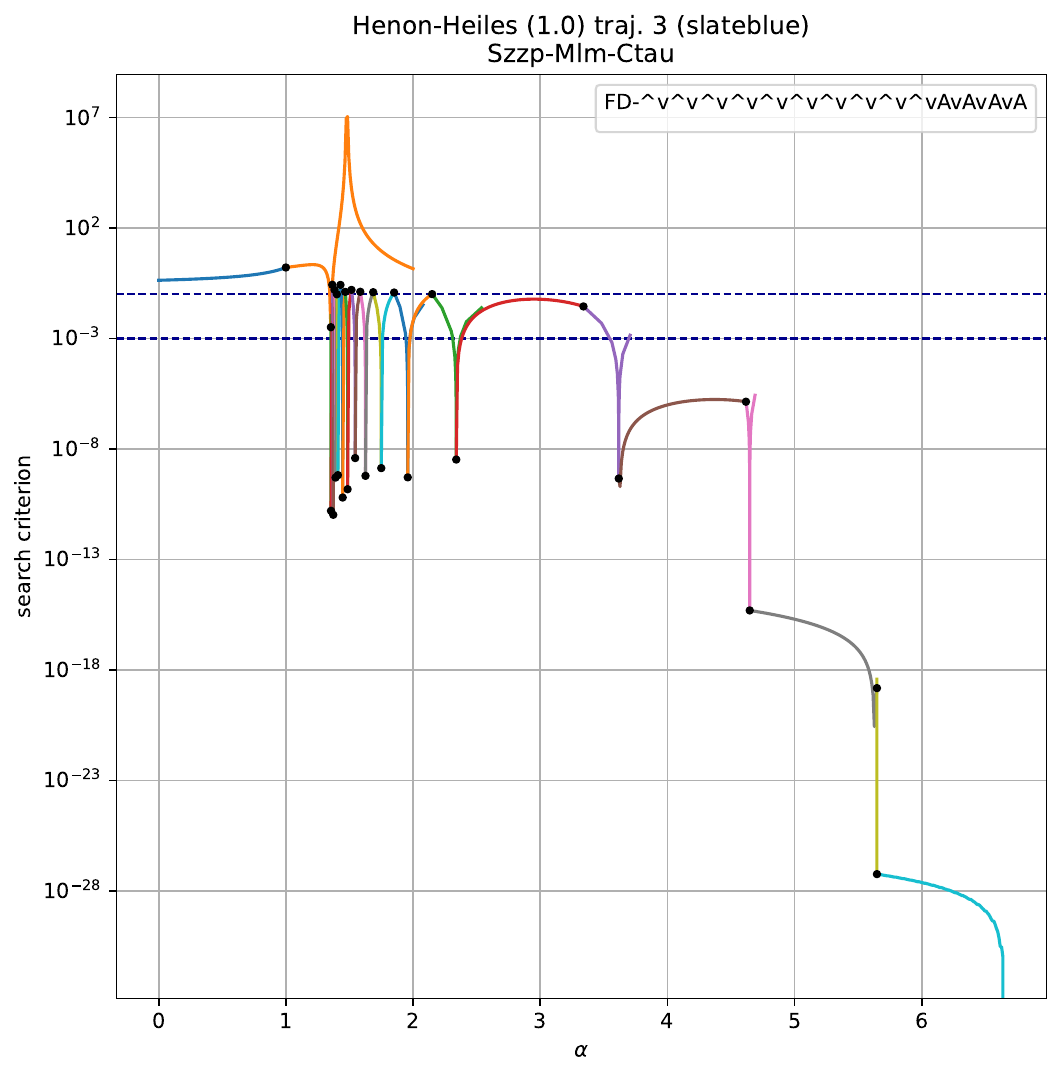}\\      
\caption{}
\label{fir_newton_hnnhls}
\end{center}
\end{figure}
In the \textbf{junction1} function, for which the trajectories are
visualized in \textbf{Figure~\ref{fig_newton_junction1}}, we only have a
single local minimum and therefore can test the plain Newton method (top
left), Newton's method with value-based line search (top right), the
zigzag method without parallelity check (bottom left), and the zigzag
method with parallelity check (bottom right).

Surprisingly, the plain Newton method (top left) only succeeds for half
of the trajectories. This is probably due to the large jumps which lead
some trajectories into regions with diverging Newton steps. Newton's
method with value-based line search (top right) fares much better and
only fails for the magenta trajectory which gets stuck right from the
start. (Unexpectedly, this makes the junction1 function an interesting
test case for line search.)

The behavior of the zigzag strategy is unsatisfying. Without and with
parallelity check (two bottom diagrams), there is lot of unnecessary
zigzagging in tight ravines which impedes the progress. Introducing the
parallelity check improves the performance. While the slateblue and
brown trajectory reach the vicinity of the minimum, they cannot leave
the ravines without parallelity check (bottom left diagram) to cover the
remaining distance to the minimum. With parallelity check (bottom right
diagram) they leave the ravines after some zigzag movements and cross
over to the minimum. The magenta trajectory fails with and without
parallelity check (but for different reasons).
\begin{figure}[t]
\begin{center}
\includegraphics[height=0.47\textwidth]{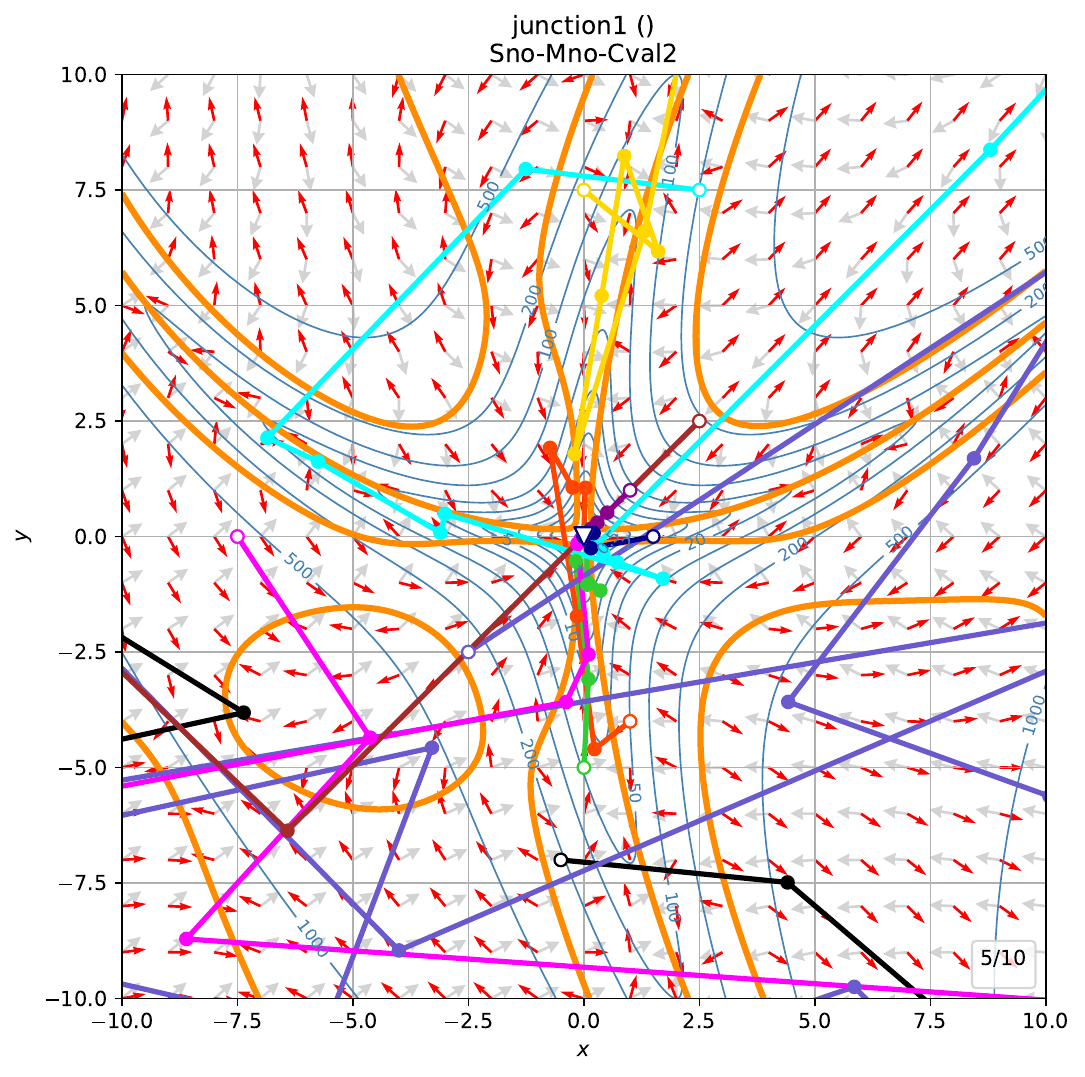}
\hspace*{1mm}
\includegraphics[height=0.47\textwidth]{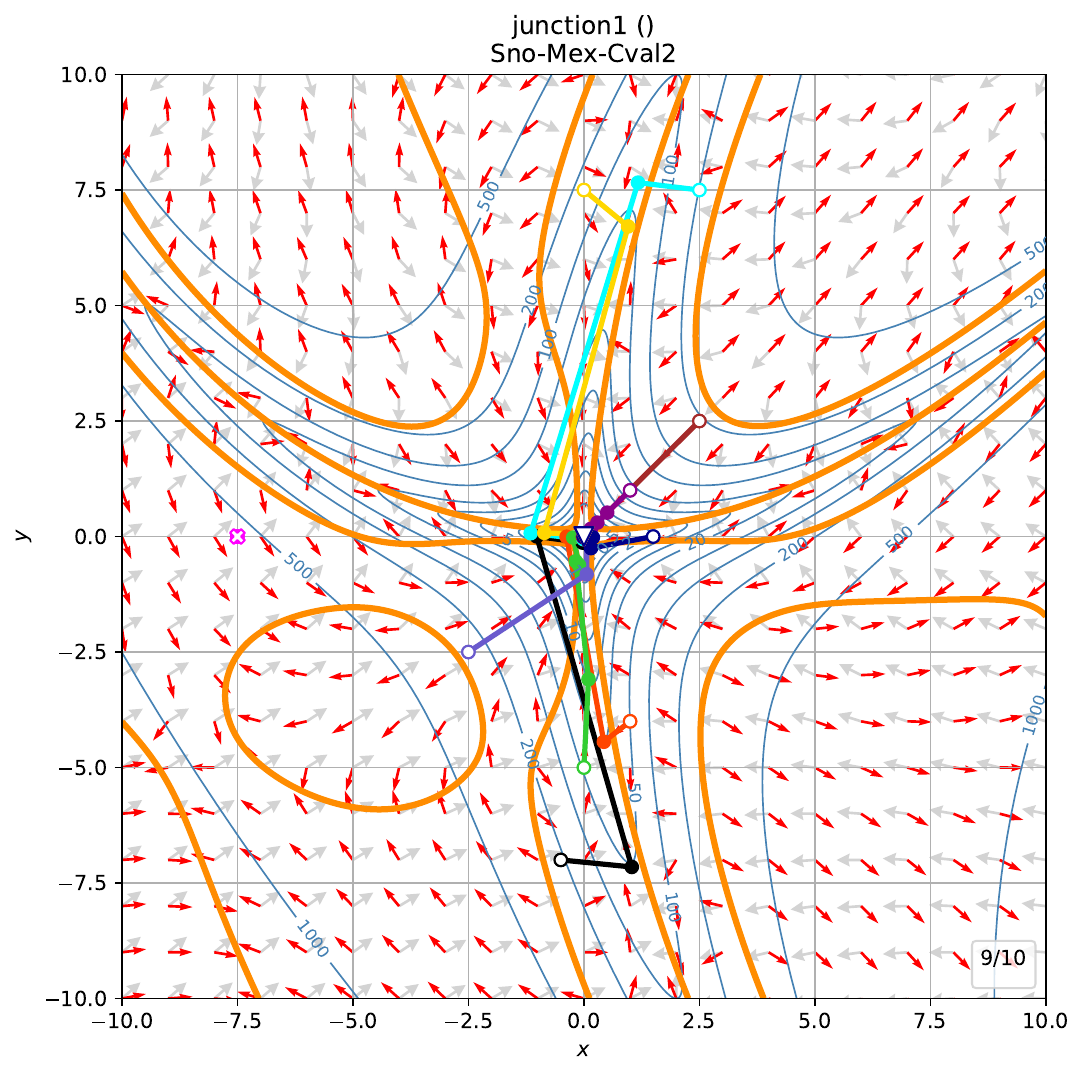}\\
\includegraphics[height=0.47\textwidth]{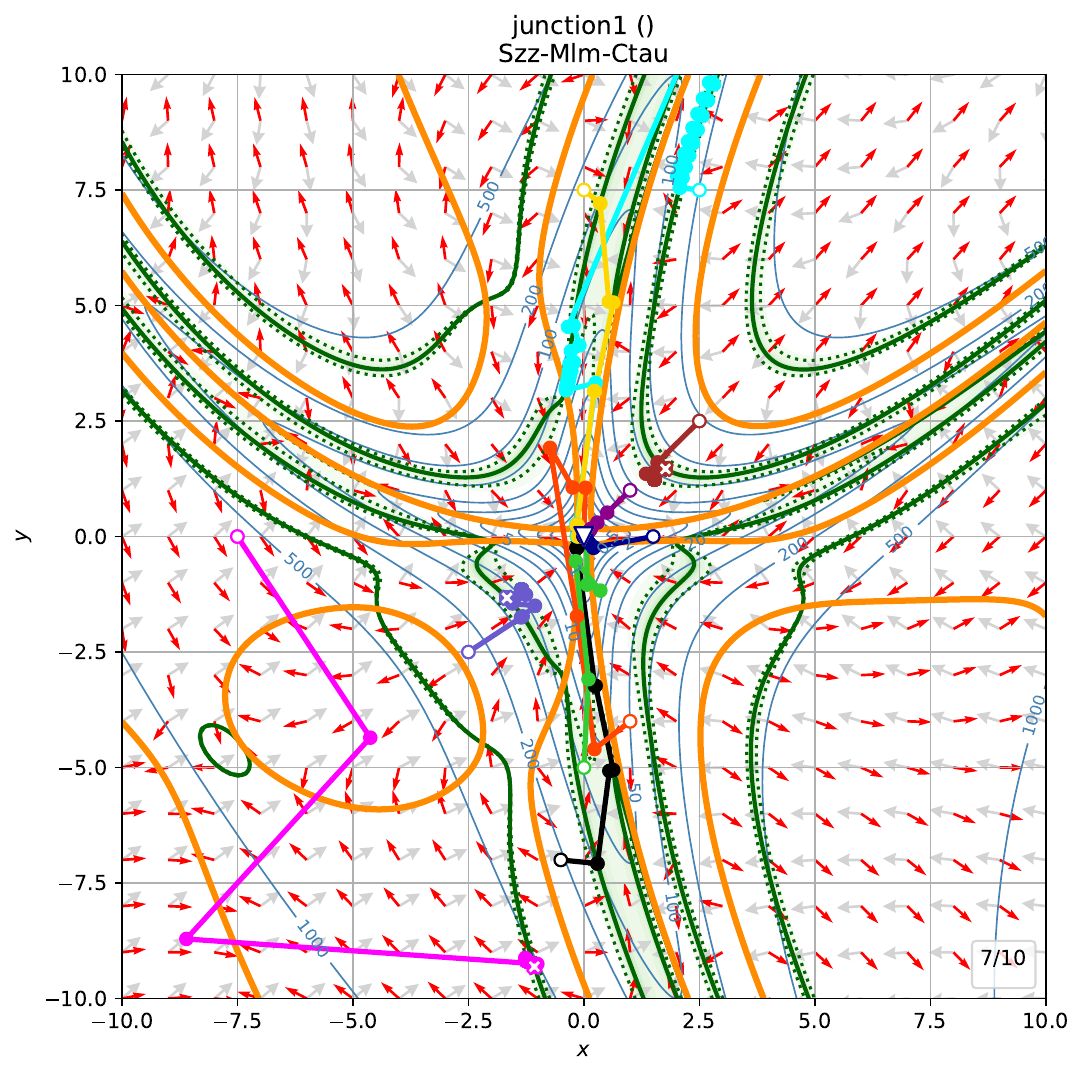}
\hspace*{1mm}
\includegraphics[height=0.47\textwidth]{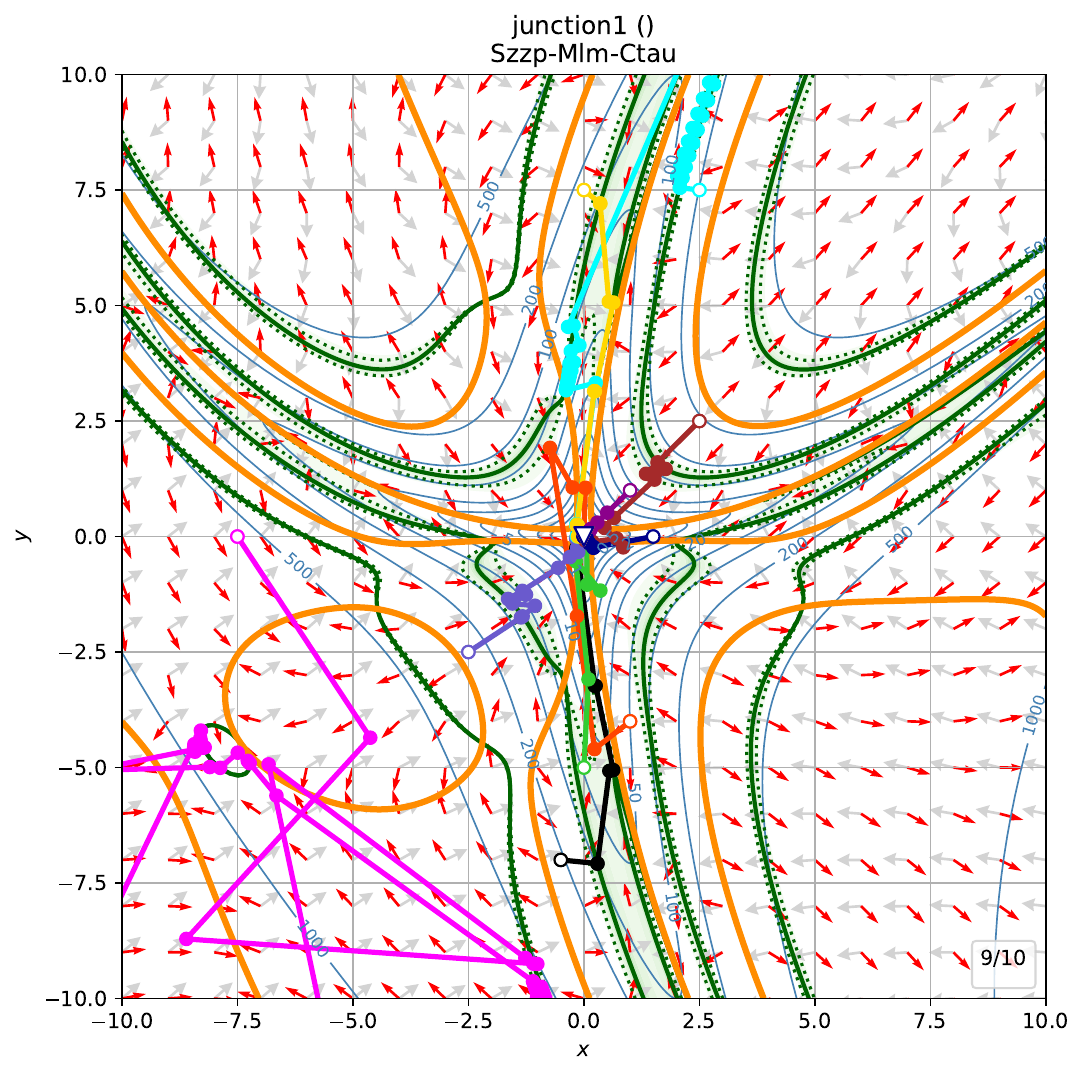}\\      
\caption{}
\label{fig_newton_junction1}
\end{center}
\end{figure}
\textbf{Junction2}: The behavior of the zigzag strategy for this
function is shown in \textbf{Figure~\ref{fig_newton_junction2}}. As in
the junction1 function, the methods succeeds in most cases, but the
overall behavior is unsatisfactory due to extended zigzagging phases in
tight ravines (see e.g. the cyan and slateblue trajectories). This is
exemplified in the top right diagram for the golden trajectory with
originates far from the minimum (bottom right diagram), is rescued twice
by the parallelity check (see the strategy string), but ultimately
manages to reach the minimum.
\begin{figure}[t]
\begin{center}
\includegraphics[height=0.47\textwidth]{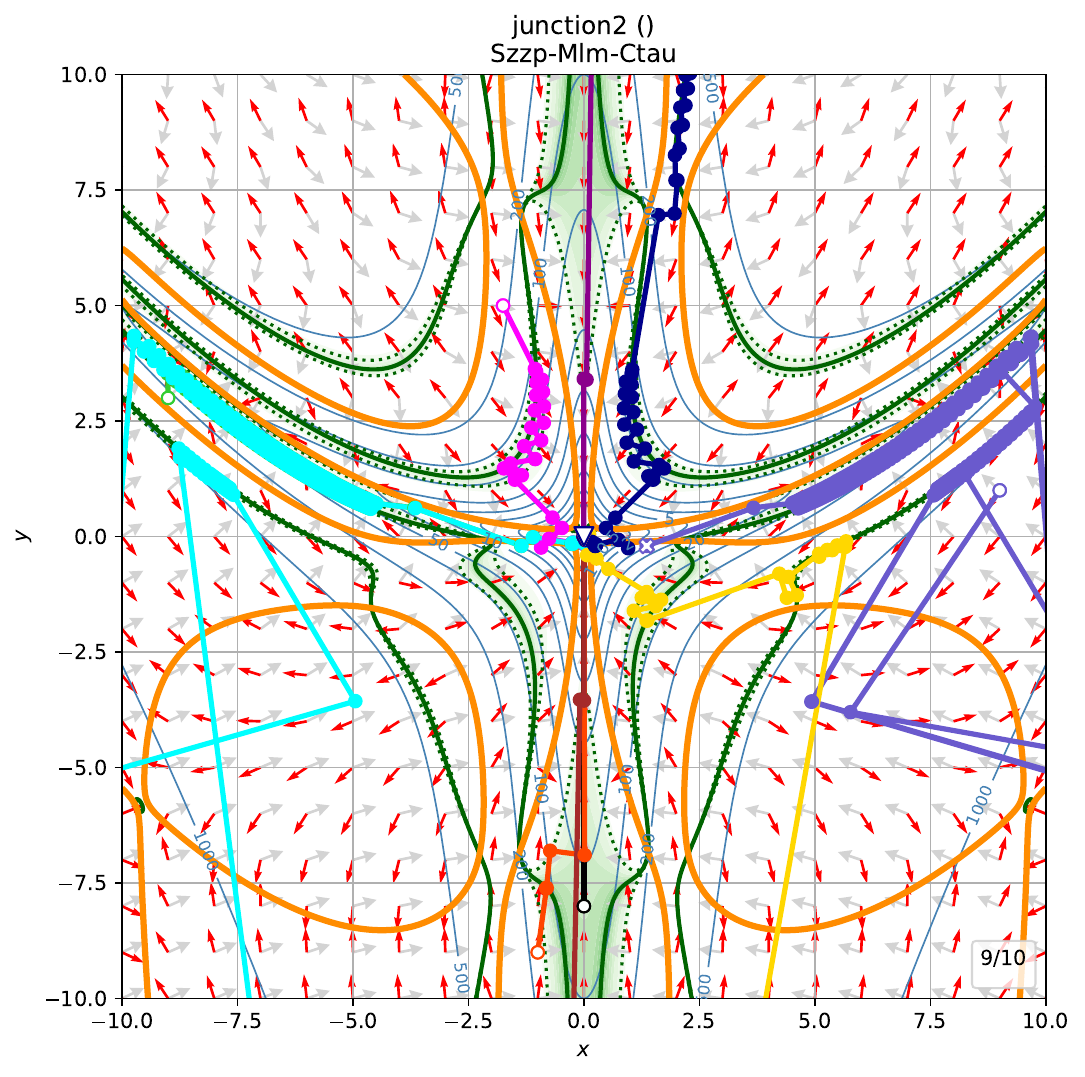}
\hspace*{1mm}
\includegraphics[height=0.47\textwidth]{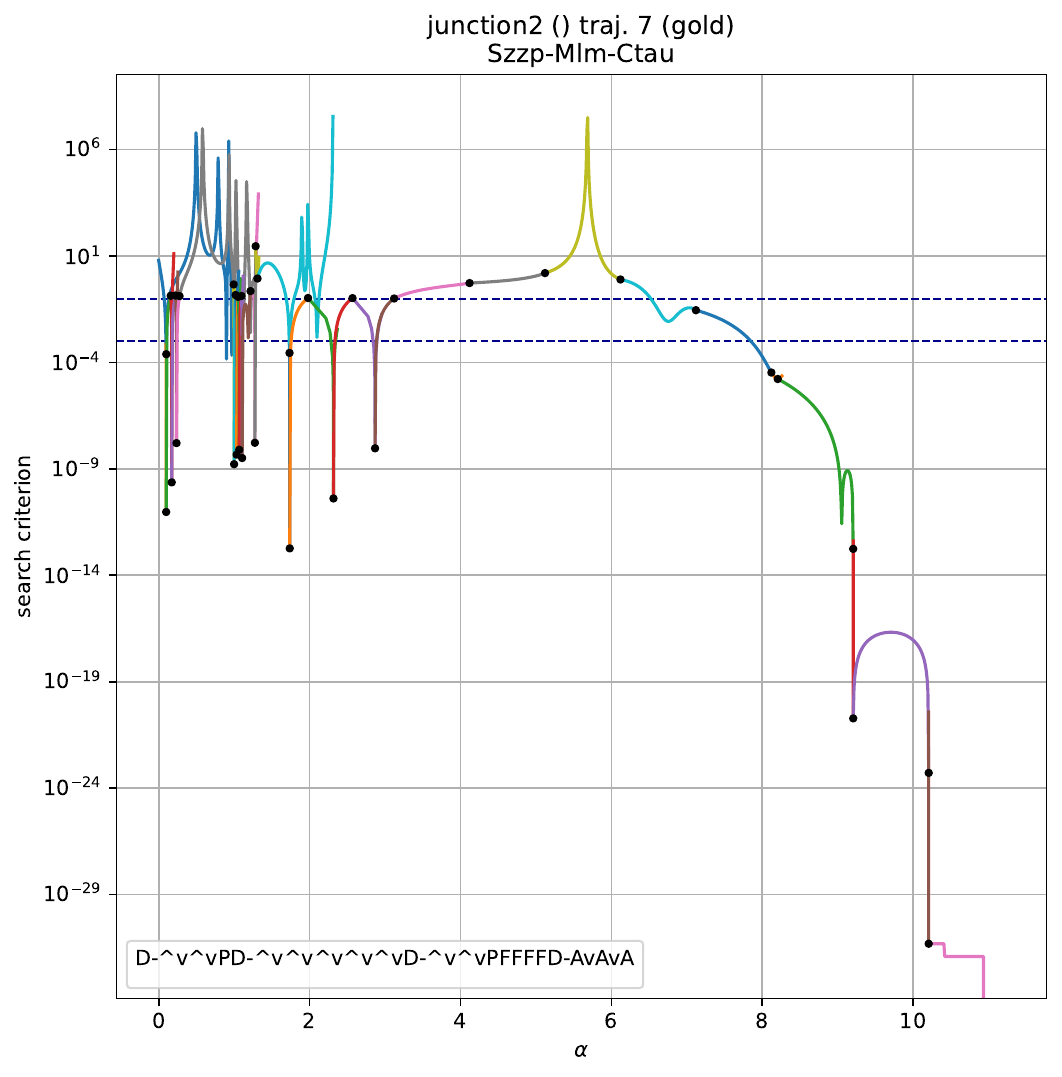}\\
\includegraphics[height=0.47\textwidth]{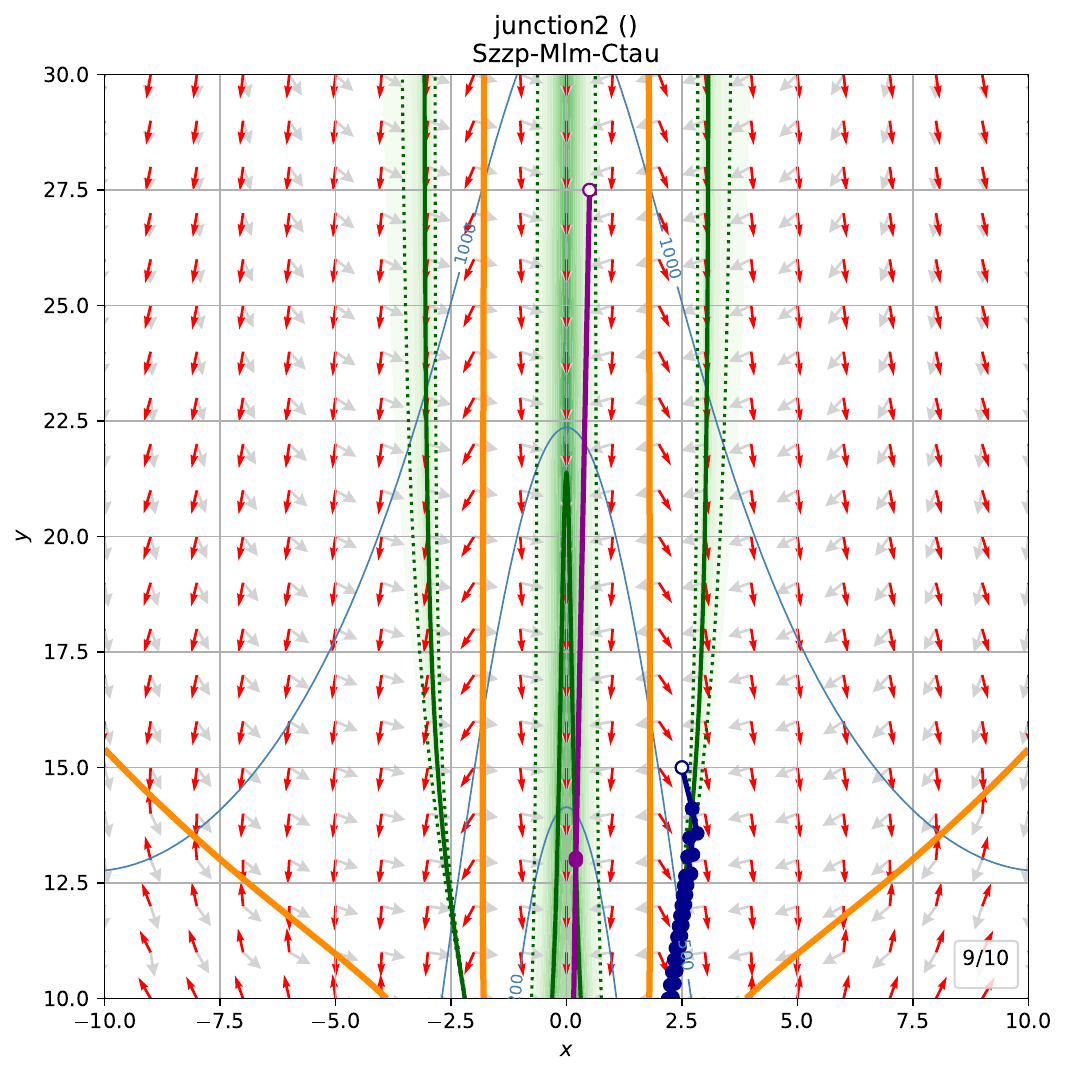}
\hspace*{1mm}
\includegraphics[height=0.47\textwidth]{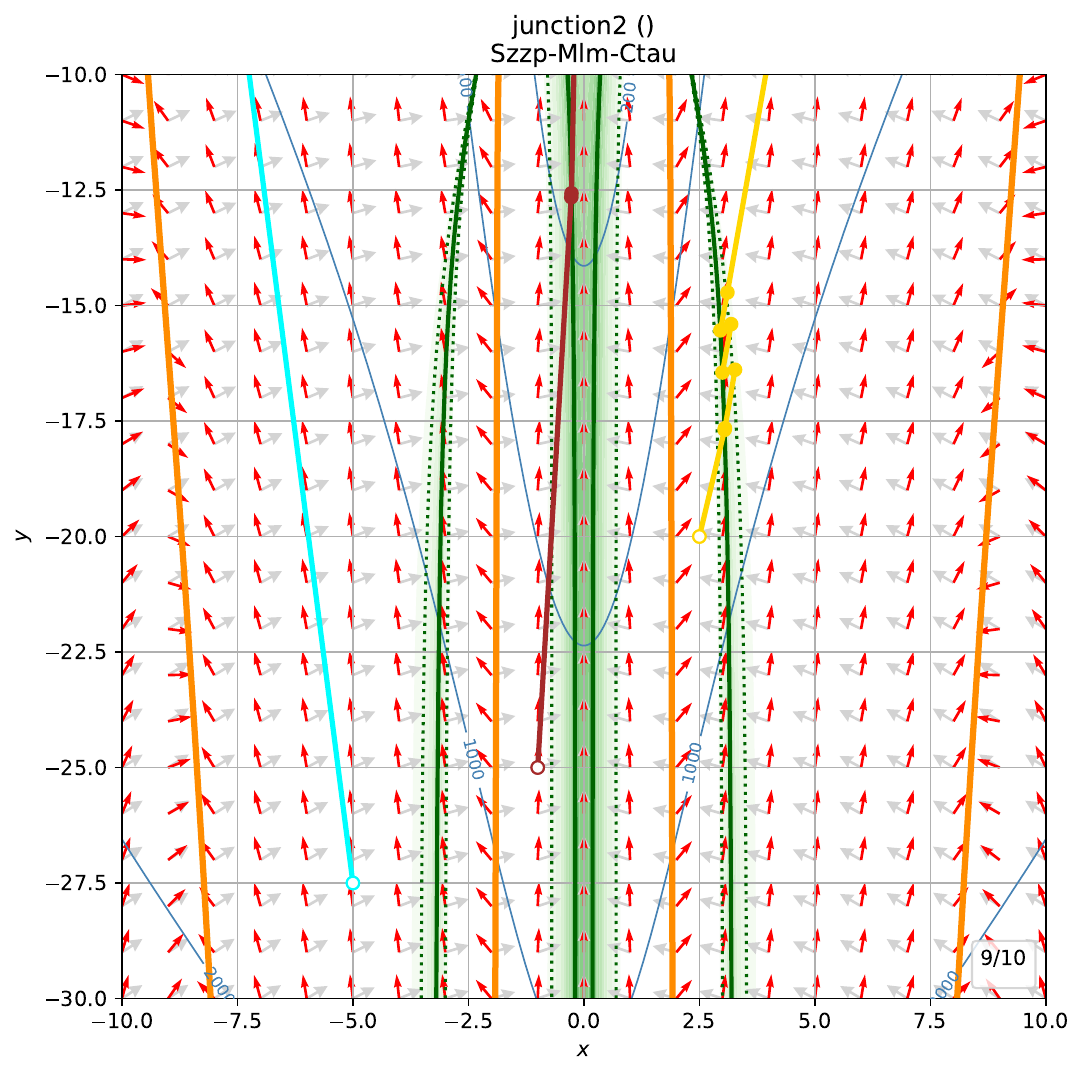}\\      
\caption{}
\label{fig_newton_junction2}
\end{center}
\end{figure}
\textbf{Goldstein-Price}: In contrast to the junction functions where
there extensive zigzagging in tight ravine impedes the progress, the
zigzag method appears to use zigzagging rarely when applied to the
Goldstein-Price function, see \textbf{Figure~\ref{fig_newton_gldstnprc}}
(left diagram). It nevertheless succeeds for all tested starting points
and approaches the minimum at $(0,-1)$ (better visible in the close-up
in the right diagram) as well as four saddle points. As expected, the
complex but rather sparse ravine structure of this function doesn't lend
itself to the application of the zigzag strategy.
\begin{figure}[tp]
\begin{center}
\includegraphics[height=0.47\textwidth]{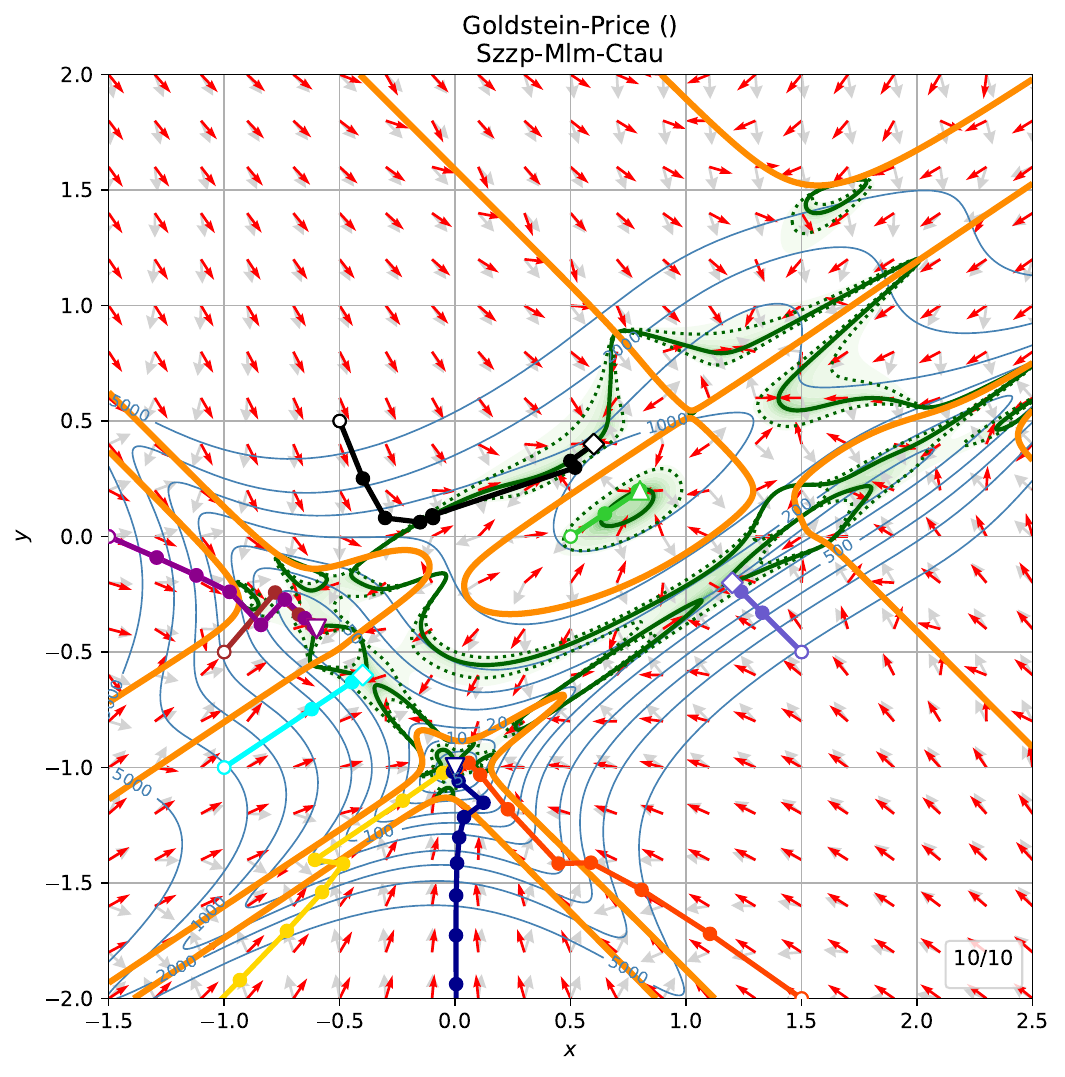}
\hspace*{1mm}
\includegraphics[height=0.47\textwidth]{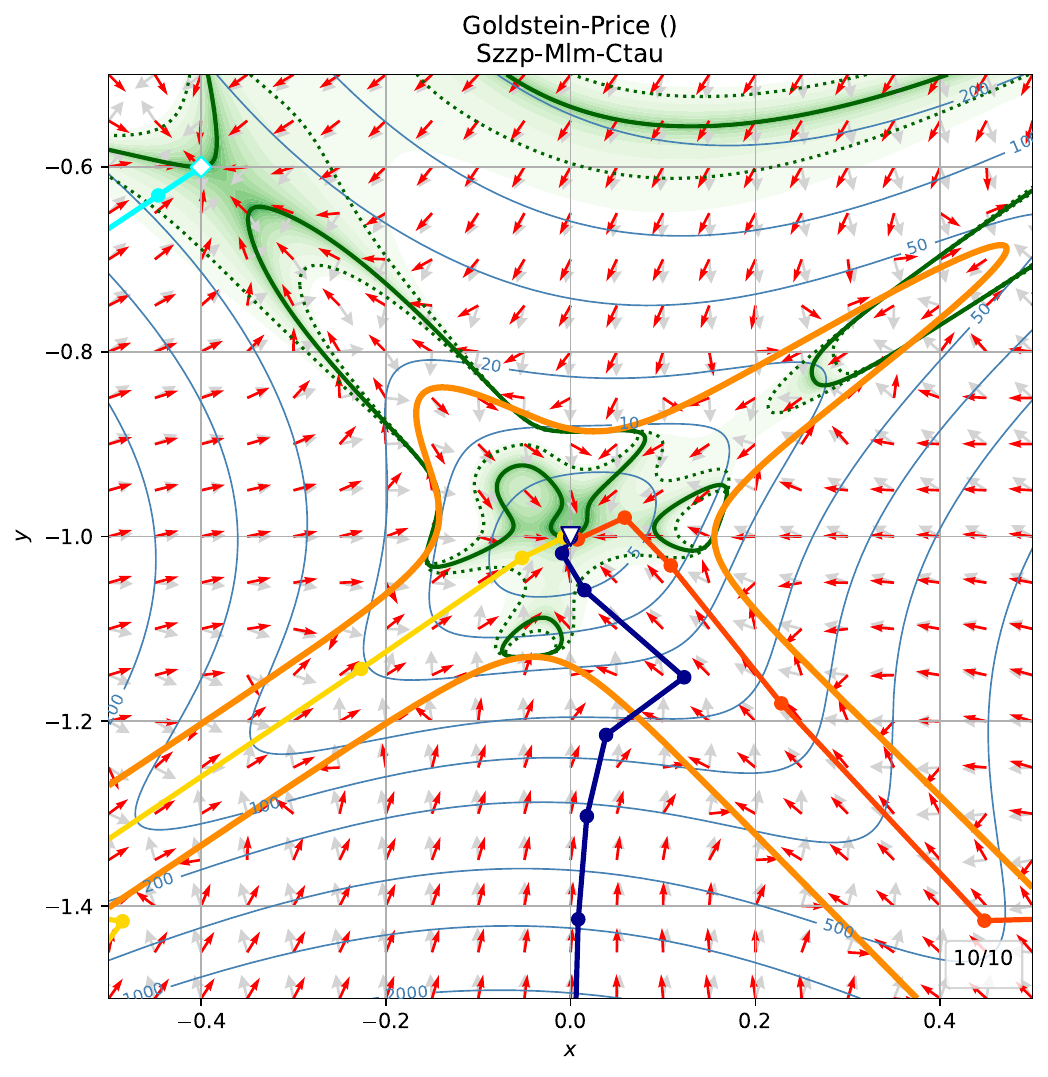}
\caption{}
\label{fig_newton_gldstnprc}
\end{center}
\end{figure}
\textbf{Figure~\ref{fig_newton_beale}} (left diagram): The behavior of
the zigzag strategy in the \textbf{Beale} function partly resembles that
in the junction functions --- unnecessary zigzagging in tight ravines
--- and partly that in the Goldstein-Price function --- the sparse
ravine structure cannot be exploited by the method. The minimum at
$(3,0.5)$ and two saddle points are found by the zigzag method. The
right diagram reveals the reason why the orangered trajectory diverges:
The close-up exhibits a tight ravine (close to a countercurrent
singularity) which blocks the path to a stationary point and leads the
method astray.
\begin{figure}[tp]
\begin{center}
\includegraphics[height=0.47\textwidth]{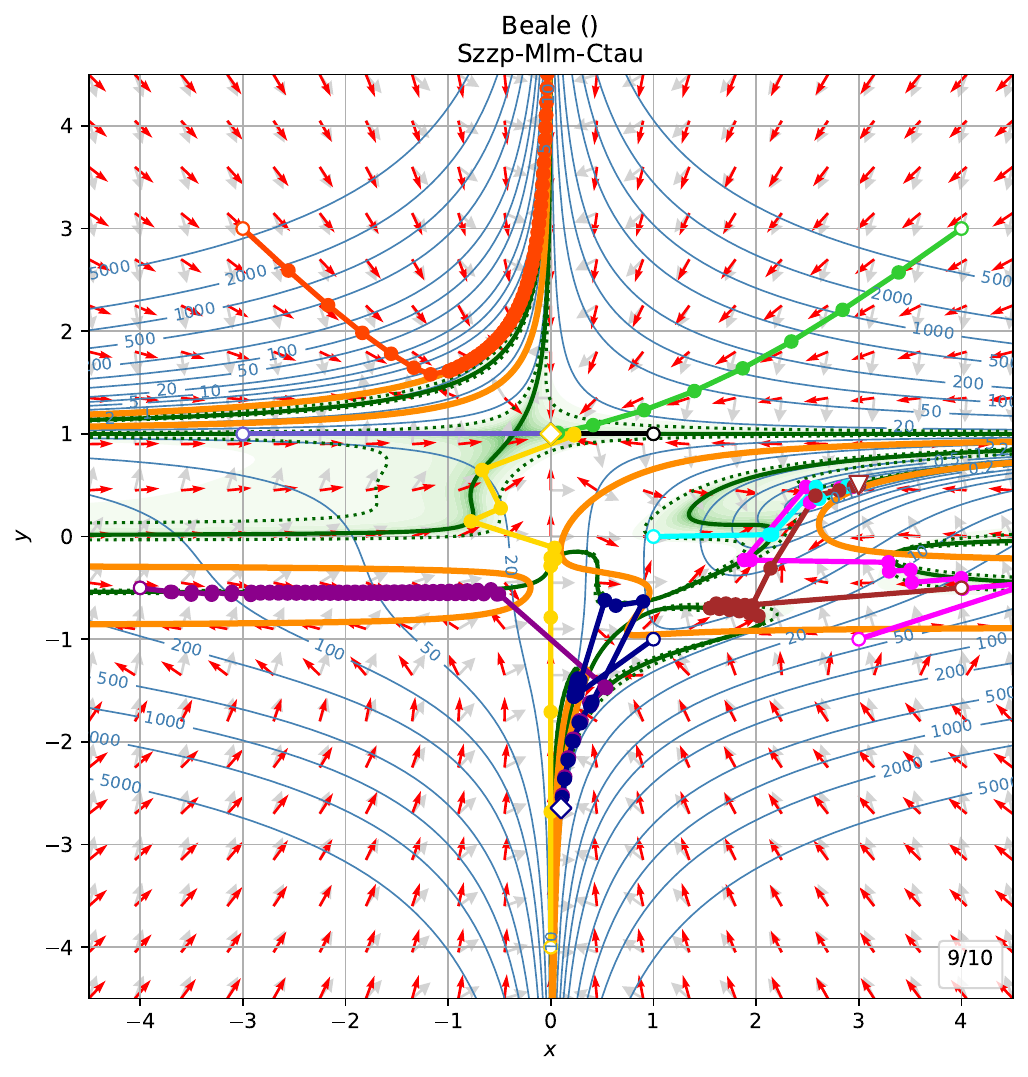}
\hspace*{1mm}
\includegraphics[height=0.47\textwidth]{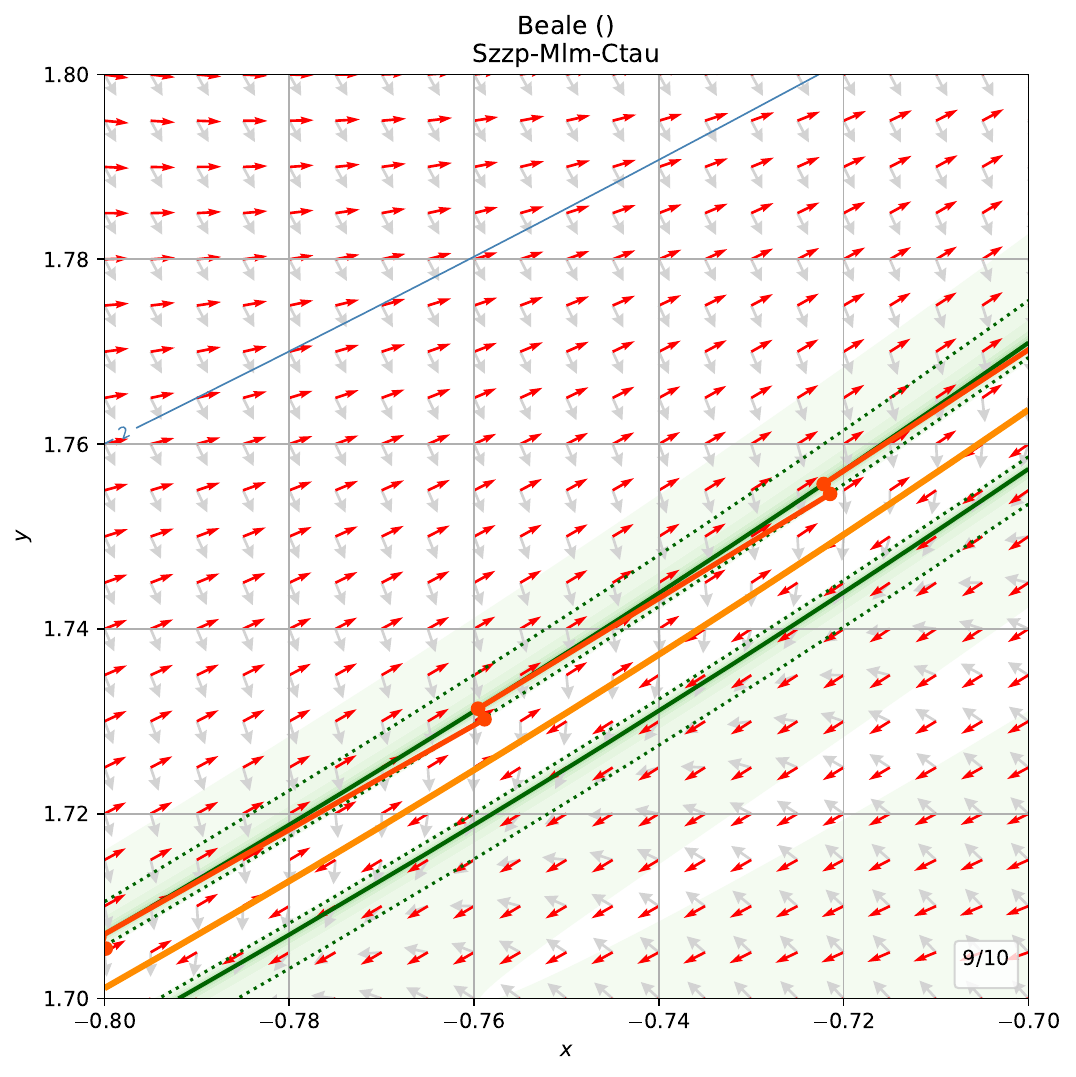}
\caption{}
\label{fig_newton_beale}
\end{center}
\end{figure}

\section{Experiments on a Higher-Dimensional Test Function}\label{sec_exp_highd}
%===============================================================================

%
The only higher-dimensional example (\textbf{eigen}) studied in this
work is based on a Lagrangian which has its stationary points in the
eigenvectors $\mb{w}$ and eigenvalues $\bs{\lambda}$ of a covariance
matrix $\mb{C}$ of dimension $n$. In this case, the Lagrange multiplier
$\bs{\lambda}$ has a specific problem-related meaning --- in the
stationary points, it coincides with the eigenvalue of $\mb{C}$ which
belongs to the eigenvector $\mb{w}$.

The Lagrangian is given by
\begin{align}
L\mleft( \mb{w},\bs{\lambda} \mright)
&=
\frac{1}{2}\, \mb{w}^T \mb{C} \mb{w} + \frac{1}{2}\, \bs{\lambda}
\mleft( 1 - \mb{w}^T \mb{w} \mright)
\end{align}
(the sign of the constraint term is chosen such that, in a stationary
point, the Lagrange multiplier becomes an eigenvalue of $\mb{C}$ rather
than a negative eigenvalue).

The joint state vector on which Newton's method is applied is
\begin{align}
\mb{x}
&=
\begin{pmatrix}\mb{w}\\\bs{\lambda}\end{pmatrix}.
\end{align}
The components of the gradient of the Lagrangian are
\begin{align}
\frac{\partial}{\partial \mb{w}}\,L
&=
\mb{w}^T \mb{C} - \bs{\lambda} \mb{w}^T\\
%====================
\frac{\partial}{\partial \bs{\lambda}}\,L
&=
\frac{1}{2}\, \mleft( 1 - \mb{w}^T \mb{w} \mright).
\end{align}
Expressed as a joint column vector, the gradient is written as
\begin{align}
\mb{g} & = \mleft( \frac{\partial}{\partial \mb{x}}\,L \mright)^T = \begin{pmatrix}\frac{\partial}{\partial \mb{w}}\,L &
\frac{\partial}{\partial \bs{\lambda}}\,L\end{pmatrix}^T = \begin{pmatrix}\mb{C} \mb{w} - \bs{\lambda} \mb{w} \\\frac{1}{2}\,
\mleft( 1 - \mb{w}^T \mb{w} \mright) \end{pmatrix}.
\end{align}
The stationary points found from $\mb{g} = \mb{0}_{n+1}$,
\begin{align}
\mb{C} \mb{w}
&=
\bs{\lambda} \mb{w}\\
%====================
\mb{w}^T \mb{w}
&=
1,
\end{align}
are all pairs of the unit-length eigenvectors and the eigenvalues of
$\mb{C}$.

From the gradient we obtain the components of the Hessian
\begin{align}
\frac{\partial}{\partial \mb{w}}\,\mleft( \frac{\partial}{\partial
\mb{w}}\,L \mright)^T
&=
\mb{C} - \bs{\lambda} \mb{I}_{n}\\
%====================
\frac{\partial}{\partial \bs{\lambda}}\,\mleft( \frac{\partial}{\partial
\mb{w}}\,L \mright)^T
&=
- \mb{w}\\
%====================
\frac{\partial}{\partial \mb{w}}\,\mleft( \frac{\partial}{\partial
\bs{\lambda}}\,L \mright)^T
&=
- \mb{w}^T\\
%====================
\frac{\partial}{\partial \bs{\lambda}}\,\mleft( \frac{\partial}{\partial
\bs{\lambda}}\,L \mright)^T
&=
0
\end{align}
which can be put together into a single matrix of size $(n + 1) \times
(n + 1)$ as
\begin{align}
\mb{H} & = \frac{\partial}{\partial \mb{x}}\,\mb{g} = \begin{pmatrix}\mb{C} - \bs{\lambda} \mb{I}_{n} & - \mb{w}\\- \mb{w}^T &
0\end{pmatrix}.
\end{align}
The third derivatives are obtained from
\begin{align}
\frac{\partial}{\partial w_{k}}\,\mb{H}
&=
\begin{pmatrix}\mb{0}_{n,n} & - \mb{e}_{k}\\- \mb{e}_{k}^T &
0\end{pmatrix} \quad\mbox{for}\; k = 1 \ldots n\\
%====================
\frac{\partial}{\partial \bs{\lambda}}\,\mb{H}
&=
\begin{pmatrix}- \mb{I}_{n} & \mb{0}_n\\\mb{0}_n^T & 0\end{pmatrix}
\end{align}
where $\mb{e}_{k}$ is a unit vector with element $1$ at index $k$.

This is probably a rare example where the third derivatives have a
sparse structure --- this could be exploited to accelerate the
computations of the terms required for the zigzag strategy.

Note that for this simple example we are not studying whether the
stationary points of the Lagrangian approached by Newton's method are
maxima, minima, or saddle points with respect to the unconstrained part
of the Lagrangian $\mb{w}^T \mb{C} \mb{w} $; see e.g.
\citet[p.~337]{nn_Nocedal06} on how this could be accomplished.

In the experiments, we use $n = 10$ dimensions. The covariance matrix is
formed by inverse spectral decomposition from a diagonal matrix and a
random orthogonal matrix, the latter obtained from QR decomposition of a
random matrix. The diagonal matrix has a diagonal starting with $1$ and
doubling the value for all subsequent entries.

The initial vector $\mb{w}$ is obtained by drawing all elements from a
uniform distribution in the range $[-1,1]$. The initial Lagrange
multiplier $\bs{\lambda}$ is drawn from a uniform distribution in the
range $[0, 100]$. The experiment is repeated $10$ times with different
random initialization.

\textbf{Figure~\ref{fig_eigen}} shows the result of a selected
trajectory (left diagram) and a superposition of all $10$ trajectories
(criterion over the index of the Newton iteration, ignoring sub-steps).
\begin{figure}[tp]
\begin{center}
\includegraphics[height=0.47\textwidth]{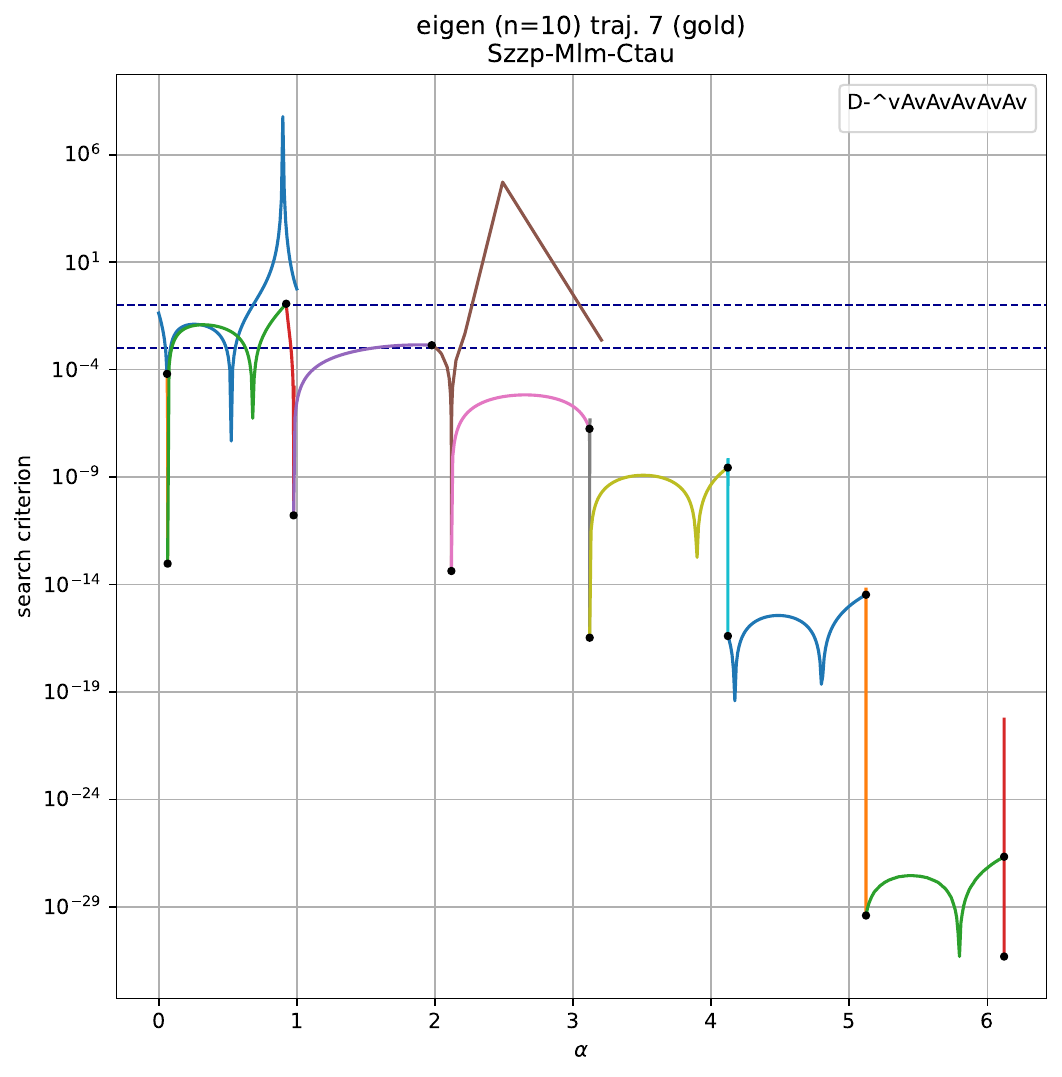}
\hspace*{1mm}
\includegraphics[height=0.47\textwidth]{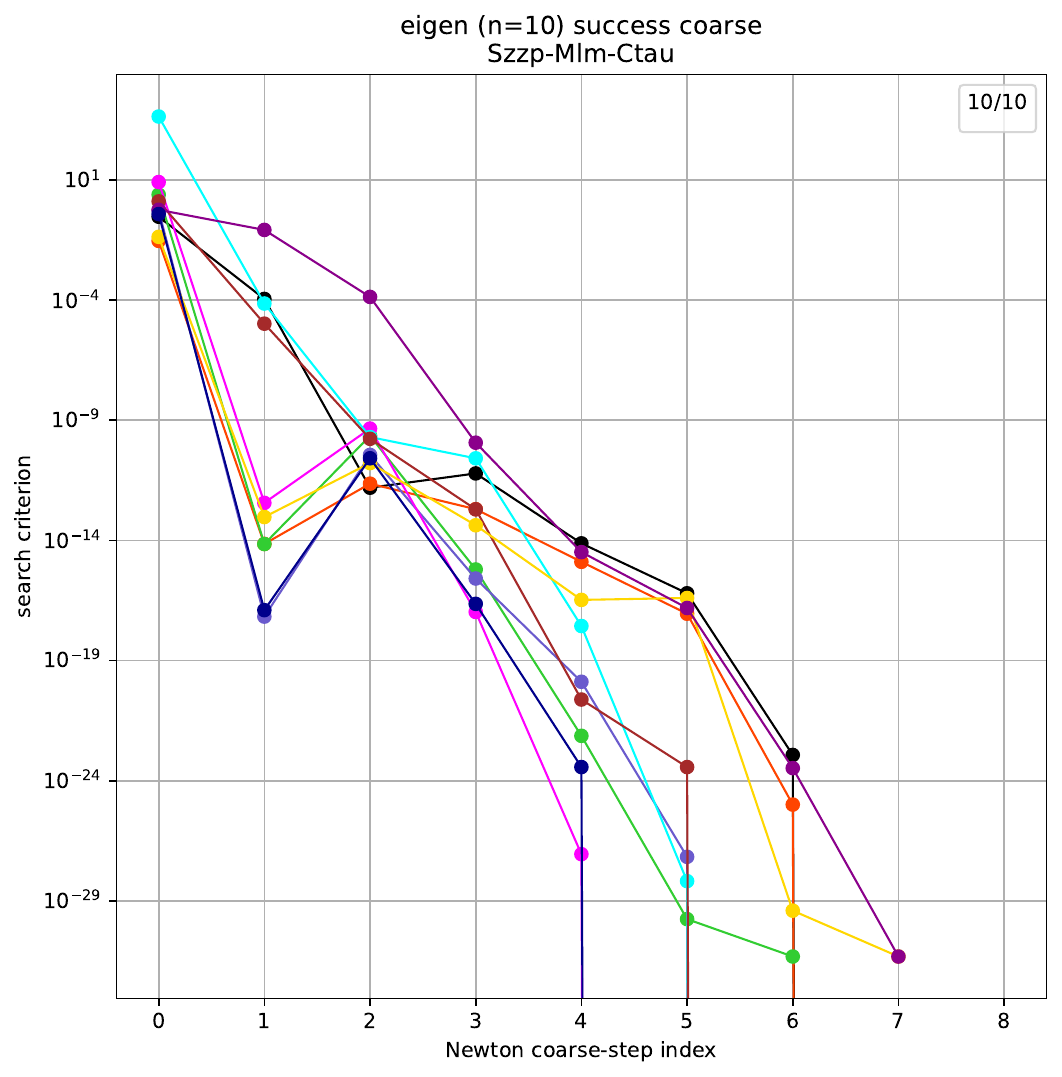}
\caption{}
\label{fig_eigen}
\end{center}
\end{figure}
It is visible that all trajectories successfully converge (right
diagram). However, in the selected trajectory (left diagram), there is
only a single damped zigzag step followed by undamped zigzag steps (and
in all other trajectories there are only undamped steps), so in this
test function there is apparently no need for a line search. At least
the experiment confirms that the down phase can descent into a ravine
hyperplane, and that zigzagging along this hyperplane leads the method
to a stationary point (for the selected trajectory 7, to the stationary
point related to eigenvalue $32$).

\section{Discussion}\label{sec_discussion}
%=========================================

%
Newton's method cannot use value-based line-search when approaching
saddle points. This is particulary problematic for objective functions
coupled to equality constraints as in this case {\em all} stationary
points are saddles. A possible solution is to use a merit function
instead of the objective function. A merit function is composed of terms
relating to the objective function and terms relating to the constraint
function. However, the selection of a suitable merit function from the
large number of known alternatives and the choice of the weighting
parameter between the terms is problematic.\footnote{As I have no
practical experience with merit functions, I decided not to use them for
a comparison with the zigzag method suggested in this work. Apart from
that, this work is at an exploratory stage where such a comparison
appeared to be premature.}

The alternative line-search method developed here mainly addresses the
problem of following deep ``valleys'' where damping of the Newton vector
length is required to stay close to the valley floor. The method uses a
criterion based on the divergence of the field of Newton steps. This
criterion ($\check{\tau}$) is non-negative and has a zero at all
stationary points. While the original idea behind this criterion was a
different one, I discovered that in many cases ``ravines'' extend
outward from the stationary points. These ravines have a flat, zero
bottom. A ``zigzag'' strategy was developed which can enter a nearby
ravine and then follow the ravine along the Newton vectors towards a
stationary point. Whenever the method is about to leave the ravine, a
``pullback'' term moves it back towards the ravine's bottom.

Experiments exploring both the ravine structure and the performance of
the zigzag method where performed for several two-dimensional test
functions and a single higher-dimensional test function (as the latter
is the only example using equality constraints, more experiments are
needed). Looking at the results, it seems neither justified to endorse
the novel line-search method nor to completely reject it. I will first
discuss the criterion and the corresponding ravine structure and then
the proposed zigzag method.

\subsection{Ravine Structure}
%----------------------------

%
The Rosenbrock function is a widely used example to explore the behavior
of optimization methods in deep valleys. Here, a single
$\check{\tau}$-ravine runs along the bottom of the valley (and through
the single stationary point, a minimum). Crucially, this is also the
case for the saddle-point version of Rosenbrock's function. Also another
version, the ``Rosenbrock ditch'' has a ravine at the valley bottom
(passing the minimum). However, even in this simple modification of the
original Rosenbrock function, the ravine structure gets surprisingly
complex, with additional ravines running adjacent to the one at the
valley bottom. To jump between them, a ravine-traversing method may
require special cases.

The Himmelblau function has a clear structure of ravines passing through
the stationary points, but additional tight, looped ravines appear which
are not related to stationary points and may lead a ravine-traversing
method astray. In addition, we see ravines passing through singularities
of the Hessian. This is problematic for traversing the ravine since it
gets steep near the cross-over points.

Also in the Henon-Heiles function has a clear ravine structure, in this
case a single loop. However, one of the attractors is not connected to
the ravine network but an isolated point. Because this point is also
surrounded by the ravine loop, it may be difficult to reach. As in the
Himmelblau function, we see problematic cross-overs between ravines and
Hessian singularities.

In more complex test functions, the ravine structure turns out to be
complex and difficult to predict. Crossing ditches (test functions
junction1 and junction2) lead to a large number of adjacent ravines,
none of which are connected to the single stationary point. Ravines
partly run close to the valley bottoms, but peter out when the approach
the minimum; in some cases, they fold back. In the Goldstein-Price and
Beale functions, there appears to be no clear relationship between
valleys and ravines, such that a ravine-traversing method can only fall
back to plain Newton steps as a default.

This raises the question whether alternatives to the divergence-based
criterion exist where the ravine structure has a stable relationship to
``valleys'' in the objective function.

Ravines have a close relationship to curves defined by singularities of
the Hessian which may be worth exploring. We often observe ravines and
singularity curves running adjacent to each other. Moreover, we noticed
that the two different versions of the pullback directions, one being
related to the ravine, the other to the singularities, are identical. It
may also be interesting to explore the meaning of the two different
singularity types --- ``countercurrent'' and ``inflection''
singularities --- that appear at least in the two-dimensional test
functions.

It is a clear disadvantage of the criterion (and the pullback terms)
that it requires the computation of third-order derivatives. Except in
simple cases (as in the higher-dimensional example ``eigen''), the
computational effort may be forbiddingly large, even when exploiting
symmetries.

\subsection{Zigzag Method}
%-------------------------

%
In order to traverse the ravine network, a zigzag method is suggested.
In all experiments, the method succeeds from the majority of starting
points. It traverses ravines by first approaching a ravine bottom
(down), and then alternatingly following the Newton vector until the
ravine is left (zig) and returning to the ravine bottom (zag). However,
there are cases where the convergence is only achieved by falling back
to default Newton steps, e.g. using the parallelity check when the
zigzagging doesn't lead to progress, or when no ravine can be reached
(e.g. Goldstein-Price). This makes the zigzag method relatively complex.
Moreover, entering a ravine requires a complex line-search method where
a coarse explicit line search, a step finding local minima, and a
refinement step using Golden Section search need to be combined. Without
the refinement, the method may miss ravines.

It is an advantage of the method compared to the original Newton method
that it prevents large jumps; instead, the first or an early step in the
down phase is damped. However, except in the Rosenbrock examples, the
following zigzag steps are often not damped, and if they are, this
sometimes indicates slow progress in tight ravines (e.g. Henon-Heiles,
junction1, junction2, Goldstein-Price).

In the final zigzag movements, the value of the divergence criterion
decreases fast. It is however, not clear to me why such low values are
not reached after earlier zag steps (as the ravine bottom is flat) ---
this may be related to the restricted number of steps permitted for the
Golden Section search or to the relatively high tolerance, or to the
shorter Newton steps close to the stationary points.

The zigzag method and the line-search method for the down phase have a
relatively high number of parameters (entry and escape threshold, number
of search steps, tolerances, search brackets, rejection thresholds,
parallelity check). However, the selection of the parameters appears to
be uncritical, as the same set of parameters could be used for all test
functions. It might be useful to randomly sample the value of the
divergence-based criterion to see whether the values outside of ravines
are in the same range for all test functions.

A number of modifications could improve the performance of the zigzag
method. It may be beneficial to check the steepness of the ravine walls
and avoid zigzagging along tight ravines. In the zig phase, the
criterion seems to be well-behaved, so the explicit search could
possibly be replaced with a more efficient optimization method. In cases
where the escape threshold is not reached in the zig phase, using the
nearest local minimum (instead of performing an undamped step) may
accelerate the approach to the stationary point.

\section{Conclusions}\label{sec_conclusions}
%===========================================

%
The line-search method suggested in this work is capable of approaching
saddle points, thus allowing its application for equality-constrained
optimization using Newton's method. It succeeds in many test functions,
but only exhibits advantages in some of them. The computational effort
compared to using a merit function is considerably higher. Nevertheless,
the idea of a traversing a ravine network may be attractive and justify
a search for alternatives to the specific, divergence-based criterion.

\section{Software}\label{sec_software}
%=====================================

%
The software used for the simulations in this work is available from
\url{https://www.ti.uni-bielefeld.de/html/downloads}. It comprises a
Python Jupyter notebook, a Python program generated from the notebook
(which can be done via a Makefile that is also provided), and a Tcl
script for systematically running a batch of simulations (for different
objective functions and different methods).

If you use the software, I would appreciate a citation of this report.
%
%

%

%###########################################################################
\appendix
%###########################################################################

\section{Proofs}
%===============

\subsection{Row Vector of a Matrix Product}\label{sec_rowvec_prod}
%-----------------------------------------------------------------

%
Given indices $n$: $1 \leq n$; $m$: $1 \leq m$; $N$: $1 \leq N$ and
matrices $\mb{A}$: arbitrary, $n \times m$; $\mb{B}$: arbitrary, $m
\times N$, we have
\begin{align}
&\phantom{{}={}} \mleft( \mb{A} \mb{B} \mright)_{i,*} \quad\mbox{for}\; i = 1 \ldots n \nonumber \\
%====================
& = \mleft( \mleft( \mb{A} \mb{B} \mright)_{i,j} \mright)^{1\times N}_{1,j}\\
%====================
& = \mleft( \mleft( \mb{A} \mright)_{i,*} \mleft( \mb{B} \mright)_{*,j}
\mright)^{1\times N}_{1,j}\\
%====================
& = \mleft( \mb{A} \mright)_{i,*} \mleft( \mleft( \mb{B} \mright)_{*,j}
\mright)^{m\times N}_{*,j}\\
%====================
&\label{eq_ABix}
 = \mleft( \mb{A} \mright)_{i,*} \mb{B}.
\end{align}

\subsection{Different Notation for Matrix Product}\label{sec_mat_prod}
%---------------------------------------------------------------------

%
Given $n$: $1 \leq n$; $m$: $1 \leq m$; $N$: $1 \leq N$ and matrices
$\mb{A}$: arbitrary, $n \times m$; $\mb{B}$: arbitrary, $m \times N$, we
have
\begin{align}
\label{eq_mat_prod}
\mb{A} \mb{B}
&=
\mleft( \mleft( \mb{A} \mright)_{i,*} \mleft( \mb{B} \mright)_{*,j}
\mright)^{n\times N}_{i,j}.
\end{align}

\subsection{Invariance of Divergence under Rotation}\label{sec_inv_div_rot}
%--------------------------------------------------------------------------

%
We want to prove that the divergence is invariant under a rotation of
the vector field. For that we first need to describe the effect of a
rotation.\footnote{This part is taken from
\url{https://math.stackexchange.com/a/2490680}, thanks to user
\texttt{gim461}.} We start from the original vector field $\mb{v}\mleft(
\mb{x} \mright)$. Given a rotation matrix $\mb{R}$ (an orthogonal
matrix; to be strict we need to specify $\opnl{det} \mleft\{ \mb{R}
\mright\} = 1$ to exclude reflections, but this doesn't play any role in
the proof). Rotating a vector field actually requires two rotations: Not
only the direction vectors of the field need to be rotated, but also the
locations where these vectors originate. Therefore, the rotated vector
field $\mb{u}\mleft( \mb{x} \mright)$ is obtained from
\begin{align}
\mb{u}\mleft( \mb{x} \mright)
&=
\mb{R}^T \mb{v}\mleft( \mb{R} \mb{x} \mright)
\end{align}
where we look up the location vector in the original field at the
rotated coordinates
\begin{align}
\mb{z}
&=
\mb{R} \mb{x}
\end{align}
but then need to rotate the corresponding direction vector in the
opposite direction:
\begin{align}
\mb{u}
&=
\mb{R}^T \mb{v}.
\end{align}
Now we analyze the effect of the vector field rotation on the
divergence. We use the trace notation of the divergence
\begin{align}
\opnl{div} \mleft\{ \mb{v}\mleft( \mb{x} \mright) \mright\}
&=
\opnl{tr} \mleft\{ \frac{\partial}{\partial \mb{x}}\,\mb{v}\mleft(
\mb{x} \mright) \mright\}
\end{align}
and obtain
\begin{align}
&\phantom{{}={}} \opnl{div} \mleft\{ \mb{R}^T \mb{v}\mleft( \mb{R} \mb{x} \mright)
\mright\} \nonumber \\
%====================
& = \opnl{tr} \mleft\{ \frac{\partial}{\partial \mb{x}}\,\mleft( \mb{R}^T
\mb{v}\mleft( \mb{R} \mb{x} \mright) \mright) \mright\}\\
%====================
\shortintertext{with derivative \eqref{eq_ddx_Abx}}
%====================
& = \opnl{tr} \mleft\{ \mb{R}^T \mleft( \frac{\partial}{\partial
\mb{x}}\,\mb{v}\mleft( \mb{R} \mb{x} \mright) \mright) \mright\}\\
%====================
\shortintertext{with chain rule \eqref{eq_chainrule2}}
%====================
& = \opnl{tr} \mleft\{ \mb{R}^T \mleft( \frac{\partial}{\partial
\mb{z}}\,\mb{v}\mleft( \mb{z} \mright) \mright) \mleft(
\frac{\partial}{\partial \mb{x}}\,\mb{z} \mright) \mright\}\\
%====================
& = \opnl{tr} \mleft\{ \mb{R}^T \mleft( \frac{\partial}{\partial
\mb{z}}\,\mb{v}\mleft( \mb{z} \mright) \mright) \mleft(
\frac{\partial}{\partial \mb{x}}\,\mleft[ \mb{R} \mb{x} \mright]
\mright) \mright\}\\
%====================
\shortintertext{with derivative \eqref{eq_ddx_Abx}}
%====================
& = \opnl{tr} \mleft\{ \mb{R}^T \mleft( \frac{\partial}{\partial
\mb{z}}\,\mb{v}\mleft( \mb{z} \mright) \mright) \mb{R} \mright\}\\
%====================
\shortintertext{exploiting the cyclic property of the trace}
%====================
& = \opnl{tr} \mleft\{ \mb{R} \mb{R}^T \mleft( \frac{\partial}{\partial
\mb{z}}\,\mb{v}\mleft( \mb{z} \mright) \mright) \mright\}\\
%====================
\shortintertext{considering that $\mb{R}$ is orthogonal}
%====================
& = \opnl{tr} \mleft\{ \frac{\partial}{\partial \mb{z}}\,\mb{v}\mleft(
\mb{z} \mright) \mright\}\\
%====================
& = \opnl{div} \mleft\{ \mb{v}\mleft( \mb{z} \mright) \mright\}
\end{align}
which proves that the divergence is invariant under rotations.

\subsection{Equivalence of Two Pullback Directions}\label{sec_pullback_equiv}
%----------------------------------------------------------------------------

%
Here we show that the two approaches to compute the pullback directions
coincide. We look at the vector derived from singularity curve. Element
$k$ of the un-normalized vector \eqref{eq_qt} is
\begin{align}
\label{eq_qtk}
\tilde{q}_{k}
&=
\sum\limits_{i = 1}^{n}\mleft( \mb{H}^{-1} \mright)_{i,*} \mleft(
\frac{\partial}{\partial x_{k}}\,\mb{H} \mright)_{*,i}.
\end{align}
We now try to transform the vector derived from the approximation of the
gradient of $\tau$ towards the expression above. Element $k$ of the
un-normalized vector \eqref{eq_pt} is
\begin{align}
&\phantom{{}={}} \tilde{p}_{k} \nonumber \\
%====================
& = \mleft( \sum\limits_{i = 1}^{n}\mleft[ \frac{\partial}{\partial
x_{i}}\,\mb{H} \mright] \mleft( \mb{H}^{-1} \mright)_{*,i} \mright)_{k}\\
%====================
& = \sum\limits_{i = 1}^{n}\mleft( \frac{\partial}{\partial x_{i}}\,\mb{H}
\mright)_{k,*} \mleft( \mb{H}^{-1} \mright)_{*,i}\\
%====================
&\label{eq_ptk}
 = \sum\limits_{i = 1}^{n}\mleft( \mb{H}^{-1} \mright)_{i,*} \mleft(
\frac{\partial}{\partial x_{i}}\,\mb{H} \mright)_{*,k}
\end{align}
where the last step exploited the symmetry of the Hessian and its
derivative.

We see that equations \eqref{eq_qtk} and \eqref{eq_ptk} only differ in
the second factor (a column vector). We analyze element $j$ for both
cases
\begin{align}
\mleft( \mleft( \frac{\partial}{\partial x_{k}}\,\mb{H} \mright)_{*,i}
\mright)_{j} & = \mleft( \frac{\partial}{\partial x_{k}}\,\mb{H} \mright)_{j,i} = \frac{\partial}{\partial x_{k}}\,\mleft( \mb{H} \mright)_{j,i} = \frac{\partial}{\partial x_{k}}\,\mleft( \frac{\partial}{\partial
x_{i}}\,\mleft[ \frac{\partial}{\partial x_{j}}\,f \mright] \mright)\\
%====================
\mleft( \mleft( \frac{\partial}{\partial x_{i}}\,\mb{H} \mright)_{*,k}
\mright)_{j} & = \mleft( \frac{\partial}{\partial x_{i}}\,\mb{H} \mright)_{j,k} = \frac{\partial}{\partial x_{i}}\,\mleft( \mb{H} \mright)_{j,k} = \frac{\partial}{\partial x_{i}}\,\mleft( \frac{\partial}{\partial
x_{k}}\,\mleft[ \frac{\partial}{\partial x_{j}}\,f \mright] \mright)
\end{align}
and see that the last terms coincide if the assumptions of Schwarz's
theorem are fulfilled.

\section{Derivatives}
%====================

%
For derivatives, this work uses the ``numerator layout''; see e.g.
\citet[p.130]{nn_Deisenroth20} and \citet[p.172]{nn_Aggarwal20}. In the
numerator layout, the derivative of a scalar with respect to a column
vector is a row vector \citep[see e.g.][p.127]{nn_Deisenroth20}.

\subsection{Chain Rules}
%-----------------------

\subsubsection{Chain Rule for Vectors with Scalar Result}
%''''''''''''''''''''''''''''''''''''''''''''''''''''''''

%
This first version of the chain rule concerns the application of a
scalar-valued function $g$ of a vector to a vector-valued function
$\mb{f}$ of a vector \citep[see e.g.][p.129]{nn_Deisenroth20}.

Dimensions are $n$: $1 \leq n$; $m$: $1 \leq m$; variables and functions
are $\mb{x}$: arbitrary, $n \times 1$; $\mb{z}$: arbitrary, $m \times
1$; $\mb{f}$: function, arbitrary, $m \times 1$; $g$: function, scalar,
$1 \times 1$; $y$: scalar, $1 \times 1$:
\begin{align}
\mb{z}
&=
\mb{f}\mleft( \mb{x} \mright)\\
%====================
y
&=
g\mleft( \mb{z} \mright)\\
%====================
\label{eq_chainrule1}
\frac{\partial}{\partial \mb{x}}\,g
&=
\mleft( \frac{\partial}{\partial \mb{z}}\,g \mright) \mleft(
\frac{\partial}{\partial \mb{x}}\,\mb{f} \mright).
\end{align}
Note that the first factor (a gradient) is a row vector, and the second
factor is a Jacobian matrix; their product is a row vector.

\subsubsection{Chain Rule for Vectors with Vector Result}
%''''''''''''''''''''''''''''''''''''''''''''''''''''''''

%
This second (and more general) version of the chain rule concerns the
application of a vector-valued function $\mb{g}$ of a vector to a
vector-valued function $\mb{f}$ of a vector \citep[see
e.g.][p.175]{nn_Aggarwal20}.

Dimensions are $n$: $1 \leq n$; $m$: $1 \leq m$; $N$: $1 \leq N$;
variables and functions are $\mb{x}$: arbitrary, $N \times 1$; $\mb{z}$:
arbitrary, $m \times 1$; $\mb{y}$: arbitrary, $n \times 1$; $\mb{f}$:
function, arbitrary, $m \times 1$; $\mb{g}$: function, arbitrary, $n
\times 1$:
\begin{align}
\mb{z}
&=
\mb{f}\mleft( \mb{x} \mright)\\
%====================
\mb{y}
&=
\mb{g}\mleft( \mb{z} \mright)\\
%====================
\label{eq_chainrule2}
\frac{\partial}{\partial \mb{x}}\,\mb{g}
&=
\mleft( \frac{\partial}{\partial \mb{z}}\,\mb{g} \mright) \mleft(
\frac{\partial}{\partial \mb{x}}\,\mb{f} \mright).
\end{align}

\subsection{Derivatives of Matrix-Vector Products}\label{sec_deriv_matvec}
%-------------------------------------------------------------------------

\subsubsection{Constant Matrix, Variable Vector}
%'''''''''''''''''''''''''''''''''''''''''''''''

%
Given dimensions $n$: $1 \leq n$ and matrices $\mb{x}$: arbitrary, $n
\times 1$; $\mb{A}$: constant, square, $n \times n$; $\mb{b}\mleft(
\mb{x} \mright)$: function of $\mb{x}$, arbitrary, $n \times 1$, we have
\begin{align}
&\phantom{{}={}} \frac{\partial}{\partial \mb{x}}\,\mleft( \mb{A} \mb{b}\mleft( \mb{x}
\mright) \mright) \nonumber \\
%====================
& = \mleft( \frac{\partial}{\partial x_{j}}\,\mleft( \mb{A} \mb{b}\mleft(
\mb{x} \mright) \mright)_{i} \mright)^{n\times n}_{i,j}\\
%====================
\shortintertext{omitting argument $\mb{x}$}
%====================
& = \mleft( \frac{\partial}{\partial x_{j}}\,\sum\limits_{k = 1}^{n}A_{i,k}
b_{k} \mright)^{n\times n}_{i,j}\\
%====================
& = \mleft( \sum\limits_{k = 1}^{n}A_{i,k} \mleft[ \frac{\partial}{\partial
x_{j}}\,b_{k} \mright] \mright)^{n\times n}_{i,j}\\
%====================
& = \mleft( \sum\limits_{k = 1}^{n}A_{i,k} \mleft( \frac{\partial}{\partial
\mb{x}}\,\mb{b} \mright)_{k,j} \mright)^{n\times n}_{i,j}\\
%====================
& = \mleft( \mleft( \mb{A} \mleft[ \frac{\partial}{\partial \mb{x}}\,\mb{b}
\mright] \mright)_{i,j} \mright)^{n\times n}_{i,j}\\
%====================
&\label{eq_ddx_Abx}
 = \mb{A} \mleft( \frac{\partial}{\partial \mb{x}}\,\mb{b} \mright).
\end{align}

\subsubsection{Variable Matrix, Constant Vector}
%'''''''''''''''''''''''''''''''''''''''''''''''

%
Version for transposed matrix: Given dimensions and indices $n$: $1 \leq
n$; $m$: $1 \leq m$; $N$: $1 \leq N$; $i$: $1 \leq i \leq n$ and
matrices $\mb{x}$: arbitrary, $N \times 1$; $\mb{A}\mleft( \mb{x}
\mright)$: function of $\mb{x}$, arbitrary, $m \times n$;
$\mb{a}_{i}\mleft( \mb{x} \mright)$: function of $\mb{x}$, arbitrary, $m
\times 1$, column vector of $\mb{A}$; $\mb{b}$: constant, arbitrary, $m
\times 1$, we have
\begin{align}
&\phantom{{}={}} \mleft( \frac{\partial}{\partial \mb{x}}\,\mleft[ \mleft\{ \mb{A}\mleft(
\mb{x} \mright) \mright\}^T \mb{b} \mright] \mright)_{i,*} \nonumber \\
%====================
& = \frac{\partial}{\partial \mb{x}}\,\mleft( \mleft[ \mb{A}\mleft( \mb{x}
\mright) \mright]^T \mb{b} \mright)_{i}\\
%====================
& = \frac{\partial}{\partial \mb{x}}\,\mleft( \mleft[ \mb{a}_{i}\mleft(
\mb{x} \mright) \mright]^T \mb{b} \mright)\\
%====================
&\label{eq_ddx_AxTb}
 = \mb{b}^T \mleft( \frac{\partial}{\partial \mb{x}}\,\mb{a}_{i}\mleft(
\mb{x} \mright) \mright).
\end{align}

\subsubsection{Variable Matrix, Variable Vector}
%'''''''''''''''''''''''''''''''''''''''''''''''

%
Given $n$: $1 \leq n$ and matrices $\mb{x}$: arbitrary, $n \times 1$;
$\mb{A}\mleft( \mb{x} \mright)$: function of $\mb{x}$, square, $n \times
n$; $\mb{a}_{i}\mleft( \mb{x} \mright)$: function of $\mb{x}$,
arbitrary, $n \times 1$, column vector of $\mb{A}$; $\mb{b}\mleft(
\mb{x} \mright)$: function of $\mb{x}$, arbitrary, $n \times 1$, we have
\begin{align}
&\phantom{{}={}} \frac{\partial}{\partial \mb{x}}\,\mleft( \mb{A}\mleft( \mb{x} \mright)
\mb{b}\mleft( \mb{x} \mright) \mright) \nonumber \\
%====================
& = \mleft( \frac{\partial}{\partial x_{j}}\,\mleft( \mb{A}\mleft( \mb{x}
\mright) \mb{b}\mleft( \mb{x} \mright) \mright)_{i} \mright)^{n\times
n}_{i,j}\\
%====================
\shortintertext{omitting argument $\mb{x}$}
%====================
& = \mleft( \frac{\partial}{\partial x_{j}}\,\sum\limits_{k = 1}^{n}A_{i,k}
b_{k} \mright)^{n\times n}_{i,j}\\
%====================
& = \mleft( \sum\limits_{k = 1}^{n}\mleft[ \frac{\partial}{\partial
x_{j}}\,A_{i,k} \mright] b_{k} + \sum\limits_{k = 1}^{n}A_{i,k} \mleft[
\frac{\partial}{\partial x_{j}}\,b_{k} \mright] \mright)^{n\times
n}_{i,j}\\
%====================
& = \mleft( \sum\limits_{k = 1}^{n}\mleft( \frac{\partial}{\partial
x_{j}}\,\mb{A} \mright)_{i,k} b_{k} + \sum\limits_{k = 1}^{n}A_{i,k}
\mleft( \frac{\partial}{\partial \mb{x}}\,\mb{b} \mright)_{k,j}
\mright)^{n\times n}_{i,j}\\
%====================
& = \mleft( \mleft( \mleft[ \frac{\partial}{\partial x_{j}}\,\mb{A} \mright]
\mb{b} \mright)_{i} + \mleft( \mb{A} \mleft[ \frac{\partial}{\partial
\mb{x}}\,\mb{b} \mright] \mright)_{i,j} \mright)^{n\times n}_{i,j}\\
%====================
&\label{eq_ddx_Axbx}
 = \mleft( \mleft( \frac{\partial}{\partial x_{j}}\,\mb{A} \mright)_{i,*}
\mb{b} \mright)^{n\times n}_{i,j} + \mb{A} \mleft(
\frac{\partial}{\partial \mb{x}}\,\mb{b} \mright).
\end{align}

\subsection{Derivative of Inverse Matrix}\label{sec_inv_deriv}
%-------------------------------------------------------------

%
\citep[][p.364]{nn_Abadir05} \footnote{I follow
\url{https://atmos.washington.edu/~dennis/MatrixCalculus.pdf} since the
notation is similar to the one used here.}

Given $n$: $1 \leq n$ and $\mb{A}$: function of $\alpha$, square, $n
\times n$; $\alpha$: scalar, $1 \times 1$.
\begin{align}
\mb{A}^{-1} \mb{A}
&=
\mb{I}_{n}\\
%====================
\shortintertext{differentiate}
%====================
\mb{A}^{-1} \mleft( \frac{\partial}{\partial \alpha}\,\mb{A} \mright) +
\mleft( \frac{\partial}{\partial \alpha}\,\mb{A}^{-1} \mright) \mb{A}
&=
\mb{0}_{n,n}\\
%====================
\shortintertext{rearrange terms}
%====================
\label{eq_inv_deriv}
\frac{\partial}{\partial \alpha}\,\mb{A}^{-1}
&=
- \mb{A}^{-1} \mleft( \frac{\partial}{\partial \alpha}\,\mb{A} \mright)
\mb{A}^{-1}.
\end{align}

\section{Test Functions}\label{sec_testfct}
%==========================================

\subsection{Rosenbrock Function}\label{sec_rosenbrock}
%-----------------------------------------------------

%
The Rosenbrock function --- taken second-hand from
\citet[p.~14]{nn_Alt11} and the Wikipedia
page\footnote{\url{https://en.wikipedia.org/wiki/Rosenbrock_function}}
--- is essentially a distorted quadratic function. We can derive the
function from the purely quadratic function
\begin{align}
f\mleft( x',y' \mright) & = \begin{pmatrix}x' & y'\end{pmatrix} \begin{pmatrix}1 & 0\\0 &
b\end{pmatrix} \begin{pmatrix}x'\\y'\end{pmatrix} = x'^2 + b y'^2.
\end{align}
If parameter $b$ is non-negative, the quadratic function has a minimum.
From the shape matrix we see that the axes of the ellipsoid contours
coincide with the axes of the coordinate system. When using the default
parameter $b = 100$, the ellipsoids are elongated in the $x'$ direction
(usually the horizontal direction in diagrams) and narrow in the $y'$
direction (vertical).

Now the underlying coordinate space is distorted using $x' = x - a$
(where the default parameter is $a = 1$) and $y' = y - c x^2 $. The
first equation shifts the center of the quadratic function to the right
(i.e. in the positive $x$ direction), the second equation distorts the
elongated valley by bending it upwards (in positive $y$ direction) on
both sides of the $y$ axis, producing a banana shape.\footnote{I suppose
that if the distortion is applied simultaneously to both coordinates as
in this derivation, there could be transformations where the topology of
the space is not preserved. This would result in duplicated regions of
the original function, or in the omission of parts of the space.
However, it should always be safe to apply transformations sequentially
to the coordinate variables.} We added the parameter $c$ (default $c =
1$) instead of the constant $1$ used in the original function; this
parameter affects the amount of bending.

By inserting these transformations and defining $g\mleft( x,y \mright) =
f\mleft( x',y' \mright)$ we obtain the Rosenbrock function
\begin{align}
\label{eq_rosenbrock}
g\mleft( x,y \mright)
&=
\mleft( x - a \mright)^2 + b \mleft( y - c x^2 \mright)^2.
\end{align}
The minimum of the Rosenbrock function (zero as in the undistorted
function) is found at $\left(x, y\right) = \left(a, c a^2 \right)$; in
this case, both arguments of the squared terms in \eqref{eq_rosenbrock}
are zero.

The Rosenbrock function is also well suited for our purpose since the
original version which has a minimum can be turned into a version which
has a saddle point in the same position. This is achieved by choosing a
negative value for $b$ (e.g. $b = - 100$). We are testing four different
versions of the Rosenbrock function: minimum and narrow valley ($b =
100$), saddle and narrow valley ($b = -100$), minimum and wider valley
($b = 10$), and saddle and wider valley ($b = -10$). In this report, we
focus on the versions with the wider valleys.

Newton's method requires that the Hessian $\mb{H}$ is non-singular. It
is interesting to observe that in the Rosenbrock function there is a
curve where $\mb{H}$ is actually singular, and this curve runs close to
the sloped valley bottom of the banana shape. We can obtain the equation
for this curve from the determinant of the Hessian (computed using
Python's \texttt{sympy} package):
\begin{align}
\label{eq_rosenbrock_detH}
\opnl{det} \mleft\{ \mb{H}\mleft( x,y \mright) \mright\}
&=
8 b^2 c^2 x^2 - 8 b^2 c y + 4 b
\end{align}
from which we obtain an equation for the curve $\opnl{det} \mleft\{
\mb{H}\mleft( x,y \mright) \mright\} = 0$:
\begin{align}
y
&=
c x^2 + {\,\frac{1}{2 b c }\,}
\end{align}
The first term is the curve of the valley bottom. The second term is the
offset of the singularity curve from the valley bottom. For the default
parameters ($b = 100$, $c = 1$), this term is $1/200 = 0.005$, so the
singularity curve runs very close to the bottom of the valley. This is
somewhat disconcerting when the Rosenbrock function is used as a test
function for Newton methods, in particular with line search where the
sequence of intermediate solutions runs close to the valley bottom and
therefore adjacent to the singularity curve. Bending a straight valley
(or just its walls) will probably always introduce such a singularity.
If the elongation in horizontal direction $b$ increases, the singularity
curve moves closer to the valley bottom. The singularity curve only
disappears if there is no vertical distortion (i.e. for $c = 0$) --- in
this case we get $b = 0$ from \eqref{eq_rosenbrock_detH}, so for any $b
\ne 0$, no singularity exists.

The singularity curve is also somewhat unusual. I assume that typically
a singularity manifold separates the regions corresponding to different
stationary points, or regions where Newton's method converges or
diverges, respectively. However, the Rosenbrock function only has a
single stationary point, and Newton's method without line search
approaches the stationary point from both sides of the singularity
curve.

\subsection{Rosenbrock Ditch Function}\label{sec_ditch}
%------------------------------------------------------

%
The Rosenbrock ``ditch'' function is a version of the Rosenbrock
function introduced in this work. Instead of starting from a purely
quadratic function, the function is ditch-shaped in $y$-direction, e.g.
it is quadratic at the bottom, but has an inflection point further up
and then approaches a limit:
\begin{align}
f\mleft( x',y' \mright) & = x'^2 + b {\,\frac{y'^2}{1 + d y'^2 }\,}.
\end{align}
There is a fourth parameter $d$ in addition to the three parameters
present in the Rosenbrock function.

As for the Rosenbrock function, we simultaneously replace $x' = x - a$
and $y' = y - c x^2 $, thus bending the function into a banana shape:
\begin{align}
f\mleft( x',y' \mright) & = \mleft( x - a \mright)^2 + b {\,\frac{\mleft( y - c x^2 \mright)^2}{1 +
d \mleft( y - c x^2 \mright)^2 }\,}.
\end{align}
The minimum (or saddle point) remains the same as in the Rosenbrock
function; however, it is not clear whether there are other stationary
points.

\subsection{Himmelblau's Function}\label{sec_himmelblau}
%-------------------------------------------------------

%
Himmelblau's function --- taken second-hand from \citet[p.~15]{nn_Alt11}
and the Wikipedia
page\footnote{\url{https://en.wikipedia.org/wiki/Himmelblau's_function}}
--- is defined by
\begin{align}
f\mleft( x,y \mright)
&=
\mleft( x^2 + y - 11 \mright)^2 + \mleft( x + y^2 - 7 \mright)^2.
\end{align}
Himmelblau's function has a local maximum and four local minima (with
values of $0.0$), as well as four saddle points.

\subsection{Henon-Heiles Function}\label{sec_henonheiles}
%--------------------------------------------------------

%
I found a description of the H\'enon-Heiles system, a potential function
from celestial mechanics, second-hand in \citet[]{nn_Enns11} and
consulted the Wikipedia
page.\footnote{\url{https://en.wikipedia.org/wiki/H\%C3\%A9non\%E2\%80\%93Heiles_system}}
Here the system is used as an objective function. It is defined by
\begin{align}
f\mleft( x,y \mright)
&=
\frac{1}{2}\, \mleft( x^2 + y^2 \mright) + a \mleft( x^2 y -
\frac{1}{3}\, y^{3} \mright).
\end{align}
The function comprises a minimum at $(0,0)$ with value $0.0$, as well as
three saddle points.

\subsection{Junction Functions}\label{sec_junction}
%--------------------------------------------------

%
The two ``junction'' functions introduced in this work exhibit a
crossing of two ``ditches'' (see section~\ref{sec_ditch}). In one
version, both ditches are bent, in the other version only one. These
functions have fixed parameters, since using parameter symbols in the
Python symbolic toolbox caused very long computation times.

The derivation starts from a function with two straight ditches running
along the axes of the coordinate system, crossing at $(0,0)$,
superimposed with a quadratic function:
\begin{align}
j(x', y') =
\frac{1000 x'^2 y'^2}{(10 + x'^2) (5 + y'^2)}
+ x'^2 + y'^2
\end{align}
In the first step, we replace $y'$ by $y - 0.05 x'^2$. The resulting
function over $(x', y)$ is called ``junction2'' (one bent ditch, one
straight ditch).

In the second step, we replace $x'$ by $x - 0.02 y^2$. This gives a
function over $(x, y)$ called ``junction1'' (two bent ditches).

Both junction functions have a minimum at $(0,0)$ with value $0$, but
may have additional stationary points elsewhere.

\subsection{Goldstein-Price and Beale Function}\label{sec_gldstnprc_beale}
%-------------------------------------------------------------------------

%
Both the Goldstein-Price and the Beale function where taken second-hand
from the Wikipedia collection of test
functions.\footnote{\url{https://en.wikipedia.org/wiki/Test_functions_for_optimization}}

The Goldstein-Price function, which can be interpreted as a crossing of
two valleys as in the junction functions (section~\ref{sec_junction}),
is defined by
\begin{align}
g(x, y) = &
\left[
      1 + (x + y + 1)^2 (19 - 14 x + 3 x^2 - 14 y + 6 x y + 3 y^2)
\right]\cdot\\
&
\left[
      30 + (2 x - 3 y)^2 (18 - 32 x + 12 x^2 + 48 y - 36 x y + 27 y^2)
\right];
\end{align}
it has a minimum at $g(0,-1)=3$.

The Beale function is defined by
\begin{align}
b(x, y) = 
(1.5 - x + x y)^2 +
(2.25 - x + x y^2)^2 +
(2.625 - x + x y^3)^2;
\end{align}
it has a minimum at $b(3,0.5)=0$.

\section{Computational Effort for the Third-Order Term}\label{sec_effort_third_order}
%====================================================================================

%
The divergence criterion \eqref{eq_tau} requires the computation of all
third-order derivatives of the Hessian. In a trivial implementation,
this would require the computation of $n^3$ values (if $n$ is the
problem dimension). However, if the assumptions of Schwarz's theorem are
fulfilled, the order of the 3 partial derivatives is arbitrary. We can
therefore reduce the number of computations. From enumerating the values
up to $n=5$, I guess the required number of computations is expressed by
the binomial coefficient
\begin{align}
\binom{n+2}{n-1}.
\end{align}
For large values of $n$, this would reduce the computational effort by a
factor of up to $6$ (obtained from a numerical experiment). Still, the
effort is prohibitively large for large $n$. It would only be manageable
if many third-order derivatives are zero (sparse Hessian, quadratic
terms).

For the implementation used in the experiments with 2D test functions,
we reduced the effort for the computation of the Hessian by $25\%$, and
for the third-order terms by $50\%$.


\begin{thebibliography}{14}
\expandafter\ifx\csname natexlab\endcsname\relax\def\natexlab#1{#1}\fi

\bibitem[Abadir and Magnus(2005)]{nn_Abadir05}
K.~M. Abadir and J.~R. Magnus.
\newblock {\em Matrix Algebra}.
\newblock Cambridge University Press, 2005.

\bibitem[Aggarwal(2020)]{nn_Aggarwal20}
C.~C. Aggarwal.
\newblock {\em Linear Algebra and Optimization for Machine Learning}.
\newblock Springer, 2020.

\bibitem[Alt(2011)]{nn_Alt11}
W.~Alt.
\newblock {\em Nichtlineare Optimierung. Eine Einf\"uhrung in Theorie,
  Verfahren und Anwendungen}.
\newblock Vieweg + Teubner, 2nd edition, 2011.

\bibitem[Andrei(2022)]{nn_Andrei22}
N.~Andrei.
\newblock {\em Modern Numerical Nonlinear Optimization}, volume 195 of {\em
  Springer Optimization and Its Applications}.
\newblock Springer, 1st edition, 2022.

\bibitem[Chong and \.{Z}ak(2008)]{nn_Chong08}
E.~K.~P. Chong and S.~H. \.{Z}ak.
\newblock {\em An Introduction to Optimization}.
\newblock Wiley, 3rd edition, 2008.

\bibitem[Corriou(2021)]{nn_Corriou21}
J.-P. Corriou.
\newblock {\em Numerical Methods and Optimization. Theory and Practice for
  Engineers}, volume 187 of {\em Springer Optimization and Its Applications}.
\newblock Springer, 2021.

\bibitem[Deisenroth et~al.(2020)Deisenroth, Faisal, and Ong]{nn_Deisenroth20}
M.~P. Deisenroth, A.~A. Faisal, and C.~S. Ong.
\newblock {\em Mathematics for Machine Learning}.
\newblock Cambridge University Press, 2020.

\bibitem[Enns(2011)]{nn_Enns11}
R.~H. Enns.
\newblock {\em It's a Nonlinear World}.
\newblock Springer Undergraduate Texts in Mathematics and Technology. Springer,
  2011.

\bibitem[Geiger and Kanzow(2002)]{nn_Geiger02}
C.~Geiger and C.~Kanzow.
\newblock {\em Theorie und Numerik restringierter Optimierungsaufgaben}.
\newblock Springer, 2002.

\bibitem[Jarre and Stoer(2019)]{nn_Jarre19}
F.~Jarre and J.~Stoer.
\newblock {\em Optimierung. Eine Einf\"uhrung in mathematische Theorie und
  Methoden}.
\newblock Springer Spektrum, 2nd edition, 2019.

\bibitem[M\"oller(2022)]{own_Moeller22}
R.~M\"oller.
\newblock Derivation of learning rules for coupled principal component analysis
  in a {L}agrange-{N}ewton framework.
\newblock {\em arXiv:2204.13460}, 2022.

\bibitem[M\"oller(2024)]{own_Moeller24}
R.~M\"oller.
\newblock A {L}agrange-{N}ewton approach to smoothing-and-mapping.
\newblock {\em arXiv:2401.13302}, 2024.

\bibitem[M\"oller and K\"onies(2004)]{own_Moeller04a}
R.~M\"oller and A.~K\"onies.
\newblock Coupled principal component analysis.
\newblock {\em IEEE Transactions on Neural Networks}, 15\penalty0 (1):\penalty0
  214--222, 2004.

\bibitem[Nocedal and Wright(2006)]{nn_Nocedal06}
J.~Nocedal and S.~J. Wright.
\newblock {\em Numerical Optimization}.
\newblock Springer Series in Operation Research and Financial Engineering.
  Springer, 2nd edition, 2006.

\end{thebibliography}
\end{document}